\documentclass[a4paper,11pt]{article}

\usepackage[scriptsize,bf,]{caption}
\usepackage{enumerate}
\usepackage{fullpage}
\usepackage{setspace}
\usepackage{hyperref}
\usepackage[english]{babel}
\usepackage{hhline}
\usepackage{graphicx}
\usepackage{epstopdf}
\usepackage{amsfonts}
\usepackage{rotating}
\usepackage{amsmath}
\usepackage{amssymb}
\usepackage{mathrsfs}
\usepackage{bm}
\usepackage{color}
\usepackage{cases}
\usepackage{babel}
\usepackage{xspace}
\usepackage{psfrag}
\usepackage{multirow}
\usepackage{caption}
\usepackage[normalem]{ulem}
\usepackage{subfig}
\usepackage{mathtools}
\usepackage{authblk}
\usepackage{array}
\newcolumntype{C}[1]{%
 >{\vbox to 9ex\bgroup\vfill\centering}%
 p{#1}%
 <{\egroup}}

\DeclareMathOperator*{\st}{s.t.}

\newtheorem{lemma}{Lemma}[section]

\newtheorem{prop}{Proposition}[section]
\newtheorem{thm}{Theorem}[section]
\newtheorem{rem}{Remark}[section]
\newtheorem{cor}{Corollary}[section]

\newcommand{\dd}{{\rm d}}

\newenvironment{proof}[1][Proof:]{\begin{trivlist}
\item[\hskip \labelsep {\bfseries #1}]}{\end{trivlist}}

\newcommand{\qed}{\nobreak \ifvmode \relax \else
      \ifdim\lastskip<1.5em \hskip-\lastskip
      \hskip1.5em plus0em minus0.5em \fi \nobreak
      \vrule height0.75em width0.5em depth0em\fi}
      
\hypersetup{
    colorlinks,
    citecolor=black,
    filecolor=black,
    linkcolor=black,
    urlcolor=black
}

\title{Chebyshev Inequalities for Products of Random Variables}
\author[1]{Napat Rujeerapaiboon}
\author[1]{Daniel Kuhn}
\author[2]{Wolfram Wiesemann}
\affil[1]{\small \textit{Risk Analytics and Optimization Chair, 
    	  		 $\acute{\text{E}}$cole Polytechnique F$\acute{\text{e}}$d$\acute{\text{e}}$rale de Lausanne, Switzerland}}
\affil[2]{\small \textit{Imperial College Business School, Imperial College London, United Kingdom}}

\begin{document}

\maketitle

\begin{abstract}
We derive sharp probability bounds on the tails of a product of symmetric non-negative random variables using only information about their first two moments. If the covariance matrix of the random variables is known exactly, these bounds can be computed numerically using semidefinite programming.  If only an upper bound on the covariance matrix is available, the probability bounds on the right tails can be evaluated analytically. The bounds under precise and imprecise covariance information coincide for all left tails as well as for all right tails corresponding to quantiles that are either sufficiently small or sufficiently large. We also prove that all left probability bounds reduce to the trivial bound~$1$ if the number of random variables in the product exceeds an explicit threshold. Thus, in the worst case, the weak-sense geometric random walk defined through the running product of the random variables is absorbed at~$0$ with certainty as soon as time exceeds the given threshold. The techniques devised for constructing Chebyshev bounds for products can also be used to derive Chebyshev bounds for sums, maxima and minima of non-negative random variables.
\vspace{0.2cm}
\end{abstract}

\onehalfspacing

\section{Introduction}
\label{section:introduction}
The classical one-sided Chebyshev inequality \cite{Bienayme53,Chebyshev67} for a random variable $\tilde\xi$ with mean $\mu$ and variance $\sigma^2$ can be represented as
\begin{equation}
	\label{eq:chebyshev}
	\mathbb P(\tilde \xi\geq \gamma)\leq \left\{ \begin{array}{ll}
	\frac{\sigma^2}{\sigma^2+(\gamma-\mu)^2} & \text{if }\gamma\geq \mu,\\
	1 & \text{if } \gamma<\mu.
	\end{array}\right.
\end{equation}
This inequality is sharp. Indeed, for $\gamma\neq \mu$ it is binding under the two-point distribution
\begin{equation}
	\label{eq:chebyshev-wcdist}
	\mathbb P^\star = \left\{ \begin{array}{ll}
	\frac{\sigma^2}{\sigma^2+(\gamma-\mu)^2} \delta_\gamma + \frac{(\gamma-\mu)^2}{\sigma^2+(\gamma-\mu)^2} \delta_{\mu-\sigma^2/(\gamma-\mu)} & \text{if }\gamma> \mu,\\[1ex]
	\frac{\sigma^2}{\sigma^2+ (\mu-\gamma)^2} \delta_\gamma + \frac{(\mu-\gamma)^2}{\sigma^2+ (\mu-\gamma)^2} \delta_{\mu+\sigma^2/(\mu-\gamma)} & \text{if } \gamma<\mu.
	\end{array}\right.
\end{equation}
In the degenerate case $\gamma=\mu$, the inequality~\eqref{eq:chebyshev} is still sharp because the distributions
\[
	\mathbb P_\kappa = \frac{1}{1+ \kappa^2} \delta_{\gamma-\sigma \kappa} + \frac{\kappa^2}{1+ \kappa^2} \delta_{\gamma+\sigma/\kappa}
\]
have mean $\mu$ and variance $\sigma^2$ for every $\kappa>0$, while $\lim_{\kappa\uparrow \infty}\mathbb P_\kappa(\tilde \xi\geq \gamma)=1$. Note, however, that no single distribution with mean $\mu=\gamma$ and variance $\sigma^2>0$ can satisfy $\mathbb P(\tilde \xi\geq \gamma)=1$.

If we have the extra information that the random variable $\tilde\xi$ is non-negative (and without much loss of generality that $\mu>0$), then one can strengthen the Chebyshev inequality~\eqref{eq:chebyshev} to
\begin{equation}
	\label{eq:chebyshev-nonnegative}
	\mathbb P(\tilde \xi\geq \gamma)\leq \left\{ \begin{array}{ll}
	\frac{\sigma^2}{\sigma^2+(\gamma-\mu)^2} & \text{if } \gamma\geq  \mu+ \sigma^2/\mu,\\
	\frac{\mu}{\gamma} & \text{if }\mu\leq \gamma<\mu+ \sigma^2/\mu,\\
	1 & \text{if } \gamma<\mu,
	\end{array}\right.
\end{equation}
see, {\em e.g.},~\cite{Godwin55,Shohat43}. The extremal distributions~\eqref{eq:chebyshev-wcdist} are supported on the non-negative real line if either $\gamma\geq  \mu+ \sigma^2/\mu>\mu$ or if $\gamma<\mu$. Thus, they certify the sharpness of~\eqref{eq:chebyshev-nonnegative} in the respective parameter domains. For $\mu\leq \gamma<\mu+ \sigma^2/\mu$ the Chebyshev inequality~\eqref{eq:chebyshev-nonnegative} for non-negative random variables reduces in fact to the classical Markov inequality $\mathbb P(\tilde \xi\geq \gamma)\leq\mu/\gamma$. In this Markov regime, the Chebyshev inequality~\eqref{eq:chebyshev-nonnegative} remains sharp because the distributions
\[
	\mathbb P_\kappa = \left[1+\frac{\sigma^2}{\kappa\gamma} - \frac{\mu(\kappa-\mu)}{\gamma(\kappa-\gamma)} - \frac{\mu(\gamma-\mu)}{\kappa(\kappa- \gamma)}\right]  \delta_0+\frac{\mu(\kappa-\mu)-\sigma^2}{\gamma(\kappa-\gamma)}\delta_\gamma +\frac{\sigma^2-\mu(\gamma-\mu)}{\kappa(\kappa-\gamma)}\delta_\kappa
\]
have mean $\mu$ and variance $\sigma^2$ for every $\kappa>\mu+\sigma^2/\mu$, while $\lim_{\kappa\uparrow \infty}\mathbb P_\kappa(\tilde \xi\geq \gamma)=\mu/\gamma$. From the textbook proof of Markov's inequality it follows that $\mathbb P^\star=[1-\mu/\gamma]\delta_0+[\mu/\gamma]\delta_\gamma$ is the only distribution on the non-negative reals that has mean $\mu$ and satisfies $\mathbb P^\star(\tilde \xi\geq \gamma)=\mu/\gamma$. However, the additional requirement that the variance of $\tilde \xi$ under $\mathbb P^\star$ must equal $\sigma^2$ implies $\gamma=\mu+\sigma^2/\mu$. Thus, for $\mu\leq \gamma<\mu+ \sigma^2/\mu$ there cannot exist any single distribution with $\mathbb P(\tilde \xi\geq \gamma)=\mu/\gamma$.

In the rest of the paper we consider a sequence of $T$ random variables $\tilde{\xi}_1, \tilde{\xi}_2, \ldots, \tilde{\xi}_T$ and assume that the first two moments of these random variables are known and permutation symmetric. Specifically, assume that all random variables share the same mean $\mu$ and variance $\sigma^2$, respectively, while all pairs of mutually distinct random variables share the same correlation coefficient $\rho$. Thus, the mean vector and the covariance matrix of $\bm{\tilde \xi}=(\tilde{\xi}_1, \ldots, \tilde{\xi}_T)^\intercal$ are given~by
\begin{equation}
\label{eq:mean-cov}
\bm \mu=\begin{bmatrix}\mu \\ \mu \\ \vdots \\ \mu \end{bmatrix}\in\mathbb R^T \quad \text{and} \quad \bm \Sigma=\begin{bmatrix} \sigma^2 & \rho \sigma^2 & \cdots & \rho\sigma^2 \\ \rho \sigma^2 & \sigma^2 & \cdots & \rho \sigma^2 \\ \vdots & \vdots & \ddots & \vdots \\
\rho \sigma^2 & \rho\sigma^2 & \cdots & \sigma^2 \end{bmatrix} \in\mathbb S^T,
\end{equation}
respectively. Throughout the paper we assume that $\sigma>0$ and $-\frac{1}{T-1}<\rho<1$. These conditions are necessary and sufficient for the covariance matrix~$\bm \Sigma$ to be strictly positive definite. Note that $\bm{\tilde\xi}$ constitutes a weak-sense stationary stochastic process in the sense of~\cite{Lindgren12}.

An elementary calculation reveals that the sum $\sum_{t=1}^T \tilde\xi_t$ has mean value $T\mu$ and variance $T\sigma^2(1 + (T-1)\rho)$. The classical Chebyshev inequality~\eqref{eq:chebyshev} applied to $\sum_{t=1}^T \tilde\xi_t$ thus implies
\begin{equation}
	\label{eq:chebyshev-sum}
	\textstyle \mathbb P(\sum_{t=1}^T \tilde\xi_t \geq \gamma)\leq \left\{ \begin{array}{ll}
	\frac{T\sigma^2(1 + (T-1)\rho)}{T\sigma^2(1 + (T-1)\rho)+(\gamma-T\mu)^2} & \text{if }\gamma\geq T\mu,\\
	1 & \text{if } \gamma<T\mu.
	\end{array}\right.
\end{equation}
This inequality is still sharp due to a projection property of distribution families with compatible first and second moments. Indeed, for any distribution $\mathbb P_\zeta$ of a random variable $\tilde \zeta$ with mean value $T\mu$ and variance $T\sigma^2(1 + (T-1)\rho)$ there exists a distribution $\mathbb P$ of the random vector $\bm{\tilde \xi}$ with mean vector $\bm \mu$ and covariance matrix $\bm \Sigma$ such that $\mathbb P_\zeta$ coincides with the marginal distribution of $\sum_{t=1}^T \tilde\xi_t$ under $\mathbb P$, that is, $\mathbb P_\zeta(\tilde \zeta\in B)=\mathbb P(\sum_{t=1}^T \tilde\xi_t\in B)$ for every Borel set $B\subseteq\mathbb R$~\cite{Yu09}. The extremal distributions~\eqref{eq:chebyshev-wcdist} certifying the sharpness of~\eqref{eq:chebyshev} can therefore be used to construct multivariate extremal distributions of $\bm{\tilde \xi}$ certifying the sharpness of~\eqref{eq:chebyshev-sum}. This result may be unexpected. Indeed, if $\tilde\xi_1,\ldots,\tilde\xi_T$ are independent and identically distributed, then,  by the central limit theorem, their sum is approximately normally distributed with mean $T\mu$ and variance $T\sigma^2$. In contrast, if $\tilde\xi_1,\ldots,\tilde\xi_T$ are only known to be uncorrelated with a common mean and variance (but not necessarily independent and identically distributed), then, by the projection theorem, their sum may follow {\em any} distribution with mean $T\mu$ and variance $T\sigma^2$.

Assume now that $\tilde\xi_t$ is non-negative for every $t=1,\ldots, T$ (and without much loss of generality that $\mu>0$). As we will prove in Proposition~\ref{prop:non-emptiness} below, a distribution $\mathbb P$ supported on $\mathbb R_+^T$ with mean vector $\bm\mu$ and covariance matrix $\bm \Sigma$ as given in~\eqref{eq:mean-cov} exists iff $\mu^2 + \rho\sigma^2 \geq 0$. We will assume that this condition holds throughout the rest of the paper. In this setting, the generalized Chebyshev inequality~\eqref{eq:chebyshev-nonnegative} applied to the non-negative random variable $\sum_{t=1}^T \tilde\xi_t$ implies
\begin{equation}
	\label{eq:chebyshev-sum-nonnegative}
	\textstyle \mathbb P(\sum_{t=1}^T \tilde\xi_t \geq \gamma)\leq \left\{ \begin{array}{ll}
	\frac{T\sigma^2(1 + (T-1)\rho)}{T\sigma^2(1 + (T-1)\rho)+(\gamma-T\mu)^2} & \text{if } \gamma\geq  T\mu+\sigma^2(1 + (T-1)\rho)/\mu,\\
	\frac{T\mu}{\gamma} & \text{if }T\mu\leq \gamma<T\mu+\sigma^2(1 + (T-1)\rho)/\mu,\\
	1 & \text{if } \gamma<T\mu.
	\end{array}\right.
\end{equation}
Even though the multivariate extension~\eqref{eq:chebyshev-sum-nonnegative} of the univariate Chebyshev inequality~\eqref{eq:chebyshev-nonnegative} can still be shown to be sharp, we are not aware of an elementary proof; see~Theorem~\ref{thm:sum2} below.

In this paper we aim to derive Chebyshev inequalities for {\em products} of non-negative random variables. Specifically, we will derive sharp upper bounds on the left and right tail probabilities $\mathbb P(\prod_{t=1}^T \tilde\xi_t\leq\gamma)$ and $\mathbb P(\prod_{t=1}^T \tilde\xi_t\geq\gamma)$, respectively. 
Products of random variables frequently arise in physics, statistics, finance, number theory and many other branches of science~\cite{Galambos04}. Indeed, they are at the heart of stochastic models of many complex phenomena. When rocks are crushed, for example, the size of a fragment is multiplied by a random factor (that is smaller than 1) in every single breakup event \cite{Frisch97}. Similar multiplicative phenomena explain the distribution of body weights, stock prices, the sizes of biological populations, income, rainfall etc.~\cite{Aitchison57}. 

Note that the stochastic process $\bm{\tilde \pi}=\{\tilde\pi_T\}_{T\in\mathbb N}$ defined through $\tilde \pi_T=\prod_{t=1}^T \tilde\xi_t$ can be interpreted as a geometric random walk driven by the weak-sense stationary process $\bm{\tilde\xi}=\{\tilde\xi_t\}_{t\in\mathbb N}$. Chebyshev inequalities for the products of the $\tilde\xi_t$ thus provide tight bounds on the quantiles of a geometric random walk when there is limited distributional information. Consequently, they are potentially relevant for the many applications in economics and operations research, where geometric Brownian motions are traditionally used to model the prices of assets \cite{Karatzas91}. An improved understanding of weak-sense geometric random walks may also stimulate new research directions in distributionally robust optimziation~\cite{Delage09, Goh10, Wiesemann14} and optimal uncertainty quantification~\cite{Hanasusanto15,Owhadi13}.

\begin{rem}[Chebyshev in Log-Space]
It seems natural to reduce Chebyshev inequalities for products of non-negative random variables to Chebyshev inequalities for their logarithms. Assume thus that the first two moments of the logarithmic random variables $\tilde\eta_t=\log(\tilde\xi_t)$, $1,\ldots,T$, are known and permutation symmetric. Specifically, denote by $\mu_\eta$, $\sigma^2_\eta$ and $\rho_\eta$ the mean, variance and correlation coefficient in log-space. Then, the Chebyshev inequality~\eqref{eq:chebyshev-sum} for sums implies
\begin{equation}
	\label{eq:chebyshev-log-space}
	\textstyle \mathbb P(\prod_{t=1}^T \tilde\xi_t\geq\gamma)= \mathbb P(\sum_{t=1}^T \tilde\eta_t\geq\log\gamma) \leq \left\{ \begin{array}{ll}
	\frac{T\sigma_\eta^2(1 + (T-1)\rho_\eta)}{T\sigma_\eta^2(1 + (T-1)\rho_\eta)+(\log\gamma-T\mu_\eta)^2} & \text{if }\log\gamma\geq T\mu_\eta,\\
	1 & \text{if } \log\gamma<T\mu_\eta.
	\end{array}\right.
\end{equation}
Note that~\eqref{eq:chebyshev-log-space} is sharp because~\eqref{eq:chebyshev-sum} is sharp. However, there is no one-to-one correspondence between the moments of the original and the logarithmic random variables. Even worse, it is possible that $\mu$ is finite while $\mu_\eta=-\infty$ (e.g., if $\xi_t=0$ with positive probability), or that $\mu_\eta$ is finite while $\mu=+\infty$ (e.g., if $\tilde\xi_t$ follows a Pareto distribution with unit shape parameter). In this work we focus on the case where the $\tilde\xi_t$ have known finite first and second moments, and we explicitly allow the event $\tilde\xi_t=0$ to have positive probability. This assumption can be crucial for truthfully capturing the bankruptcy risks in financial applications, for instance.
\end{rem}

The starting point of this paper is the intriguing observation that modern optimization theory provides powerful tools for constructing and analyzing probability inequalities~\cite{Bertsimas05}. Assume for instance that we aim to find a sharp probability inequality for a target event characterized through finitely many polynomial inequalities on a random vector $\bm{\tilde \xi}$. Assume further that the desired inequality should hold for all distributions of $\bm{\tilde\xi}$ satisfying finitely many polynomial support and moment constraints. In the special case of the Chebyshev inequality~\eqref{eq:chebyshev}, the target event corresponds to the set $\{\xi\in\mathbb R:\xi\geq \gamma\}$, while the relevant distribution family corresponds to the class of all distributions on $\mathbb R$ with mean~$\mu$ and variance~$\sigma^2$. Constructing the desired probability inequality is thus tantamount to maximizing the probability of the target event over the given distribution family. This leads to a generalized moment problem over probability measures. Under a mild regularity condition, this moment problem admits a strong dual linear program subject to polynomially parameterized semi-infinite constraints~\cite{Isii60, Isii62, Karlin66}. A key insight of~\cite{Bertsimas05} is that this dual problem can be approximated systematically by tractable semidefinite programs. The resulting approximations are safe ({\em i.e.}, they are guaranteed to provide {\em upper} bounds on the probability of the semialgebraic event). Moreover, these approximations are always tight in the univariate case but generically loose in the multivariate setting.

Stronger statements are available for probability inequalities that rely exclusively on first- and second-order moments. Specifically, if the support of the random vector $\bm{\tilde \xi}$ is unrestricted, the best upper bound on the probability of a {\em convex} target event is given by $1/(1+d^2)$, where $d$ represents the distance of the target event from the mean vector of $\bm{\tilde \xi}$ under the Mahalanobis norm induced by the covariance matrix of $\bm{\tilde \xi}$ \cite{Marshall60}. More generally, if the target event constitutes a {\em union} of finitely many convex sets, over each of which convex quadratic optimization problems can be solved in polynomial time, then the best Chebyshev bound can be computed by an efficient algorithm reminiscent of the ellipsoid method of convex optimization~\cite{Bertsimas05}. Recently it has been observed that if the target event is defined by quadratic inequalities, the best Chebyshev bound coincides exactly with the optimal value of a single tractable semidefinite program~\cite{Vandenberghe07}. In spite of these encouraging results, the computation of Chebyshev bounds becomes hard in the presence of support constraints. Specifically, if $\bm{\tilde \xi}$ is supported on the non-negative orthant,  it is already NP-hard to find sharp Chebyshev bounds for convex polyhedral target events~\cite{Bertsimas05}. 

For a random vector $\bm{\tilde \xi}$ with zero mean and unrestricted support, the above methods have been used to derive a sharp Chebyshev bound on 
$\mathbb P(\prod_{t=1}^T \tilde\xi_t \geq 1,\,  \tilde\xi_t> 0\,~\forall t)$, which is expressed in terms of the solution of a tractable convex program \cite{Marshall60}. As the $\tilde\xi_t$ are allowed to adopt negative values, however, we believe that the practical relevance of this bound is limited. In this paper we aim to derive sharp Chebyshev bounds on $\mathbb P(\prod_{t=1}^T \tilde\xi_t\geq\gamma)$ and $\mathbb P(\prod_{t=1}^T \tilde\xi_t\leq\gamma)$ under the explicit assumption that $\bm{\tilde \xi}$ is supported on the non-negative orthant. Note that the second target event $\{\bm\xi\in\mathbb R^T_+:\prod_{t=1}^T\xi_t\leq \gamma\}$ is neither convex nor representable as a finite union of convex sets, nor representable through finitely many quadratic constraints in $\bm{\xi}$. Thus, none of the existing techniques could be used to bound its probability even if there were no support constraints. As support constraints generically lead to intractability~\cite{Bertsimas05}, we focus here on the special case where the first- and second-order moments are permutation-symmetric.

The main results of this paper can be summarized as follows.
\begin{itemize}
\item[(i)] If the distribution $\mathbb P$ of the non-negative random variables has mean $\bm\mu$ and covariance matrix $\bm\Sigma$ as given in~\eqref{eq:mean-cov}, then the sharp upper Chebyshev bounds on $\mathbb P(\prod_{t=1}^T \tilde\xi_t\geq\gamma)$ and $\mathbb P(\prod_{t=1}^T \tilde\xi_t\leq\gamma)$ can both be expressed as the optimal values of explicit semidefinite programs, which are amenable to efficient numerical solution via interior point algorithms.
\item[(ii)] If the distribution $\mathbb P$ of the non-negative random variables has mean $\bm\mu$ and a covariance matrix bounded above by $\bm\Sigma$ in a positive semidefinite sense, then we obtain an explicit analytical formula for the sharp upper Chebyshev bound on $\mathbb P(\prod_{t=1}^T \tilde\xi_t\geq\gamma)$.
\item[(iii)] The Chebyshev bound in~(ii) coincides with the corresponding bound in~(i) for all values of $\gamma$ that are either sufficiently small or sufficiently large. For intermediate values of $\gamma$ the numerical bound in~(i) may be strictly smaller than the analytical bound in~(ii).
\item[(iv)] If the distribution $\mathbb P$ of the non-negative random variables has mean $\bm\mu$ and a covariance matrix bounded above by $\bm\Sigma$ in a positive semidefinite sense, then the sharp upper Chebyshev bound on $\mathbb P(\prod_{t=1}^T \tilde\xi_t\leq\gamma)$ coincides with the corresponding numerical bound in~(i). Thus, there is a distribution that makes this bound sharp and has covariance matrix $\bm\Sigma$.
\item[(v)] The Chebyshev bound in~(iv) reduces to the trivial bound $1$ for every $\gamma>0$ if $T$ exceeds an explicit threshold $T_0$. Thus, in the worst case, the weak-sense geometric random walk $\bm{\tilde \pi}=\{\tilde\pi_T\}_{T\in\mathbb N}$ defined through $\tilde \pi_T=\prod_{t=1}^T \tilde\xi_t$ is absorbed at $0$ {\em with certainty} if $T\geq T_0$.
\item[(vi)] The techniques devised for constructing Chebyshev bounds for products of random variables can also be used to derive Chebyshev bounds on sums, maxima and minima (and possibly other permutation-symmetric functionals) of non-negative random variables.
\end{itemize}

The rest of the paper is structured as follows. In Section~\ref{sec:optimization-perspective} we formalize the connection between probability inequalities and convex optimization. Left- and right-sided Chebyshev inequalities for products of random variables are then derived in Sections~\ref{section:right-prob} and~\ref{section:left-prob}, respectively, while generalized Chebyshev inequalities that account for imprecise knowledge of the covariances are discussed in Section~\ref{section:cov_bounds}. Chebyshev inequalities for other permutation-symmetric functionals of the random variables are presented in Section~\ref{sec:extensions}, and examples are given in Section~\ref{sec:examples}.

\paragraph{Notation} The symbol $\mathbb I$ stands for the identity matrix, $\bm 1$ for the vector of all ones, and $\mathbf e_i$ for the $i$-th standard basis vector. Their dimensions will always be clear from the context. The space of symmetric $T\times T$ matrices is denoted by $\mathbb S^T$, and its subset of all positive (negative) semidefinite matrices is denoted by $\mathbb S_+^T$. For $\bm A,\bm B\in\mathbb S^T$, the statements $\bm A\succeq\bm B$ and $\bm B\preceq\bm A$ both mean that $\bm A-\bm B\in \mathbb S^T_+$. The indicator function $1_\mathcal E$ of a logical statement $\mathcal E$ is defined through $1_\mathcal E=1$ if $\mathcal E$ holds true; $=0$ otherwise. Random variables are denoted by tilde signs, while their realizations are denoted by the same symbols without tildes. The Dirac distribution concentrating unit mass at $\bm\xi$ is denoted by $\delta_{\bm\xi}$. For any closed set $\mathcal S\subseteq \mathbb R^T$, we let $\mathcal{M}_+(\mathcal{S})$ be the cone of all non-negative Borel measures supported on $\mathcal S$.

\section{Optimization Perspective on Chebyshev Inequalities}
\label{sec:optimization-perspective}
To analyze probability bounds using tools from optimization, we first introduce an {\em ambiguity set} $\mathcal P$, that is, a family of distributions for which the desired probability bound should hold. In this paper we mainly focus on the ambiguity set of all distributions supported on $\mathbb R_+^T$ that share the permutation-symmetric mean and covariance matrix defined in \eqref{eq:mean-cov}, that is, we set 
\begin{equation}
\begin{aligned}\label{eq:our_ambiguity_set}
	\mathcal{P} = 
	\left \lbrace 
		\mathbb{P} \in \mathcal M_+(\mathbb R^T_+)~: 
		\begin{array}{l}
			\mathbb{P} \left( \tilde{\bm{\xi}} \geq \bm 0 \right) = 1 ,\; \mathbb{E}_\mathbb{P} \left( \tilde{\bm{\xi}} \right) = \bm \mu ,\;
			\mathbb{E}_\mathbb{P} \left( \tilde{\bm{\xi}}\tilde{\bm{\xi}}^\intercal \right) = \bm \Sigma+\bm \mu\bm\mu^\intercal  
		\end{array}
	\right \rbrace.
\end{aligned}
\end{equation}
We highlight that $\mathcal P$ is characterized by only four parameters: $T, \mu, \sigma, \rho$. Without much loss of generality, we assume henceforth that $\mu > 0$, $\sigma>0$ and $-\frac{1}{T-1}<\rho<1$. The last two conditions are equivalent to~$\bm \Sigma\succ\bm 0$. To rule out trivial special cases, we further restrict attention to $T \geq 2$. However, all of these conditions do not yet guarantee that $\mathcal P$ is non-empty. Proposition~\ref{prop:non-emptiness} below provides a necessary and sufficient condition for the non-emptiness of $\mathcal P$.

\begin{prop}[Non-emptiness of $\mathcal P$]
\label{prop:non-emptiness}
The ambiguity set $\mathcal{P}$ is non-empty iff $\mu^2 + \rho\sigma^2 \geq 0$.
\end{prop}

\begin{proof}
If $\mathcal{P}$ is non-empty, then any $\mathbb P \in \mathcal P$ satisfies
\begin{equation*}
	\bm 0 \leq \mathbb{E}_\mathbb{P} \left( \tilde{\bm\xi} \tilde{\bm\xi}^\intercal \right) = \bm{\Sigma} + \bm{\mu\mu}^\intercal
	\Longleftrightarrow
	\left\lbrace \begin{array}{l}
		\mu^2 + \sigma^2 \geq 0 \\
		\mu^2 + \rho\sigma^2 \geq 0
	\end{array} \right.
	\Longleftrightarrow
	\mu^2 + \rho\sigma^2 \geq 0,
\end{equation*}
where the equivalences follow from the definition of $\bm{\Sigma}$ and the assumption that $\rho < 1$.

Assume now that $\mu^2 + \rho\sigma^2 \geq 0$. We show that $\mathcal{P}$ contains a discrete distribution $\mathbb{P}$ satisfying
\begin{equation}
\label{eq:construct-xyz}
\begin{aligned}
	\mathbb{P} \left( \tilde{\bm\xi} = y\bm{1} + (x-y)\mathbf e_i \right) = \frac{p}{T}, \quad i = 1, \ldots, T,  \quad \text{and} \quad
	\mathbb{P} \left( \tilde{\bm\xi} = z\bm{1} \right) = 1- p
\end{aligned}
\end{equation}
for $x \geq y \geq 0$, $z \geq 0$ and $p \in [0,1]$. For this distribution to be contained in $\mathcal{P}$, it must also satisfy the following moment conditions:
\begin{enumerate}[(i)]
	\item $\displaystyle \mathbb{E}_\mathbb{P} [\bm{\tilde{\xi}}] = \bm{\mu} \mspace {76mu} \quad \Longleftrightarrow \quad \frac{p}{T}(x + (T-1)y) + (1-p)z = \mu$;
	\item $\displaystyle \mathbb{E}_\mathbb{P} [\bm{\tilde{\xi}} \bm{\tilde{\xi}}{}^\intercal] = \bm{\Sigma} + \bm{\mu} \bm{\mu}^\intercal \quad \Longleftrightarrow \quad \frac{p}{T}(x^2 + (T-1)y^2) + (1-p)z^2 = \mu^2 + \sigma^2$, \\[3mm]
		     \hphantom{a} $\mspace{216mu} \displaystyle \frac{p}{T}(2xy + (T-2)y^2) + (1-p)z^2 = \mu^2 + \rho\sigma^2$.
\end{enumerate}
To construct $\mathbb{P}$, it is notationally convenient to perform the change of variables $m_1 \leftarrow \frac{1}{T}(x + (T-1)y)$ and $m_2 \leftarrow \frac{1}{T}(x^2 + (T-1)y^2)$. For a given $(m_1, m_2)$, we can then recover $(x,y)$ via
\begin{equation*}
	x = m_1 + \sqrt{(T-1)(m_2-m_1^2)} \quad \text{and} \quad 
	y = m_1 - \sqrt{(m_2-m_1^2)/(T-1)}.
\end{equation*}
Note that the correspondence between $(x, y)$ and $(m_1, m_2)$ is one-to-one and onto over $\{ (x, y) \in \mathbb{R}^2_+ \, : \, x \geq y \}$ and $\{ (m_1, m_2) \in \mathbb{R}^2_+ \, : \, m_1^2 \leq m_2 \leq Tm_1^2 \}$. Now, for $\mathbb P$ to be in $\mathcal P$, we require that 
\begin{enumerate}[(i')]
	\item $\displaystyle \mathbb{E}_\mathbb{P} [\bm{\tilde{\xi}}] = \bm{\mu} \mspace {76mu} \quad \Longleftrightarrow \quad pm_1 + (1-p)z = \mu$;
	\item $\displaystyle \mathbb{E}_\mathbb{P} [\bm{\tilde{\xi}} \bm{\tilde{\xi}}{}^\intercal] = \bm{\Sigma} + \bm{\mu} \bm{\mu}^\intercal \quad \Longleftrightarrow \quad pm_2 + (1-p)z^2 = \mu^2 + \sigma^2$, \\[3mm]
		     \hphantom{a} $\mspace{216mu} \displaystyle \frac{p}{T-1}(Tm_1^2 - m_2) + (1-p)z^2 = \mu^2 + \rho\sigma^2$.
\end{enumerate}
In the remainder of the proof, we thus need to show that there is $m_1, m_2, z \geq 0$, $m_1^2 \leq m_2 \leq Tm_1^2$, and $p \in [0,1]$ satisfying (i') and (ii'). To this end, consider the choice
\begin{equation}
\label{eq:construct-p}
	p = \left\{ \begin{array}{ll}
		\min \left\{ \frac{T\mu^2}{T\mu^2 + (1+(T-1)\rho)\sigma^2}, \frac{\rho T}{1 + (T-1)\rho} \right\} & \text{if } \rho > 0, \\
		\vspace{2mm}
		\frac{T\mu^2}{T\mu^2 + \sigma^2} & \text{if } \rho = 0, \\
		\vspace{2mm}
		\frac{-\rho T}{1 - \rho} & \text{if } \rho < 0,
	\end{array} \right.
\end{equation}
which satisfies $p \in [0,1]$ by construction, as well as
\begin{equation*}
	m_1 = \mu + \sigma\sqrt{\frac{(1-p)(1+(T-1)\rho)}{pT}}, \
	m_2 = m_1^2 + \frac{(1-\rho)(T-1)\sigma^2}{pT}, \
	z = \mu - \sigma\sqrt{\frac{p(1+(T-1)\rho)}{(1-p)T}}.
\end{equation*}
Note that the terms inside the square roots are non-negative since $\rho > - 1 / (T - 1)$.

\paragraph{Step 1:} We show that $m_1, m_2, z \geq 0$. The non-negativity of $m_1$ and $m_2$ holds by construction. To check that $z \geq 0$, we distinguish the cases $\rho > 0$, $\rho = 0$ and $\rho < 0$. For $\rho > 0$, we obtain $z = 0$ for $p = \frac{T\mu^2}{T\mu^2 + (1+(T-1)\rho)\sigma^2}$. Since the square root term in the expression for $z$ is increasing in $p$, we thus conclude that $z \geq 0$. The case where $\rho = 0$ is analogous since $\frac{T\mu^2}{T\mu^2 + \sigma^2} = \frac{T\mu^2}{T\mu^2 + (1+(T-1)\rho)\sigma^2}$ for $\rho = 0$. For $\rho < 0$, on the other hand, we obtain $z = \mu - \sigma \sqrt{-\rho}$ for our choice of $p$. The resulting $z$ is thus non-negative due to the assumption that $\mu^2 + \rho\sigma^2 \geq 0$.

\paragraph{Step 2:} To check that $m_1^2 \leq m_2 \leq Tm_1^2$, we first use the definition of $m_2$ and the assumption that $\rho < 1$ to verify that $m_1^2 \leq m_2$. The other inequality holds if and only if
\begin{equation}
\label{eq:m1m2}
\begin{aligned}
	m_2 \leq Tm_1^2 
	\quad &\Longleftrightarrow \quad
	\sqrt{\frac{1-\rho}{pT}}\sigma \leq m_1 \\
	\quad &\Longleftrightarrow \quad
	\mu\sqrt{pT} + \left( \sqrt{(1+(T-1)\rho)(1-p)} - \sqrt{1-\rho} \right)\sigma \geq 0,
\end{aligned}
\end{equation}
where the first and second equivalence follow from the definitions of $m_2$ and $m_1$, respectively. We now show that the last inequality holds by distinguishing the cases $\rho > 0$, $\rho = 0$ and $\rho < 0$.

For $\rho > 0$, we observe that the expression $\sqrt{(1+(T-1)\rho)(1-p)} - \sqrt{1-\rho}$ in~\eqref{eq:m1m2} evaluates to 0 for $p = \frac{T\rho}{1 + (T-1)\rho}$ and that it is decreasing in $p$. Since $\mu\sqrt{pT} \geq 0$ by construction, we thus conclude that the last inequality in~\eqref{eq:m1m2} holds, and hence $m_2 \leq Tm_1^2$ when $\rho \geq 0$.
In combination with~\eqref{eq:construct-p} and~\eqref{eq:m1m2}, the above inequality ensures that $m_2 \leq Tm_1^2$.

For $\rho = 0$, equation~\eqref{eq:m1m2} simplifies to
\begin{equation*}
	\mu\sqrt{pT} + (\sqrt{1-p} - 1)\sigma \geq 0
	\quad \Longleftrightarrow \quad \frac{\mu \sqrt{T}}{\sigma} \geq \frac{1 - \sqrt{1 - p}}{\sqrt{p}} \quad
 	     \Longleftarrow \quad \frac{\mu \sqrt{T}}{\sigma} \geq \sqrt{p}, 
\end{equation*}
where the two implications follow from algebraic manipulations and the fact that $\sqrt{p} \geq \frac{1 - \sqrt{1 - p}}{\sqrt{p}}$ for $p \in [0, 1]$, respectively. One readily verifies that the last inequality is satisfied by $p = \frac{T\mu^2}{T\mu^2 + \sigma^2}$.

For $\rho < 0$, substituting $p$ in~\eqref{eq:m1m2} with its definition from~\eqref{eq:construct-p} yields
\begin{equation*}
\begin{aligned}
	\mu\sqrt{pT} + (\sqrt{(1+(T-1)\rho)(1-p)} - \sqrt{1-\rho})\sigma 
	&= \frac{T\mu\sqrt{-\rho}}{\sqrt{1-\rho}} + \left( \frac{1+(T-1)\rho}{\sqrt{1-\rho}} - \sqrt{1-\rho} \right)\sigma \\
	&\geq \frac{-T\rho\sigma}{\sqrt{1-\rho}} + \left( \frac{1+(T-1)\rho}{\sqrt{1-\rho}} - \sqrt{1-\rho} \right)\sigma  \\
	&= 0, 
\end{aligned}
\end{equation*}
where the equalities follow from direct calculations and the inequality holds since $\mu^2 + \rho \sigma^2 \geq 0$. We thus conclude that $m_2 \leq Tm_1^2$ whenever $\rho < 0$ as postulated.

\paragraph{Step 3:} We show that our choice of $m_1, m_2$ and $z$ meets the requirements (i') and (ii'), regardless of the value of $p$. First, a direct calculation shows that requirement (i') follows from the definitions of $m_1$ and $z$. Next, the first requirement in (ii') follows from
\begin{equation*}
\begin{aligned}
	pm_2 + (1-p)z^2 &= pm_2 + (1-p)z^2 - \left( pm_1 + (1-p)z \right)^2 + \mu^2 \\
	&= p (m_2 - m_1^2) + \left( pm_1^2 + (1-p)z^2 \right) - \left( pm_1 + (1-p)z \right)^2 + \mu^2 \\
	&= p (m_2 - m_1^2) + p(1-p) (m_1 - z)^2 + \mu^2 \\
	&= \textstyle\frac{1}{T}(1-\rho)(T-1)\sigma^2 + \frac{1}{T}(1 + (T-1)\rho)\sigma^2 + \mu^2 \\
	&= \sigma^2 + \mu^2,
\end{aligned}
\end{equation*}
where the first equality holds since the requirement (i') is met, and the fourth equality follows from the definitions of $m_1$, $m_2$ and $z$.

Finally, to prove the second requirement in (ii'), we first observe that
\begin{equation*}
	pm_2 - \frac{p}{T-1}(Tm_1^2 - m_2) 
	= \frac{pT}{T-1}(m_2 - m_1^2)
	= (1 - \rho)\sigma^2,
\end{equation*}
where the second equality follows from the definition of $m_2$. Note that the term on the left (right) side of this equality constitutes the difference between the left (right) sides of the requirements in (ii'). The second requirement in (ii') and the claim thus follow. 
\qed
\end{proof}

In order to establish Chebyshev bounds for products of random variables, we will formulate generalized moment problems that optimize over the probability measures in the ambiguity set $\mathcal{P}$. We can then leverage powerful duality results from convex optimization to reformulate these moment problems as explicit semidefinite programs that are amenable to efficient solution via interior point methods. The \emph{weak duality} principle, which holds true for every optimization problem, states that the optimal value of a (primal) minimization problem is bounded from below by the optimal value of its associated dual (maximization) problem. To establish tight probability bounds, we need to invoke the \emph{strong duality} principle, which states that under certain conditions the optimal values of the primal and dual optimization problems coincide. In our setting, strong duality holds whenever $\mu^2 + \rho \sigma^2 > 0$.


\begin{thm}[Slater Condition]
\label{thm:slater}
If $\mu^2 + \rho\sigma^2 > 0$, then the moment vector $(1, \bm{\mu}, \bm{\Sigma} + \bm{\mu} \bm{\mu}^\intercal)$ is contained in the interior of the moment cone $\mathcal K$ defined through
\begin{equation*}
	\mathcal{K} = \left\{	\left( 
		\int_{\mathbb{R}^T_+} \mathbb{P}(\dd \bm\xi), \;
		\int_{\mathbb{R}^T_+} \bm{\xi} \, \mathbb{P}(\dd \bm\xi), \;
		\int_{\mathbb{R}^T_+} \bm{\xi}\bm{\xi}^\intercal \, \mathbb{P}(\dd \bm\xi)
	\right): \mathbb{P} \in \mathcal{M}_+(\mathbb{R}^T_+) \right\}.
\end{equation*}
\end{thm}
\begin{proof}
We first show that $\mathcal{P}$ contains a distribution of the form~\eqref{eq:construct-xyz} where the inequalities $x \geq y \geq 0$, $z \geq 0$ and $p \in [0, 1]$ hold \emph{strictly}, as well as $x + (T-1)y > Tz$ (Step~1). This distribution allows us to show that $(1, \bm{\mu}, \bm{\Sigma} + \bm{\mu} \bm{\mu}^\intercal)$ is in the relative interior of $\mathcal{K}_1 = \mathcal{K} \cap (\{ 1\} \times \mathbb{R}^T_+ \times \mathbb{S}^T_+)$ (Step~2), from which the result follows directly by re-scaling the measures in $\mathcal{K}_1$ (Step~3).

\paragraph{Step 1:} 
We distinguish the cases $\rho < 0$ and $\rho \geq 0$. For $\rho < 0$, one readily verifies that the choice of $p$, $x$, $y$ and $z$ in the proof of Proposition~\ref{prop:non-emptiness} satisfies $x > y > 0$, $z > 0$, $p \in (0, 1)$ and $x + (T-1)y > Tz$ by construction. Moreover, these inequalities are also satisfied strictly for $\rho \geq 0$ if we replace $p$ in \eqref{eq:construct-p} with any value from the open interval $(0, p)$.

\paragraph{Step 2:} To prove that $(1, \bm{\mu}, \bm{\Sigma} + \bm{\mu} \bm{\mu}^\intercal) \in \text{rel} \, \text{int} \, \mathcal{K}_1$, we show that all perturbed ambiguity sets
\begin{equation*}
\begin{aligned}
	\mathcal P (\bm\mu^\epsilon, \bm\Omega^\epsilon) = 
	\left \lbrace 
		\mathbb{P} \in \mathcal M_+(\mathbb R^T_+)~:~
			\mathbb{P} \left( \tilde{\bm{\xi}} > \bm 0 \right) = 1 ,\; \mathbb{E}_\mathbb{P} \left( \tilde{\bm{\xi}} \right) = \bm \mu^\epsilon ,\;
			\mathbb{E}_\mathbb{P} \left( \tilde{\bm{\xi}}\tilde{\bm{\xi}}^\intercal \right) = \bm\Omega^\epsilon  
	\right \rbrace
\end{aligned}
\end{equation*}
with $\bm\mu^\epsilon \in \mathcal{B}_\epsilon (\bm\mu)$ and $\bm\Omega^\epsilon \in \mathcal{B}_\epsilon (\bm\Sigma + \bm\mu \bm\mu^\intercal)$ are non-empty for sufficiently small $\epsilon$, where $\mathcal{B}_\epsilon (\bm{x})$ denotes the $\epsilon$-ball around $\bm{x}$ in the respective space. Note that the covariance matrix of any distribution in $\mathcal{P} (\bm\mu^\epsilon, \bm\Omega^\epsilon)$ is positive definite for small $\epsilon$ since $\bm{\Sigma} \succ \bm{0}$ and the eigenvalues are continuous functions of the second-order moment matrix.
In the following, we construct a discrete distribution $\mathbb{P}^\epsilon \in\mathcal P(\bm \mu^\epsilon,\bm \Omega^\epsilon)$ with
\begin{equation}
\label{eq:disc_dist}
\begin{aligned}
	\mathbb{P}^\epsilon \left( \tilde{\bm\xi} = \bm{\xi}^{\epsilon, i} \right) = \frac{p}{T} \quad i = 1, \ldots, T  \quad \text{and} \quad
	\mathbb{P}^\epsilon \left( \tilde{\bm\xi} = \bm{\xi}^{\epsilon, T+1} \right) = 1- p,
\end{aligned}
\end{equation}
where $p$ is the constant chosen in Step~1. The moment conditions for $\mathbb P^\epsilon$ then simplify to:
\begin{enumerate}[(i)]
	\item $\displaystyle \mathbb{E}_{\mathbb{P}^\epsilon} [\bm{\tilde{\xi}}] = \bm{\mu}^\epsilon \mspace {21mu} \quad \Longleftrightarrow \quad \frac{p}{T} \sum_{i=1}^T \xi^{\epsilon, i}_t + (1-p) \xi^{\epsilon, T+1}_t = \mu^\epsilon_t \mspace{87mu} \forall t = 1, \ldots, T$;
	\item $\displaystyle \mathbb{E}_{\mathbb{P}^\epsilon} [\bm{\tilde{\xi}} \bm{\tilde{\xi}}{}^\intercal] = \bm{\Omega}^\epsilon \quad \Longleftrightarrow \quad \frac{p}{T} \sum_{i=1}^T \left( \xi^{\epsilon, i}_t \right)^2 + (1-p) \left( \xi^{\epsilon, T+1}_t \right)^2 = \Omega^\epsilon_{tt} \mspace{25mu} \forall t = 1, \ldots, T$, \\[3mm]
		     \hphantom{a} $\mspace{166mu} \displaystyle \frac{p}{T} \sum_{i=1}^T \xi^{\epsilon, i}_s \xi^{\epsilon, i}_t + (1-p) \xi^{\epsilon, T+1}_s \xi^{\epsilon, T+1}_t = \Omega^\epsilon_{st} \mspace{30mu} \forall 1 \leq s < t \leq T$.
\end{enumerate}
These moment conditions represent a system of nonlinear equations $\bm{F} (\bm{\mu}^\epsilon, \bm{\Omega}^\epsilon; \{ \bm{\xi}^{\epsilon, i} \}_{i=1}^{T+1}) = \bm{0}$ in the moments $\bm{\mu}^\epsilon$ and $\bm{\Omega}^\epsilon$ as well as the atoms $\bm{\xi}^{\epsilon, i}$, $i = 1, \ldots, T + 1$, of the distribution~$\mathbb P^\epsilon$. From Step~1 we know that $\bm{F} (\bm{\mu}, \bm{\Sigma} + \bm{\mu} \bm{\mu}^\intercal; \{ \bm{\xi}^i \}_{i=1}^{T+1}) = \bm{0}$ for $\bm\xi^i = y\bm 1 + (x-y)\mathbf e_i$, $i = 1, \ldots, T$, $\bm\xi^{T+1} = z \bm 1$ and for some $x, y, z \in \mathbb{R}_+$ satisfying $x > y > 0$, $z > 0$ and $x + (T - 1) y > T z$. Moreover, the implicit function theorem proves the existence of continuously differentiable functions $\bm{g}^i : \mathbb{R}^T_+ \times \mathbb{S}^T_+ \rightarrow \mathbb{R}^T$, $i = 1, \ldots, T + 1$, such that $\bm{F} (\bm{\mu}^\epsilon, \bm{\Omega}^\epsilon; \{ \bm{g}^i (\bm{\mu}^\epsilon, \bm{\Omega}^\epsilon) \}_{i=1}^{T+1}) = \bm{0}$ for all $\bm\mu^\epsilon \in \mathcal{B}_\epsilon (\bm\mu)$ and $\bm\Omega^\epsilon \in \mathcal{B}_\epsilon (\bm\Sigma + \bm\mu \bm\mu^\intercal)$, provided that $\epsilon$ is sufficiently small, $\bm{F}$ is continuously differentiable, and the Jacobian of $\bm{F}$ with respect to $\bm{\xi}^{\epsilon, i}$ has full row rank at $(\bm{\mu}^\epsilon, \bm{\Omega}^\epsilon, \{ \bm{\xi}^{\epsilon, i} \}_{i=1}^{T+1}) = (\bm{\mu}, \bm\Sigma + \bm\mu \bm\mu^\intercal, \{ \bm{\xi}^i \}_{i=1}^{T+1})$. Thus, the functions $\bm{g}^i$ allow us to construct distributions of the form~\eqref{eq:disc_dist} that satisfy the moment conditions of the perturbed ambiguity sets $\mathcal P (\bm\mu^\epsilon, \bm\Omega^\epsilon)$ for all $\bm\mu^\epsilon \in \mathcal{B}_\epsilon (\bm\mu)$ and $\bm\Omega^\epsilon \in \mathcal{B}_\epsilon (\bm\Sigma + \bm\mu \bm\mu^\intercal)$. Since each $\bm{g}^i$ is continuous, we have $\bm{g}^i (\bm{\mu}^\epsilon, \bm{\Omega}^\epsilon) > \bm{0}$ for all $\bm\mu^\epsilon \in \mathcal{B}_\epsilon (\bm\mu)$ and $\bm\Omega^\epsilon \in \mathcal{B}_\epsilon (\bm\Sigma + \bm\mu \bm\mu^\intercal)$ when $\epsilon$ is sufficiently small, that is, the support of $\mathbb P^\epsilon$ is contained in $\mathbb R^T_+$, and thus $\mathbb P^\epsilon$ is indeed contained in $\mathcal P (\bm\mu^\epsilon, \bm\Omega^\epsilon)$.

The moment function $\bm{F}$ is continuously differentiable by construction. To apply the implicit function theorem, we therefore only need to show that the Jacobian $\mathbf{J}$ of $\bm{F}$ with respect to $\bm\xi^{\epsilon, 1}$, \ldots, $\bm\xi^{\epsilon, T+1}$ has full row rank at $(\bm{\mu}^\epsilon, \bm{\Omega}^\epsilon, \{ \bm{\xi}^{\epsilon, i} \}_{i=1}^{T+1}) = (\bm{\mu}, \bm\Sigma + \bm\mu \bm\mu^\intercal, \{ \bm{\xi}^i \}_{i=1}^{T+1})$. For ease of exposition, we divide the first $T^2$ and the last $T$ columns of $\mathbf{J}$ by $\frac{p}{T}$ and $1-p$, respectively, and we divide the rows corresponding to the first requirement in (ii) by 2. We then obtain 
\begin{equation*}
	\mathbf J = \left[ \begin{array}{c|c|c|c|c}
		\mathbb I & \mathbb I & \cdots & \mathbb I & \mathbb I \\[-1mm]
		\hline
		& & & & \\[-5mm]
		y \mathbb I + (x-y) \mathbf e_1 \mathbf e_1^\intercal & y \mathbb I + (x-y) \mathbf e_2 \mathbf e_2^\intercal & \cdots & y \mathbb I + (x-y) \mathbf e_T \mathbf e_T^\intercal & z \mathbb I \\
		\hline
		& & & & \\[-5mm]
		\mathbf C^1 & \mathbf C^2 & \cdots & \mathbf C^T & \mathbf C^{T+1} 
	\end{array} \right],
\end{equation*}
where for $i = 1, \ldots, T$, the matrix $\mathbf{C}^i \in \mathbb{R}^{\binom{T}{2} \times T}$ satisfies
\begin{equation*}
	C^i_{st,j} = \begin{dcases}
		x & \text{if } (s,t) \in \{ (i, j), (j, i) \}, \\
		y & \text{if } (s,t) \in \{ (j, \tau) \, : \, \tau \neq i \} \cup \{ (\tau, j) \, : \, \tau \neq i \}, \\
		0 & \text{otherwise.}
	\end{dcases}
\end{equation*}
Here, the indices $s$ and $t$, $1 \leq s < t \leq T$, encode the row and the index $j$ refers to the column of $\mathbf{C}^i$, respectively. The matrix $\mathbf C^{T+1}$ is defined analogously with $x$ and $y$ replaced by $z$.

Consider the linear combination $(\bm{m}^\intercal, \bm{v}^\intercal, \bm{c}^\intercal) \, \mathbf{J}$ of all rows of $\mathbf{J}$ with the coefficients $m_t$ $(t = 1, \hdots, T)$ for the first block of $T$ rows, $v_t$ $(t = 1, \hdots, T)$ for the second block of $T$ rows, and $c_{st}$ for the third block of $\binom{T}{2}$ rows. For notational convenience, we define $c_{st} = c_{ts}$ for $s > t$. To prove that $\mathbf{J}$ has full row rank, we need to show that $(\bm{m}^\intercal, \bm{v}^\intercal, \bm{c}^\intercal) \, \mathbf{J}$ evaluates to $\bm{0}^\intercal$ only if $\bm{m}$, $\bm{v}$ and $\bm{c}$ vanish. To this end, consider the first and the $(T+1)$th element (i.e., the first elements of the first two column blocks) of the equation $(\bm{m}^\intercal, \bm{v}^\intercal, \bm{c}^\intercal) \, \mathbf{J}=\bm 0^\intercal$, which are equivalent to
\begin{equation*}
\begin{aligned}
	m_1 + xv_1 + y \sum_{t=2}^T c_{1t} = 0 \quad \text{and} \quad
	m_1 + yv_1 + x c_{12} + y \sum_{t=3}^T c_{1t} = 0.
\end{aligned}
\end{equation*}
Subtracting the two equations implies that $(x-y)(v_1 - c_{12}) = 0$, which in turn yields $v_1 = c_{12}$ since $x \neq y$. Generalizing this observation to the $t$th columns in each pair of column blocks $s$ and $t$, we find that all $v_t$ and $c_{st}$ must be equal to a single variable $v$. Next, consider the $(T^2+1)$th and $(T^2+2)$th columns (i.e., the first two elements of the last column block) of  the equation $(\bm{m}^\intercal, \bm{v}^\intercal, \bm{c}^\intercal) \, \mathbf{J}=\bm 0^\intercal$, which are equivalent to

\begin{equation*}
\begin{aligned}
	m_1 + zv_1 + z\sum_{t = 2}^T c_{1t} = 0 \quad \text{and} \quad
	m_2 + zv_2 + z\left( c_{21} + \sum_{t = 3}^T c_{2t}\right) = 0.
\end{aligned}
\end{equation*}
However, since $v_t = c_{st} = v$ for all $s$ and $t$, we conclude that $m_1 = m_2$. Again, generalizing this observation to each pair of columns in the last column block, we can identify all $m_t$ by a single number $m$. Replacing $v_t$ and $c_{st}$ by $v$ and $m_t$ by $m$, the previous two equations simplify to
\begin{equation*}
	m + (x + (T-1)y)v = 0 \quad \text{and} \quad
	m + Tzv = 0,
\end{equation*}
and we conclude that $m = v = 0$ since we established earlier that $x + (T-1)y \neq Tz$. Hence, the Jacobian $\mathbf J$ indeed has full row rank, which concludes Step 2.

\paragraph{Step 3:} We have shown in Step~2 that $\mathcal{P} (\bm{\mu}^\epsilon, \bm{\Omega}^\epsilon) \neq \emptyset$ for all $\bm{\mu}^\epsilon \in \mathcal{B}_\epsilon (\bm{\mu})$ and $\bm{\Omega}^\epsilon \in \mathcal{B}_\epsilon (\bm{\Sigma} + \bm{\mu} \bm{\mu}^\intercal)$, which implies that $(1, \bm{\mu}, \bm{\Sigma} + \bm{\mu} \bm{\mu}^\intercal) \in \text{rel} \, \text{int} \, \mathcal{K}_1$. Since $\{ \lambda \mathcal{K}_1 \, : \, \lambda \in \mathbb{R}_+ \} \subseteq \mathcal{K}$, we have $\lambda \mathcal{P} (\bm{\mu}^\epsilon, \bm{\Omega}^\epsilon) \subseteq \mathcal{K}$ for all $\lambda \geq 0$. As the moments are linear in the measure, we thus conclude that $(1, \bm{\mu}, \bm{\Sigma} + \bm{\mu} \bm{\mu}^\intercal) \in \text{int} \, \mathcal{K}$ as desired.
\qed
\end{proof}

Theorem~\ref{thm:slater} will allow us to use the strong duality theorem of~\cite[Proposition~3.4]{Shapiro01}, which states that a linear optimization problem over the distributions in $\mathcal{P}$ has the same optimal value as its associated dual problem. In the remainder of the paper, we will make extensive use of this insight, and we therefore assume from now on that $\mu^2 + \rho \sigma^2 > 0$.

\section{Left-Sided Chebyshev Bounds}
\label{section:right-prob}

In this section we study \emph{left-sided Chebyshev bounds} of the form
\begin{equation*}
\text{L}(\gamma) = \sup_{\mathbb{P} \in \mathcal{P}} \mathbb{P} \left( \prod_{t=1}^T \tilde{\xi}_t \leq \gamma \right),
\end{equation*}
where the ambiguity set $\mathcal{P}$ is defined in~\eqref{eq:our_ambiguity_set}. We begin with the main result of this section.

\begin{thm}[Left-Sided Chebyshev Bound]
\label{thm:rprob}
Let $\gamma > 0$. For all $T \geq 3$, the left-sided Chebyshev bound $\text{\em L}(\gamma)$ coincides with the optimal objective value of the semidefinite program
\begin{equation}
\label{eq:rprob0}
\begin{aligned}
	&\inf && \alpha + T\mu\beta + 
			 T(\mu^2 + \sigma^2)\gamma_1 +
			 T\left[ T\mu^2 + \sigma^2 + (T-1)\rho\sigma^2\right] \gamma_2 \\
	&\st  && \alpha, \beta, \gamma_1, \gamma_2 \in \mathbb{R},~
	         \lambda_1,\lambda_2,\lambda_3 \geq 0,~
	         \bm p \in \mathbb{R}^{2T+1}, ~\bm{P} \in \mathbb{S}^{T+1}_+,~
	         \bm q \in \mathbb{R}^{2T-1}, ~ \bm{Q} \in \mathbb{S}^{T}_+ \\
	&     && \alpha \geq 1, \quad
	         \gamma_1 + \gamma_2 \geq 0, \quad
	         \gamma_1 + T\gamma_2 \geq 0 \\
	&     && \gamma_2 + \frac{\gamma_1}{T} + \alpha \geq \left\Vert \left( \beta - \lambda_1, \gamma_2 + \frac{\gamma_1}{T} - \alpha \right) \right\Vert_2 \\
	&     && \gamma_2 + \gamma_1 + \alpha - 1 \geq \left\Vert \left( \beta - \lambda_2, \gamma_2 + \gamma_1 - \alpha + 1 \right) \right\Vert_2 \\
	&     && \gamma_2 +\frac{\gamma_1}{T} + \lambda_3 + \alpha - 1 \geq \left\Vert \left( \beta - \lambda_3 T\gamma^{1/T}, \gamma_2+ \frac{\gamma_1}{T} + \lambda_3 - \alpha + 1 \right) \right\Vert_2 \\
&     && p_0 = (T-1)\gamma_1\gamma^{\frac{2}{T-1}} + (T-1)^2\gamma_2\gamma^{\frac{2}{T-1}}, \quad  p_1+q_0 =(T-1)\beta\gamma^{\frac{1}{T-1}} \\
	&     && p_2+q_1= \alpha-1, \quad p_T+q_{T-1}=  2(T-1)\gamma_2\gamma^{\frac{1}{T-1}} ,\quad p_{T+1}+q_T= \beta\\ 
	&     && p_{2T}= \gamma_1+\gamma_2,\quad p_t + q_{t-1} = 0 \quad \forall t = 3, \hdots, T-1,T+2,\ldots , 2T-1 \\
	&     && p_t = \textstyle\sum_{i+j = t} P_{i,j} \quad \forall t = 0, \hdots, 2T, \quad
			 q_t = \textstyle\sum_{i+j = t} Q_{i,j} \quad \forall t = 0, \hdots, 2T-2,
\end{aligned}
\end{equation}
where we use the convention that the entries of $\bm p$, $\bm P$, $\bm q$ and $\bm Q$ are numbered starting from~$0$. For $T=2$, $\text{\em L}(\gamma)$ is given by a variant of~\eqref{eq:rprob0} where the constraints $p_2+q_1= \alpha-1$ and $p_T+q_{T-1}=  2(T-1)\gamma_2\gamma^{\frac{1}{T-1}}$ are combined to $p_2+q_1= \alpha-1+2(T-1)\gamma_2\gamma^{\frac{1}{T-1}}$.
\end{thm}
\begin{proof} We first reformulate the maximum probability of the left tail of the product~$\prod_{t=1}^T \tilde{\xi}_t$ falling below $\gamma$ as the generalized moment problem
\begin{equation}
\label{eq:rprob}
\begin{array}{r@{}cl}
	\text{L}(\gamma) =
	&\displaystyle \sup 
		& \displaystyle \int_{\mathbb{R}_+^{T}}  1_{\lbrace \prod_{t=1}^T \xi_t \leq \gamma \rbrace}\ \mathbb{P} (\dd \bm\xi)  \\
	&\st& \displaystyle \mathbb P\in\mathcal M_+(\mathbb R^T_+) \\
	&   &  \displaystyle \int_{\mathbb{R}_+^{T}} \mathbb{P} (\dd\bm\xi) = 1 \\
	&   & \displaystyle \int_{\mathbb{R}_+^{T}} \bm\xi\ \mathbb{P} (\dd\bm\xi) = \bm \mu \\
	&   & \displaystyle \int_{\mathbb{R}_+^{T}} \bm\xi\bm\xi^\intercal\ \mathbb{P} (\dd\bm\xi) = \bm \Sigma+\bm{\mu\mu}^\intercal.
\end{array}
\end{equation}
This moment problem admits a strong conic dual in the Lagrange multipliers $\alpha \in \mathbb{R}$, $\bm{\beta} \in \mathbb{R}^T$ and $ \mathbf{\Gamma} \in \mathbb{S}^T$ corresponding to the normalization, mean and covariance constraints in~\eqref{eq:rprob}, respectively, see Theorem~\ref{thm:slater} and~\cite[Proposition~3.4]{Shapiro01}. Recalling that $\bm \mu = \mu \bm 1$ and $\bm \Sigma=(1-\rho)\sigma^2\mathbb I+ \rho\sigma^2\bm 1\bm 1^\intercal$, the dual problem can be expressed as
\begin{equation}
\label{opt:rprob_dual}
\begin{array}{rl}
	\text{L}(\gamma) =\
	\inf & \alpha + \mu \bm{1^{\intercal} \beta} + 
			 \left< (1 - \rho)\sigma^2\mathbb{I} + \left( \mu^2 + \rho\sigma^2 \right) \bm{11}^\intercal, \bm{\Gamma}\right> \\
	\st  & \alpha \in \mathbb{R},~ \bm{\beta} \in \mathbb{R}^T,~ \bm{\Gamma} \in \mathbb{S}^T \\
	& \alpha + \bm{\xi^{\intercal} \beta} + \bm{\xi^{\intercal}\Gamma \xi} \geq 0 \quad \forall \bm{\xi} \geq \bm{0} \\
	& \alpha + \bm{\xi^{\intercal} \beta} + \bm{\xi^{\intercal}\Gamma \xi} \geq 1 \quad \forall \bm{\xi} \geq \bm{0}:\ \textstyle\prod_{t=1}^{T} \xi_t \leq \gamma.
	\end{array}
\end{equation}
By Lemma~\ref{prop:rprob_symmetry} below, the symmetry of problem~\eqref{opt:rprob_dual} implies that we may restrict attention to permutation-symmetric solutions of the form $\left( \alpha, \bm{\beta}, \bm{\Gamma}\right)$  with $\bm{\beta}= \beta\bm{1}$ and $\bm{\Gamma} = \gamma_1\mathbb{I} + \gamma_2 \bm{11^{\intercal}}$ for some $\beta, \gamma_1, \gamma_2 \in \mathbb{R}$. Thus, problem~\eqref{opt:rprob_dual} simplifies to
\begin{equation}
\label{opt:rprob_dual_simpl}
\begin{array}{rl}
	\text{L}(\gamma) =\
	\inf & \alpha + T\mu\beta +	T(\mu^2 + \sigma^2)\gamma_1 + T\left[ T\mu^2 + \sigma^2 + (T-1)\rho\sigma^2\right] \gamma_2 \\
	\st  & \alpha, \beta, \gamma_1, \gamma_2 \in \mathbb{R} \\
	& \alpha + \beta \Vert \bm{\xi} \Vert_1 + \gamma_1 \Vert \bm{\xi} \Vert_2^2 + \gamma_2 \Vert \bm{\xi} \Vert_1^2 \geq 0 \quad \forall \bm{\xi} \geq \bm{0} \\
	& \alpha + \beta \Vert \bm{\xi} \Vert_1 + \gamma_1 \Vert \bm{\xi} \Vert_2^2 + \gamma_2 \Vert \bm{\xi} \Vert_1^2 \geq 1
\quad \forall \bm{\xi} \geq \bm{0}:\ \textstyle\prod_{t=1}^{T} \xi_t \leq \gamma.
	\end{array}
\end{equation}
Lemma~\ref{lem:semi_infinite_constraints} then implies that~(\ref{opt:rprob_dual_simpl}) can be reduced to
\begin{equation}
\label{opt:rprob_dual_simpl2}
\begin{array}{rl}
	\text{L}(\gamma) =\
	\inf  & \alpha + T\mu\beta +	T(\mu^2 + \sigma^2)\gamma_1 + T\left[ T\mu^2 + \sigma^2 + (T-1)\rho\sigma^2 \right] \gamma_2 \\
	\st   & \alpha, \beta, \gamma_1, \gamma_2 \in \mathbb{R} \\
	& \displaystyle \inf_{s \geq 0}\ \alpha + \beta s + \gamma_2 s^2 + \frac{\gamma_1}{T}s^2 \geq 0 \\
	& \displaystyle \inf_{s \geq 0}\ \alpha + \beta s + \gamma_2 s^2 + \gamma_1 s^2 \geq 1 \\
	& \displaystyle \inf_{s \geq 0}\ \alpha + \beta s + \gamma_2 s^2 + \gamma_1 s^2 f_T\left(0, \frac{\gamma}{s^T} \right) \geq 1.
\end{array}
\end{equation}
By assigning a Lagrange multiplier $\lambda_1\geq 0$ to the constraint $s\geq 0$ and using the $\mathcal{S}$-lemma \cite{Polik07}, the first constraint in~\eqref{opt:rprob_dual_simpl2} can be reformulated as the linear matrix inequality
\begin{equation*}
\begin{aligned}
	\left[ \begin{array}{cc}
		\gamma_2 + \frac{\gamma_1}{T} & \frac{\beta - \lambda_1}{2} \\
		\frac{\beta - \lambda_1}{2} & \alpha
	\end{array} \right] \succeq \bm{0}\
	&\iff
	\left\lbrace \begin{array}{l}
	\alpha \geq 0 \\ 
	\gamma_2 + \frac{\gamma_1}{T} \geq 0 \\
	(\gamma_2 + \frac{\gamma_1}{T})\alpha \geq  \frac{1}{4} ( \beta - \lambda_1)^2
	\end{array} \right. \\
	&\iff
	\left\lbrace \begin{array}{l}
	\alpha \geq 0 \\ 
	\gamma_1 + T\gamma_2 \geq 0 \\
	\gamma_2 + \frac{\gamma_1}{T} + \alpha \geq \left\Vert \left( \beta - \lambda_1, \gamma_2 + \frac{\gamma_1}{T} - \alpha \right) \right\Vert_2,
	\end{array} \right.
\end{aligned}
\end{equation*}
where the first equivalence follows from the observation that a $2\times 2$-matrix is positive semidefinite iff it has non-negative diagonal elements as well as a non-negative determinant, while the second equivalence uses a well-known reformulation of hyperbolic constraints as second-order cone constraints \cite[p.~197]{Boyd04}. Similarly, the second constraint in~\eqref{opt:rprob_dual_simpl2} holds iff there exists $\lambda_2\geq 0$~with
\begin{equation*}
\begin{aligned}
	\left[ \begin{array}{cc}
		\gamma_2 + \gamma_1 & \frac{\beta - \lambda_2}{2} \\
		\frac{\beta - \lambda_2}{2} & \alpha - 1
	\end{array} \right] \succeq \bm{0}\
	&\iff
	\left\lbrace \begin{array}{l}
	\alpha \geq 1 \\ 
	\gamma_2 + \gamma_1 \geq 0 \\
	\gamma_2 + \gamma_1 + \alpha - 1 \geq \left\Vert \left( \beta - \lambda_2, \gamma_2 + \gamma_1 - \alpha + 1 \right) \right\Vert_2.
	\end{array} \right.
\end{aligned}
\end{equation*}
Lemma~\ref{prop:f_t-decomposition} below further allows us to decompose the third constraint in~\eqref{opt:rprob_dual_simpl2} into two simpler semi-infinite constraints.
\begin{subequations}
\begin{alignat}{2}
	&\hspace{-3mm}\inf_{s \in \left[ 0, T\gamma^{1/T} \right]}   \alpha + \beta s + \gamma_2 s^2 + \gamma_1 \frac{s^2}{T} \geq 1 \label{eq:rprob_dual_3_3} \\
	&\hspace{-3mm}\inf_{s \geq T\gamma^{1/T}} \left\{ \alpha + \beta s + \gamma_2 s^2 + \gamma_1 \min_{\underline\xi, \overline\xi \geq 0} \left\{ \underline\xi{}^2 + (T-1)\overline\xi{}^2 : \underline\xi + (T-1)\overline\xi = s,~ \underline\xi \,\overline\xi{}^{T-1} = \gamma \right\}\right\}\geq 1 \label{eq:rprob_dual_3_4}
\end{alignat}
\end{subequations}
As $s \in \left[ 0, T\gamma^{1/T} \right]$ iff $s ( T\gamma^{1/T} - s )\geq 0$, we can once again use the $\mathcal S$-lemma to show that~\eqref{eq:rprob_dual_3_3} holds iff there exists $\lambda_3\geq 0$ with 
\begin{equation*}
\begin{aligned}
	\left[ \begin{array}{cc}
		\gamma_2 + \frac{\gamma_1}{T} + \lambda_3 & \frac{\beta - \lambda_3 T\gamma^{1/T}}{2} \\
		\frac{\beta - \lambda_3 T\gamma^{1/T}}{2} & \alpha - 1
	\end{array} \right] \succeq \bm{0}\
	&\iff
	\left\lbrace \begin{array}{l}
	\alpha \geq 1 \\ 
	\gamma_2 + \frac{\gamma_1}{T} + \lambda_3 \geq 0 \\
	\gamma_2 + \frac{\gamma_1}{T} + \lambda_3 + \alpha - 1 \\ \quad  \geq \left\Vert \left( \beta - \lambda_3 T\gamma^{1/T}, \gamma_2+ \frac{\gamma_1}{T} + \lambda_3 - \alpha + 1 \right) \right\Vert_2.
	\end{array} \right.
\end{aligned}
\end{equation*}
Finally, it remains to be shown that~(\ref{eq:rprob_dual_3_4}) also admits a conic reformulation. To do so, we first argue that one can replace~(\ref{eq:rprob_dual_3_4}) with
\begin{equation}
\label{eq:rprob_dual_3_5}
	\inf_{s \geq T\gamma^{1/T}, \,\underline\xi, \overline\xi \geq 0} \left\{ \alpha + \beta s + \gamma_2 s^2 + \gamma_1 \left[\underline\xi{}^2 + (T-1)\overline\xi{}^2\right] : \underline\xi + (T-1)\overline\xi = s,~ \underline\xi \,\overline\xi{}^{T-1} = \gamma\right\} \geq 1
\end{equation}
without changing the optimal value of problem~\eqref{opt:rprob_dual_simpl2}. If $\gamma_1 \geq 0$, then~\eqref{eq:rprob_dual_3_5} is indeed equivalent to~\eqref{eq:rprob_dual_3_4}. On the other hand, if $\gamma_1 < 0$, we find
\begin{equation*}
\begin{aligned}
	&\inf_{s \geq T\gamma^{1/T}} \left\{ \alpha + \beta s + \gamma_2 s^2 + \gamma_1 \min_{\underline\xi, \overline\xi \geq 0} \left\{ \underline\xi{}^2 + (T-1)\overline\xi{}^2 : \underline\xi + (T-1)\overline\xi = s,~ \underline\xi \,\overline\xi{}^{T-1} = \gamma \right\}\right\} \\
	\geq &\inf_{s \geq T\gamma^{1/T}} \left\{ \alpha + \beta s + \gamma_2 s^2 + \gamma_1 \max_{\underline\xi, \overline\xi \geq 0} \left\{ \underline\xi{}^2 + (T-1)\overline\xi{}^2 : \underline\xi + (T-1)\overline\xi = s,~ \underline\xi \,\overline\xi{}^{T-1} = \gamma \right\}\right\} \\
	= &\inf_{s \geq T \gamma^{1/T}, \,\underline\xi, \overline\xi \geq 0} \left\{ \alpha + \beta s + \gamma_2 s^2 + \gamma_1 \left[\underline\xi{}^2 + (T-1)\overline\xi{}^2\right] : \underline\xi + (T-1)\overline\xi = s,~ \underline\xi \,\overline\xi{}^{T-1} = \gamma\right\} \\
	\geq &\inf_{s \geq T \gamma^{1/T}}\ \alpha + \beta s + \gamma_2 s^2 + \gamma_1 s^2 ,
\end{aligned} 
\end{equation*}
which means that~\eqref{eq:rprob_dual_3_4} is implied by the second semi-infinite constraint in problem~\eqref{opt:rprob_dual_simpl2}. By eliminating $s=\underline\xi + (T-1)\overline\xi$, the maximization problem on the left hand side of~(\ref{eq:rprob_dual_3_5}) reduces to
\begin{equation*}
\begin{aligned}
	\inf_{\underline\xi,\,\overline\xi \geq 0,~\underline\xi\, \overline\xi{}^{T-1} = \gamma}~\alpha + \beta \left[ \underline\xi + (T-1)\overline\xi \right] + \gamma_2 \left[ \underline\xi + (T-1)\overline\xi \right]^2 + \gamma_1 \left[ \underline\xi{}^2 + (T-1)\overline\xi{}^2 \right].
\end{aligned}
\end{equation*}
Note that the constraint $s \geq T\gamma^{1/T}$ has been dropped in the above formulation. This constraint is redundant due to the inequality of arithmetic and geometric means, which implies that
\[
	s = \underline\xi + (T-1)\overline\xi \geq T (\underline\xi \overline\xi^{T-1})^{1/T} = T \gamma^{1/T}.
\]
By setting $\kappa = \underline \xi^{1/(T-1)}$, we can further replace $\underline\xi$ and $\overline\xi$ with $\kappa^{T-1}$ and $\gamma^{1/(T-1)}/\kappa$, respectively. Using elementary manipulations, one can then show that~(\ref{eq:rprob_dual_3_5}) reduces to
\begin{align}
	\inf_{\kappa\geq 0}~~&
	(T-1)\gamma_1\gamma^{\frac{2}{T-1}} + 
	(T-1)^2\gamma_2\gamma^{\frac{2}{T-1}} + 
	(T-1)\beta\gamma^{\frac{1}{T-1}}\kappa + 
	(\alpha-1)\kappa^2 \nonumber \\ 
	&\hspace{3cm} + 2(T-1)\gamma_2\gamma^{\frac{1}{T-1}}\kappa^T +  \beta\kappa^{T+1} + 
	(\gamma_1 + \gamma_2)\kappa^{2T} ~ \geq ~ 0. \label{eq:polynomial}
\end{align}
Note that the objective of the maximization problem on the left hand side of~\eqref{eq:polynomial} constitutes a polynomial of degree $2T$ in $\kappa$ and is therefore representable as $l(\kappa)=\sum_{i=0}^{2T} a_i\kappa^i$, where
\begin{equation}
	\label{eq:a-coefficients}
	a_i=\left\{ \begin{array}{ll}
		(T-1)\gamma_1\gamma^{\frac{2}{T-1}} + (T-1)^2\gamma_2\gamma^{\frac{2}{T-1}}& \text{if } i=0,\\
		(T-1)\beta\gamma^{\frac{1}{T-1}} & \text{if } i=1,\\
		\alpha-1 & \text{if } i=2,\\
		2(T-1)\gamma_2\gamma^{\frac{1}{T-1}} & \text{if } i=T,\\
		\beta & \text{if }i=T+1,\\
		\gamma_1 + \gamma_2 & \text{if }i=2T,\\
		0 & \text{otherwise.}
	\end{array}
	\right.
\end{equation}
Here we assumed that $T>2$. For $T=2$, the quadratic monomial in $l(\kappa)$ would have the coefficient $\alpha-1+2(T-1)\gamma_2\gamma^{\frac{1}{T-1}}$ instead of $\alpha-1$. Thus, the case $T=2$ could be handled via a case distinction, which we omit for the sake of brevity. 

Constraint~\eqref{eq:rprob_dual_3_5} thus requires the polynomial $l(\kappa)$ to be non-negative for all $\kappa\geq 0$. By the Markov-Lukacs Theorem~\cite{Nudelman}, this is equivalent to postulating that $l(\kappa)$ admits a sum-of-squares representation of the form $l(\kappa) = p(\kappa) + \kappa q(\kappa)$, where $p(\kappa) = \sum_{i=0}^{2T}p_i\kappa^i$ and $q(\kappa) = \sum_{i=0}^{2T-2} q_i\kappa^i$ are sum-of-squares polynomials of degrees $2T$ and $2T-2$, respectively. By matching the coefficients of all monomials, one verifies that the identity $l(\kappa) = p(\kappa) + \kappa q(\kappa)$ holds iff
\begin{equation}
\label{eq:coefficient-matching}
	p_0 = a_0, \quad 
	p_t + q_{t-1} = a_t \quad \forall t = 1, \hdots, 2T-1 \quad \text{and} \quad
	p_{2T} = a_{2T},
\end{equation}
Moreover, by~\cite[Theorem~3]{Nesterov00}, $p(\kappa)$ and $q(\kappa)$ are sum-of-squares polynomials iff there exist positive semidefinite matrices $\bm{P} \in \mathbb{S}^{T+1}_+$ and $\bm{Q} \in \mathbb{S}^{T}_+$ such that
\begin{equation}
\label{eq:sum-of-squares}
	p_t = \sum_{i+j = t} P_{i,j} \quad \forall t = 0, \hdots, 2T
	\quad \text{and} \quad
	q_t = \sum_{i+j = t} Q_{i,j} \quad \forall t = 0, \hdots, 2T-2.
\end{equation}
Thus, \eqref{eq:rprob_dual_3_5} holds iff the conic constraints~\eqref{eq:coefficient-matching} and~\eqref{eq:sum-of-squares} are satisfied. The claim now follows by replacing the three semi-infinite constraints in~\eqref{opt:rprob_dual_simpl2} with their explicit conic reformulations. \qed
\end{proof}

The proof of Theorem~\ref{thm:rprob} relies on 4 auxiliary lemmas, which we prove next.

\begin{lemma}
\label{prop:rprob_symmetry}
Problem~(\ref{opt:rprob_dual}) has a permutation symmetric minimizer $\left( \alpha^\star, \bm{\beta}^\star, \bm{\Gamma}^\star\right)$ that satisfies $\bm{\beta}^\star = \beta^\star\bm{1}$ and $\bm{\Gamma}^\star = \gamma_1^\star \mathbb{I} + \gamma_2^\star \bm{11^{\intercal}}$ for some $\beta^\star, \gamma_1^\star, \gamma_2^\star \in \mathbb{R}$.
\end{lemma}
\begin{proof}
Let $\mathfrak P$ be the set of all permutations of the index set $\{1,\ldots, T\}$. For any $\pi\in \mathfrak P$ we denote by $\mathbf{P}_\pi\in \mathbb{R}^{T \times T}$ the permutation matrix defined through $(\mathbf{P}_\pi)_{ij}=1$ if $\pi(i)=j$; $=0$ otherwise. Let $\left( \alpha, \bm{\beta}, \bm{\Gamma}\right)$ by any optimal solution to~(\ref{opt:rprob_dual}), which exists by~\cite[Proposition~3.4]{Shapiro01}. We first show that the permuted solution $(\alpha_\pi,\bm\beta_\pi,\bm \Gamma_\pi)= \left( \alpha, \mathbf{P}_\pi\bm{\beta}, \mathbf{P}_\pi\bm{\Gamma}\mathbf{P}_\pi^\intercal\right)$ is also optimal in~(\ref{opt:rprob_dual}). To this end, we observe that
\begin{equation*}
\begin{aligned}
	&\alpha_\pi + \mu \bm{1^{\intercal}}\bm{\beta}_\pi + 
		\left< (1 - \rho)\sigma^2\mathbb{I} + \left( \mu^2 + \rho\sigma^2 \right) \bm{11}^\intercal, \bm{\Gamma}_\pi \right> \\
	=~&\alpha + \mu \bm{1^{\intercal}}\mathbf{P}_\pi\bm{\beta} + 
		\left< (1 - \rho)\sigma^2\mathbb{I} + \left( \mu^2 + \rho\sigma^2 \right) \bm{11}^\intercal, \mathbf{P}_\pi\bm{\Gamma}\mathbf{P}_\pi^\intercal \right> \\
	=~&\alpha + \mu (\mathbf{P}_\pi^\intercal\bm{1})^{\intercal}\bm{\beta} + 
		\left< (1 - \rho)\sigma^2 \mathbf{P}_\pi^\intercal \mathbf{P}_\pi + \left( \mu^2 + \rho\sigma^2 \right) \mathbf{P}_\pi^\intercal \bm{1} (\mathbf{P}_\pi^\intercal \bm{1})^\intercal , \bm{\Gamma} \right> \\
	=~&\alpha + \mu \bm{1}^\intercal \bm{\beta} + 
		\left< 1 - \rho)\sigma^2\mathbb{I} + \left( \mu^2 + \rho\sigma^2 \right) \bm{11}^\intercal, \bm{\Gamma} \right>,
\end{aligned}
\end{equation*}
where the first equality follows from the definition of $\alpha_\pi$, $\bm \beta_\pi$ and $\bm \Gamma_\pi$, the second equality exploits the cyclicity property of the trace scalar product, and the third equality holds due to the permutation symmetry of $\bm 1$ and the fact that $\mathbf{P}_\pi^\intercal = \mathbf{P}_{\pi^{-1}}=\mathbf{P}_\pi^{-1}$. Thus, $(\alpha_\pi,\bm\beta_\pi,\bm \Gamma_\pi)$ has the same objective value as $\left( \alpha, \bm{\beta}, \bm{\Gamma}\right)$.
To show that $\left( \alpha_\pi, \bm{\beta}_\pi, \bm{\Gamma}_\pi\right)$ is feasible in~\eqref{opt:rprob_dual}, we note that
\begin{equation*}
\begin{array}{cll}
	& \alpha_\pi + \bm{\xi}^{\intercal}\bm{\beta}_\pi + \bm{\xi}^{\intercal}\bm{\Gamma}_\pi \bm{\xi} \geq 1_{\lbrace \prod_{t=1}^T \xi_t \leq \gamma \rbrace} & \forall \bm{\xi} \geq \bm{0} \\
	\iff & \alpha + (\mathbf{P}_{\pi^{-1}}\bm{\xi})^{\intercal} \bm{\beta} + (\mathbf{P}_{\pi^{-1}}\bm{\xi})^{\intercal}\ \bm{\Gamma} (\mathbf{P}_{\pi^{-1}}\bm{\xi}) \geq 1_{\lbrace \prod_{t=1}^T \xi_t \leq \gamma \rbrace} & \forall \bm{\xi} \geq \bm{0} \\
	\iff & \alpha + \bm{\xi^{\intercal} \beta} + \bm{\xi}^{\intercal} \bm{\Gamma}\bm{\xi} \geq 1_{\lbrace \prod_{t=1}^T \xi_{\pi(t)} \leq \gamma \rbrace} & \forall \bm{\xi} \geq \bm{0} \\
	\iff & \alpha + \bm{\xi^{\intercal} \beta} + \bm{\xi}^{\intercal} \bm{\Gamma}\bm{\xi} \geq 1_{\lbrace \prod_{t=1}^T \xi_t \leq \gamma \rbrace} & \forall \bm{\xi} \geq \bm{0},
\end{array}
\end{equation*}
where the first equivalence follows from the definition of $\alpha_\pi$, $\bm \beta_\pi$ and $\bm \Gamma_\pi$ and because $\mathbf{P}_\pi^\intercal = \mathbf{P}_\pi^{-1}$, the second equivalence holds because permutations are bijective, and the third equivalence relies on the permutation symmetry of the non-negative orthant. Thus, $(\alpha_\pi,\bm\beta_\pi,\bm \Gamma_\pi)$ satisfies the semi-infinite constraints in~\eqref{opt:rprob_dual} whenever $(\alpha,\bm\beta,\bm \Gamma)$ does. We conclude that $(\alpha_\pi,\bm\beta_\pi,\bm \Gamma_\pi)$ is feasible and thus optimal in~\eqref{opt:rprob_dual} for every $\pi\in\mathfrak P$.

Due to the convexity of the (semi-infinite) linear program~(\ref{opt:rprob_dual}), the equally weighted average $\left( \alpha^\star, \bm{\beta}^\star, \bm{\Gamma}^\star\right)= \frac{1}{T!}\sum_{\pi\in\mathfrak P} (\alpha_\pi,\bm\beta_\pi,\bm \Gamma_\pi)$ constitutes another optimal solution. It is now clear that $\mathbf P_\pi \bm \beta^\star=\bm \beta^\star$ and $\mathbf P_\pi \bm \Gamma^\star\mathbf P_\pi^\intercal=\bm \Gamma^\star$ for any $\pi\in \mathfrak P$ since $\pi(\mathfrak P)=\mathfrak P$. Thus, the claim follows. \qed
\end{proof}

\begin{lemma}\label{lem:semi_infinite_constraints}
For $\alpha, \beta, \gamma_1, \gamma_2, \Delta \in \mathbb{R}$ and $\underline{\gamma}, \overline{\gamma} \in \mathbb{R}_+ \cup \{ \infty \}$, $\underline{\gamma} \leq \overline{\gamma}$, we have
\begin{equation}
\label{eq:rprob_dual_lvl1}
\begin{aligned}
	&\inf_{\bm{\xi} \geq \bm{0}} \left\{ \alpha + \beta \Vert \bm{\xi} \Vert_1 + \gamma_1 \Vert \bm{\xi} \Vert_2^2 + \gamma_2 \Vert \bm{\xi} \Vert_1^2 \, : \, \textstyle \prod_{t = 1}^T \xi_t \in [ \underline{\gamma}, \overline{\gamma} ] \right\} \geq \Delta \\
	\iff\quad & \left\lbrace \begin{array}{l} 
	\displaystyle \inf_{s \geq T \underline{\gamma}^{1/T}} ~ \alpha + \beta s + \gamma_2 s^2 + \gamma_1 s^2 ~ f_T (\underline{\gamma} / s^T, \overline{\gamma} / s^T) \geq \Delta \\
	\displaystyle \inf_{s \geq T \underline{\gamma}^{1/T}} ~ \alpha + \beta s + \gamma_2 s^2 + \gamma_1 s^2 ~ g_T (\underline{\gamma} / s^T, \overline{\gamma} / s^T) \geq \Delta,
	\end{array} \right.
\end{aligned}
\end{equation}
where
\begin{subequations}
 \label{opt:aux-fg}
\begin{align}
	 \label{opt:aux-f}
	f_T(\underline{\gamma}, \overline{\gamma}) &= \inf_{\bm \xi\geq \bm 0}  \left\{\Vert \bm{\xi} \Vert_2^2 :  \Vert \bm{\xi} \Vert_1 = 1,~ \textstyle \prod_{t=1}^T \xi_t \in [ \underline{\gamma}, \overline{\gamma} ] \right\} \\
\text{and} \quad
	g_T(\underline{\gamma}, \overline{\gamma}) &= \sup_{\bm \xi\geq \bm 0}  \left\{\Vert \bm{\xi} \Vert_2^2 :  \Vert \bm{\xi} \Vert_1 = 1,~ \textstyle \prod_{t=1}^T \xi_t \in [ \underline{\gamma}, \overline{\gamma} ] \right\}.
	 \label{opt:aux-g}
	\end{align}
\end{subequations}
Moreover, we have $f_T (\underline{\gamma}, \infty) = 1/T$ for $\underline{\gamma} \leq T^{-T}$ and $g_T (0, \overline{\gamma}) = 1$ for $\overline{\gamma} \in \mathbb{R}_+ \cup \{ \infty \}$.
\end{lemma}

\begin{figure}[tb]
 	\centering
	\includegraphics[width=0.37\paperwidth]{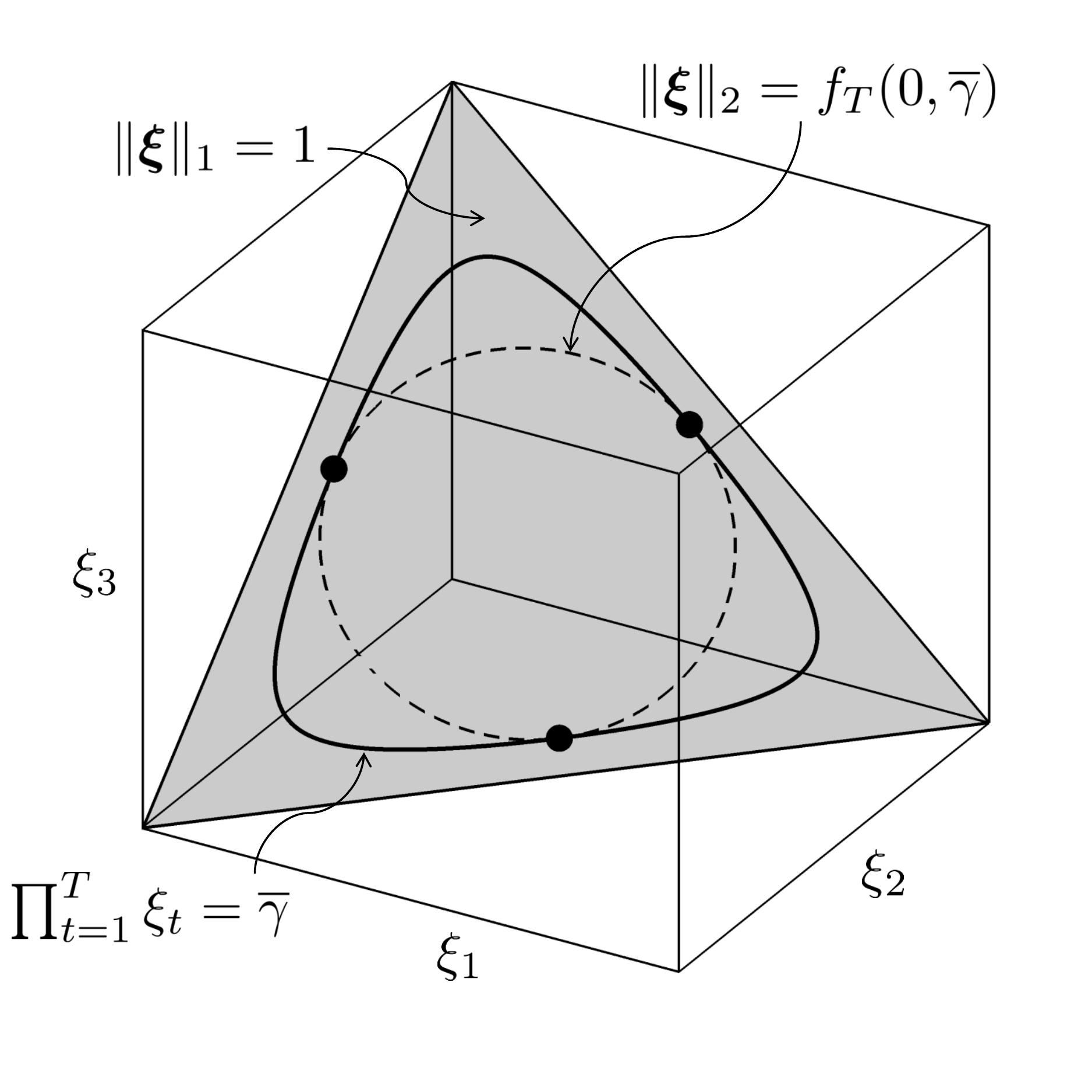}
	\includegraphics[width=0.37\paperwidth]{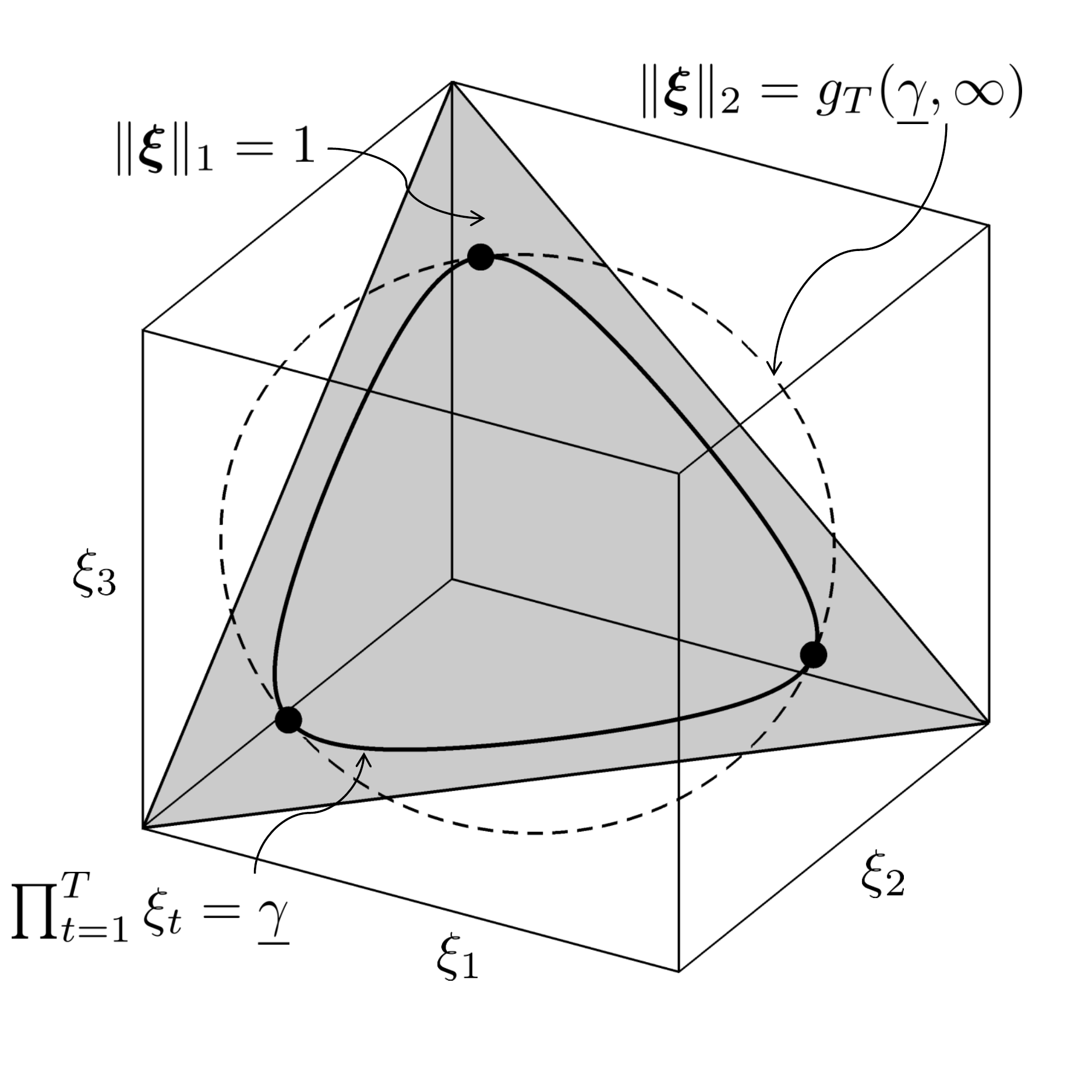}
	\caption{The subproblems \eqref{opt:aux-f} (left) and \eqref{opt:aux-g} (right) determine the smallest and the largest spheres centered at the origin that intersect with the hyperplane $\Vert \bm\xi \Vert_1 = 1$ (shaded areas) and the hyperbola $\prod_{t=1}^T \xi_t = \overline\gamma, \underline\gamma$ (solid lines). The dashed circles represent level sets of the objective function $\Vert \bm\xi \Vert_2^2$. Both graphs illustrate the case where $T = 3$.}
	\label{fig:f_and_g}
\end{figure}

\noindent Figure~\ref{fig:f_and_g} visualizes the two parametric subproblems \eqref{opt:aux-f} and \eqref{opt:aux-g}. Note that both problems are non-convex whenever $\overline{\gamma} < \infty$ as their last constraints are equivalent to $(\prod_{t=1}^T \xi_t)^{1/T} \in [ \underline{\gamma}^{1/T}, \overline{\gamma}^{1/T} ]$ and because geometric means are concave \cite[\S~3.1]{Boyd04}. Moreover, the subproblem~\eqref{opt:aux-g} remains non-convex for $\overline{\gamma} = \infty$ since it maximizes a convex objective function.

\begin{proof}[Proof of Lemma~\ref{lem:semi_infinite_constraints}:]
The first constraint in~\eqref{eq:rprob_dual_lvl1} can be reduced to
\begin{equation}\label{eq:aux_constraint_lemma}
\inf_{s \geq T \underline{\gamma}^{1/T}} ~~ \alpha + \beta s + \gamma_2 s^2 +  
	\inf_{\bm{\xi} \geq \bm{0}} \left\{ \gamma_1 \Vert \bm{\xi} \Vert_2^2 : \Vert \bm{\xi} \Vert_1 = s,~ \textstyle \prod_{t = 1}^T \xi_t \in [ \underline{\gamma}, \overline{\gamma} ] \right\} ~\geq \Delta \\
	\end{equation}
by decomposing the maximization over all $\bm{\xi} \geq \bm{0}$ into two nested maximization problems over all $s \geq T \underline{\gamma}^{1/T}$ and over all $\bm{\xi} \geq \bm{0}$ with $\|\bm\xi\|_1=s$, respectively. Here, the lower bound on $s$ is owed to the fact that there is $\bm{\xi} \geq \bm{0}$ satisfying $\|\bm\xi\|_1=s$ and $\prod_{t = 1}^T \xi_t \in [\underline{\gamma}, \overline{\gamma}]$ if and only if $s \geq T \underline{\gamma}^{1/T}$. A case distinction on the sign of $\gamma_1$ shows that constraint~\eqref{eq:aux_constraint_lemma} holds if and only if
\begin{equation*}
\begin{aligned}
& \left\lbrace \begin{array}{l} 
	\displaystyle  \inf_{s \geq T \underline{\gamma}^{1/T}} ~ \alpha + \beta s + \gamma_2 s^2 +  
	\gamma_1 \inf_{\bm{\xi} \geq \bm{0}} \left\{ \Vert \bm{\xi} \Vert_2^2 : \Vert \bm{\xi} \Vert_1 = s, ~\textstyle \prod_{t = 1}^T \xi_t \in [ \underline{\gamma}, \overline{\gamma} ] \right\} ~\geq \Delta \\
	\displaystyle  \inf_{s \geq T \underline{\gamma}^{1/T}} ~ \alpha + \beta s + \gamma_2 s^2 +  
	\gamma_1 \sup_{\bm{\xi} \geq \bm{0}} \left\{ \Vert \bm{\xi} \Vert_2^2 : \Vert \bm{\xi} \Vert_1 = s,~ \textstyle \prod_{t = 1}^T \xi_t \in [ \underline{\gamma}, \overline{\gamma} ] \right\} ~\geq \Delta 
	\end{array} \right. 
\end{aligned}
\end{equation*}
is satisfied. The change of variables $\bm{\xi} \leftarrow s \bm{\xi}$ shows that this constraint system is equivalent to the second constraint system in~\eqref{eq:rprob_dual_lvl1}. Finally, we have $f_T (\underline{\gamma}, \infty) = 1/T$ for $\underline{\gamma} \leq T^{-T}$ and $g_T (0, \overline{\gamma}) = 1$ for $\overline{\gamma} \in \mathbb{R}_+ \cup \{ \infty \}$ since the inequalities $\frac{1}{T}\Vert \bm{\xi} \Vert_1^2 \leq \Vert \bm{\xi} \Vert_2^2 \leq \Vert \bm{\xi} \Vert_1^2$ are tight for $\bm \xi = \bm 1$ and $\bm \xi = \mathbf e_i$, respectively. \qed
\end{proof}

\begin{lemma}
\label{prop:f_t-decomposition}
For $T \geq 2$, $\underline{\gamma}= 0$ and $\overline{\gamma}\geq 0$, the optimal value $f_T (0, \overline{\gamma})$ of~\eqref{opt:aux-f} equals
\begin{equation}
\label{eq:f_t-decomposition0}
f_T(0, \overline{\gamma}) = \left\{ \begin{array}{ll}
\displaystyle \min_{\underline \xi \geq 0,\, \overline \xi \geq 0}\left\{  \underline \xi^2 + (T-1) \overline \xi{}^2 : \underline \xi + (T-1) \overline \xi = 1,~ \underline \xi \, \overline \xi{}^{T-1} = \overline{\gamma}\right\}  & \text{if } 0\leq \overline{\gamma} \leq T^{-T},\\[2ex]
\frac{1}{T} & \text{if } \overline{\gamma} > T^{-T}.
\end{array} \right.
\end{equation}
\end{lemma}
\begin{proof} We first observe that the non-convex optimization problem~\eqref{opt:aux-f} is bounded below by its relaxation $\min_{\Vert \bm{\xi} \Vert_1 = 1}  \Vert \bm{\xi} \Vert_2^2$. Note, however, that the optimal solution $\bm \xi=\frac{1}{T}\bm 1$ of this relaxation is feasible and thus optimal in~\eqref{opt:aux-f} whenever $\overline{\gamma} \geq T^{-T}$. Thus, we have $f_T(0, \overline{\gamma})= \frac{1}{T}$ for $\overline{\gamma} \geq T^{-T}$. For $0\leq \overline{\gamma}<T^{-T}$, on the other hand, the product constraint $\prod_{t=1}^T \xi_t \leq \overline{\gamma}$ must be binding, for otherwise convex combinations of the optimal solution $\bm{\xi}$ with $\frac{1}{T}\bm 1$ would improve the objective function of $f_T (0, \overline{\gamma})$, which is a contradiction. In summary, we thus find
\begin{equation}
\label{eq:f_t-decomposition}
f_T(0, \overline{\gamma}) = \left\{ \begin{array}{ll}
\inf_{\bm \xi \geq \bm 0}  \left\{\Vert \bm{\xi} \Vert_2^2 :  \Vert \bm{\xi} \Vert_1 = 1,~ \textstyle \prod_{t=1}^T \xi_t = \overline{\gamma} \right\} & \text{if } 0\leq \overline{\gamma} <T^{-T},\\
\frac{1}{T} & \text{if } \overline{\gamma} \geq T^{-T}.
\end{array} \right.
\end{equation}
When $\overline{\gamma}=0$, the product constraint in the first line of~\eqref{eq:f_t-decomposition} can only be satisfied if $\xi_t=0$ for at least one $t$. By permutation symmetry, we may assume without loss of generality that $\xi_T=0$. Then, the product constraint is automatically satisfied and may be disregarded, implying that the minimization problem in the first line of~\eqref{eq:f_t-decomposition} is solved by $\xi_1=\xi_2=\cdots =\xi_{T-1}=\frac{1}{T-1}$ and $\xi_T=0$. We thus conclude that $f_T(0, 0)=\frac{1}{T-1}$ and therefore
\begin{equation}
\label{eq:f_t-decomposition1}
f_T(0, \overline{\gamma}) = \left\{ \begin{array}{ll}
\frac{1}{T-1} & \text{if } \overline{\gamma} =0,\\
\inf_{\bm \xi > \bm 0}  \left\{\Vert \bm{\xi} \Vert_2^2 :  \Vert \bm{\xi} \Vert_1 = 1,~ \textstyle \prod_{t=1}^T \xi_t = \overline{\gamma} \right\} & \text{if } 0< \overline{\gamma} <T^{-T},\\
\frac{1}{T} & \text{if } \overline{\gamma} \geq T^{-T}.
\end{array} \right.
\end{equation}
We now study the non-convex  parametric optimization problem
\begin{equation}
	\label{eq:aux2}
	\min_{\bm \xi > \bm 0}  \left\{\Vert \bm{\xi} \Vert_2^2 :  \Vert \bm{\xi} \Vert_1 = 1,~ \textstyle \prod_{t=1}^T \xi_t = \overline{\gamma} \right\}
\end{equation}
on the domain $0< \overline{\gamma} <T^{-T}$. Observe that~\eqref{eq:aux2} has a non-empty compact feasible set for any admissible $\overline{\gamma}$ and is therefore solvable. Assigning Lagrange multipliers $a$ and $b$ to the norm and product constraints, respectively, we find that any optimal solution to~\eqref{eq:aux2} must satisfy the stationarity conditions
\begin{align*}
	2\xi_t + a +\frac{b}{\xi_t} \prod_{t'=1}^T \xi_{t'} = 0  \quad \forall t = 1, \ldots, T \quad \iff \quad 2\xi_t^2 + a \xi_t+ b\overline{\gamma} = 0 \quad \forall t = 1, \hdots, T,
\end{align*}
where the equivalence follows from primal feasibility. Note that each $\xi_t$ needs to satisfy an identical quadratic equation, which must have two distinct positive real roots\footnote{The existence of at least one real root is guaranteed because~\eqref{eq:aux2} is solvable and because any optimal solution must satisfy the stationarity conditions. In fact, the stationarity conditions must admit {\em two} distinct positive real roots because otherwise~$\bm \xi=\frac{1}{T}\bm 1$ would be the only conceivable optimal solution, which is impossible for $\overline{\gamma}<T^{-T}$.} $\underline \xi$ and $\overline\xi$. The roots depend on $a$, $b$ and $\overline{\gamma}$, but this dependence is notationally suppressed to avoid clutter. At optimality, the decision variables $\xi_1,\xi_2 \hdots, \xi_T$ can thus be partitioned into two groups, where all variables in the first group are equal to~$\underline \xi$, and all variables in the second group are equal to~$\overline \xi$. This structural insight allows us to simplify problem~\eqref{eq:aux2}. Indeed, by permutation symmetry, it is sufficient to consider only solutions that satisfy $\xi_1=\cdots=\xi_k=\underline\xi$ and $\xi_{k+1}=\cdots = \xi_T=\overline\xi$ for some $\underline \xi, \overline \xi>0$ and for some $k\in\{1,\ldots, \lfloor \frac{T}{2} \rfloor \}$. Thus, the optimal value of~\eqref{eq:aux2} coincides with
\begin{equation}
	\label{eq:aux3}
	\min_{k \in\{ 1, \ldots, \lfloor \frac{T}{2}\rfloor\}} f_{T,k}(\overline{\gamma}),
\end{equation}
where the functions $f_{T,k}:(0, T^{-T}) \rightarrow \mathbb{R}$ for $k = 1, 2, \hdots, \lfloor \frac{T}{2} \rfloor$ are defined through
\begin{equation}
\label{opt:aux-f-enum}
\begin{aligned}
	f_{T,k}(\overline{\gamma}) = \min_{\underline \xi > 0,\, \overline \xi > 0}\left\{  k \underline \xi^2 + (T-k) \overline \xi{}^2 : k \underline \xi + (T-k) \overline \xi = 1,~ \underline \xi^k \overline \xi{}^{T-k} = \overline{\gamma}\right\}.
\end{aligned}
\end{equation}
By Lemma~\ref{prop:aux-f-optimal} below, the optimal value of~\eqref{eq:aux3} is given by $f_{T,1}(\overline{\gamma})$. 
Hence, if we replace the minimization problem in~\eqref{eq:f_t-decomposition1} with $f_{T,1}(\overline{\gamma})$, we obtain
\begin{equation*}
f_T(0, \overline{\gamma}) = \left\{ \begin{array}{ll}
\displaystyle \min_{\underline \xi \geq 0,\, \overline \xi \geq 0}\left\{  \underline \xi^2 + (T-1) \overline \xi{}^2 : \underline \xi + (T-1) \overline \xi = 1,~ \underline \xi \, \overline \xi{}^{T-1} = \overline{\gamma}\right\}  & \text{if } 0\leq \overline{\gamma} < T^{-T},\\[2ex]
\frac{1}{T} & \text{if } \overline{\gamma} \geq T^{-T}.
\end{array} \right.
\end{equation*}
The statement of the lemma now follows since the minimization problem in the equation above evaluates to $1 / T$ at $\overline{\gamma} = T^{-T}$. Indeed, the minimization problem is bounded below by $\min_{\Vert \bm{\xi} \Vert_1 = 1}  \Vert \bm{\xi} \Vert_2^2$, and the optimal value $1/T$ of this bound is achieved by the feasible solution $\underline{\xi} = \overline{\xi} = 1/T$ of the minimization problem at $\overline{\gamma} = T^{-T}$. \qed
\end{proof}

\begin{lemma}
\label{prop:aux-f-optimal}
For $T \geq 2$ and $0<\overline{\gamma}< T^{-T}$, the optimal value of~\eqref{eq:aux3} is given by $f_{T,1}(\overline{\gamma})$.
\end{lemma}
\begin{proof}
	The statement holds trivially true when $\lfloor \frac{T}{2} \rfloor = 1$, that is, for $T \in\{2,3\}$. Next, we show that $f_{4,1}(\overline{\gamma}) < f_{4,2}(\overline{\gamma})$ for any $\overline{\gamma}\in(0,4^{-4})$. This inequality not only implies that the statement holds true for $T = 4$ but will also be instrumental for proving the statement for $T >4$. 

Fix $\overline{\gamma}\in (0,4^{-4})$ and note that
\begin{equation*}
\begin{aligned}
	f_{4,2}(\overline{\gamma}) &= \min_{ \underline \xi  > 0,\overline \xi  > 0}  \left\{ 2 \underline \xi^2 + 2 \overline \xi{}^2 : 2 \underline \xi + 2 \overline \xi = 1, ~\underline \xi^2 \overline \xi{}^2 = \overline{\gamma}\right\} \\
	&=\frac{1}{2} \min_{ \underline \xi  > 0,\overline \xi  > 0}  \left\{  \underline \xi^2 +  \overline \xi{}^2 :  \underline \xi +  \overline \xi = 1, ~\underline \xi \overline \xi = 4\sqrt{\overline{\gamma}}\right\}\\
	&= \frac{1}{2} f_{2,1}(4\sqrt{\overline{\gamma}}) ~=~ \frac{1}{2} - 4\sqrt{\overline{\gamma}} ,
\end{aligned}
\end{equation*}
where the second equality follows from the substitution $\underline\xi\leftarrow 2\underline\xi$ and $\overline\xi\leftarrow 2\overline\xi$, and the last equality holds because $f_{2,1}(\overline{\gamma})=1-2\overline{\gamma}$ for any $\overline{\gamma}\in(0,2^{-2})$, which can be verified by direct calculation. Thus, we need to show that $f_{4,1}(\overline{\gamma})< \frac{1}{2} - 4\sqrt{\overline{\gamma}}$, where
\begin{align}
	f_{4,1}(\overline{\gamma}) &= \min_{ \underline \xi  > 0,\overline \xi  > 0}  \left\{  \underline \xi^2 + 3 \overline \xi{}^2 : \underline \xi + 3 \overline \xi = 1, ~\underline \xi  \overline \xi{}^3 = \overline{\gamma}\right\} \nonumber \\
	&= \min_{\overline \xi  > 0}  \left\{  (1-3\overline\xi)^2 + 3 \overline \xi{}^2 : (1-3\overline \xi)  \overline \xi{}^3 = \overline{\gamma}\right\}. \label{eq:f41}
\end{align}
It is therefore sufficient to find $\overline\xi{}^\star$ feasible in~\eqref{eq:f41} with
\begin{eqnarray*}
	(1-3\overline\xi{}^\star)^2 + 3 (\overline \xi{}^\star)^2< 1/2 - 4\sqrt{\overline{\gamma}} &  \iff &  12(\overline\xi{}^\star)^2 - 6\overline\xi{}^\star + (1/2 + 4\sqrt{\overline{\gamma}}) < 0\\
	& \iff &  \overline\xi{}^\star \in \left(\zeta^-, \zeta^+ \right),
\end{eqnarray*}
where $\zeta^\pm= (3\pm \sqrt{3 - 48\sqrt{\overline{\gamma}}})/12$ are the roots of $12(\overline\xi{}^\star)^2 - 6\overline\xi{}^\star + (1/2 + 4\sqrt{\overline{\gamma}})$. Equivalently, we should demonstrate the existence of some $\overline\xi{}^\star\in (\zeta^-, \zeta^+)$ with $(1-3\overline \xi{}^\star)  (\overline \xi{}^\star)^3 - \overline{\gamma}=0$. By the intermediate value theorem, this holds if
\begin{equation}
	\label{eq:zetapm}
	(1-3\zeta^-)  (\zeta^-)^3 - \overline{\gamma}>0\quad \text{and}\quad (1-3\zeta^+)  (\zeta^+)^3 - \overline{\gamma}<0.
\end{equation}
But these inequalities are automatically satisfied under the assumption that $\overline{\gamma} \in (0,4^{-4})$. Indeed, recalling the definition of $\zeta^-$ and defining $z^-=12\zeta^--3= - \sqrt{3 - 48\sqrt{\overline{\gamma}}}$, we have
\begin{eqnarray*}
	(1-3\zeta^-)  (\zeta^-)^3 - \overline{\gamma}= \left(1 - \frac{3+z^-}{4} \right) \left( \frac{3+z^-}{12}\right)^3 - \left( \frac{3-(z^-)^2}{48}\right)^2=  - \frac{1}{12^3} (z^-)^3(z^- + 2)>0,
\end{eqnarray*}
where the inequality holds because $z^- \in (-\sqrt{3},0)$ for $\overline{\gamma} \in (0,4^{-4})$. Similarly, defining $z^+=12\zeta^+-3= \sqrt{3 - 48\sqrt{\overline{\gamma}}}$, we can prove that $(1-3\zeta^+)  (\zeta^+)^3 - \overline{\gamma}<0$. Thus, we have shown that $f_{4,1}(\overline{\gamma}) < f_{4,2}(\overline{\gamma})$ for any $\overline{\gamma}\in(0,4^{-4})$, which establishes the assertion for $T=4$.

Fix now some $T \geq 5$ and assume for the sake of argument that there exist $k \in \lbrace 2, \hdots, \lfloor \frac{T}{2} \rfloor \rbrace$ and $\overline{\gamma}\in (0, T^{-T})$ with $f_T(0, \overline{\gamma}) = f_{T,k}(\overline{\gamma}) < f_{T,1} (\overline{\gamma})$. Hence, there are some $\underline \xi > 0$ and $\overline \xi> 0$ with $\underline\xi\neq \overline \xi$ such that the minimum of $f_T (0, \overline{\gamma})$ in~\eqref{opt:aux-fg} is attained by the solution $\xi_1=\cdots=\xi_k=\underline\xi$ and $\xi_{k+1}=\cdots =\xi_T=\overline \xi$. Fixing $\xi_1,\ldots,\xi_{k-2}$ and $\xi_{k+3},\ldots,\xi_T$ at their optimal values and optimizing only over the remaining four decision variables in $f_T (0, \overline{\gamma})$ yields
\begin{equation*}
\begin{array}{lcl}
f_T(0, \overline{\gamma}) ~= & \displaystyle \min_{\xi_{k-1},\xi_k,\xi_{k+1}, \xi_{k+2}\geq 0} & (k-2)\underline \xi^2 + (T-k-2)\overline \xi{}^2 +\sum_{t=k-1}^{k+2}\xi_t^2\\[2ex]
& \st & (k-2)\underline \xi + (T-k-2)\overline \xi+\sum_{t=k-1}^{k+2}\xi_t=1\\
& & \underline \xi ^{k-2}\overline \xi{}^{T-k-2}\,\prod_{t=k-1}^{k+2} \xi_t \leq \overline{\gamma}.
\end{array}
\end{equation*}
Defining the strictly positive constant $c=1- (k-2)\underline \xi - (T-k-2)\overline \xi=2\underline\xi+2\overline\xi$ and using the substitution $y_t\leftarrow \xi_{k-2+t}/c$ for $t=1,\ldots,4$ further yields
\begin{eqnarray}
f_T(0, \overline{\gamma}) &= &(k-2)\underline \xi^2 + (T-k-2)\overline \xi{}^2 +\nonumber \\ &&\quad \min_{y_1,y_2,y_3,y_4\geq 0} 
\textstyle \left\{ \sum_{t=1}^{4} c^2\, y_t^2: \sum_{t=1}^4 y_t=1,~ \prod_{t=1}^{4} y_t \leq \frac{\overline{\gamma}}{c^4\,\underline \xi ^{k-2}\,\overline \xi{}^{T-k-2}}\right\} \label{eq:min-ys}\\
&=& (k-2)\underline \xi^2 + (T-k-2)\overline \xi{}^2 + c^2 f_4\left( \frac{\overline{\gamma}}{c^4\,\underline \xi ^{k-2}\,\overline \xi{}^{T-k-2}} \right), \nonumber
\end{eqnarray}
where the second equality follows from the definition of~$f_4(0, \overline{\gamma})$ in~\eqref{opt:aux-fg}. By construction, the minimization problem in~\eqref{eq:min-ys} must be solved by $y_1=y_2=\underline\xi$ and $y_3=y_4=\overline\xi$. However, this contradicts our previous results. In fact, we know that the solution of $f_4 (0, \overline{\gamma})$ must have the following properties for $T=4$. If $\overline{\gamma}/[c^4\,\underline \xi ^{k-2}\,\overline \xi{}^{T-k-2}]<4^{-4}$, then three out of the four $\xi_t$ variables must be equal at optimality. Conversely, if $\overline{\gamma}/[c^4\,\underline \xi ^{k-2}\,\overline \xi{}^{T-k-2}]\geq 4^{-4}$, then all four $\xi_t$ variables must be equal. This contradicts our assumption that there exist $k \in \lbrace 2, \hdots, \lfloor \frac{T}{2} \rfloor \rbrace$ and $\overline{\gamma}\in (0, T^{-T})$ with $f_T(0, \overline{\gamma}) = f_{T,k}(\overline{\gamma}) < f_{T, 1} (\overline{\gamma})$. Thus, the assertion holds for all $T>4$. \qed
\end{proof}

We now show that in the worst case, the weak-sense geometric random walk $\bm{\tilde \pi}=\{\tilde\pi_T\}_{T\in\mathbb N}$ defined through $\tilde \pi_T=\prod_{t=1}^T \tilde\xi_t$ is absorbed at $0$ with certainty if $T$ exceeds a threshold $T_0$.

\begin{thm}[Certainty of Absorption]
\label{thm:rprob_eq0}
For $T > \frac{\mu^2 + \sigma^2}{(1 - \rho)\sigma^2} + 1$ we have $\text{\em L}(\gamma) = 1$ for every $\gamma> 0$.
%
\end{thm}
\begin{proof}
From the proof of Proposition~\ref{prop:non-emptiness} we know that there exists a discrete distribution $\mathbb P_0 = \sum_{k\in \mathcal K} p_k \delta_{\bm{\xi}^k}\in\mathcal P$ with scenarios $\bm\xi^k$ and associated probabilities $p_k>0$, where $k$ ranges over a finite index set $\mathcal K$ of cardinality $T+1$. By the permutation symmetry, any discrete distribution of the form $\mathbb P_0\in\mathcal P$ can be used to construct a corresponding symmetric distribution 
\begin{equation}
	\label{eq:symm-dist}
	\mathbb{P} =  \frac{1}{T!}\sum_{\pi \in \mathfrak P} \sum_{k \in \mathcal{K}} p_k \delta_{\mathbf{P}_\pi \bm{\xi}^k},
\end{equation}
which is also an element of $\mathcal P$. Here, $\mathfrak P$ denotes the group of all permutations of $\{1,\ldots, T\}$, while $\mathbf P_\pi\in\mathbb R^{T\times T}$ denotes the permutation matrix induced by $\pi\in\mathfrak P$; see also Lemma~\ref{prop:rprob_symmetry}. Next, we define $m_1^k=\frac{1}{T} \sum_{t=1}^T\xi_t^k$ and $m_2^k = \frac{1}{T} \sum_{t=1}^T(\xi_t^k)^2$ as the arithmetic and quadratic means of scenario $\bm{\xi}^k$, respectively. It turns out that the first two moments of $\bm{\tilde \xi}$ can be expressed in terms of $m_1^k$ and $m_2^k$. Note, for instance, that for any $t\neq s$ we have
\begin{equation*}
\begin{aligned}
\mathbb E_{\mathbb P}\left( \tilde\xi_t \tilde\xi_s \right)& =\frac{1}{T!}\sum_{\pi \in \mathfrak P} \sum_{k \in \mathcal{K}} p_k \, \xi^k_{\pi(t)} \xi^k_{\pi(s)}  = \sum_{k \in \mathcal{K}} \frac{p_k}{T!} \sum_{r=1}^T \xi^k_r \sum_{\pi \in \mathfrak P:\, \pi(s)=r} \xi^k_{\pi(t)}\\
& = \sum_{k \in \mathcal{K}} \frac{p_k}{T!} \sum_{r=1}^T \xi^k_r (T-2)! \left(T m_1^k-\xi^k_r\right) = \sum_{k \in \mathcal{K}} \frac{p_k}{T-1} \left(T (m_1^k )^2-m_2^k \right),
\end{aligned}
\end{equation*}
where the first equality follows from the definition of $\mathbb P$ and because the $t$-th component of $\mathbf{P}_\pi \bm{\xi}^{(k)}$ is given by $\xi^k_{\pi(t)}$, while the third equality holds because there are $(T-2)!$ permutations that map $s$ to $r$ and $t$ to any fixed index different from $r$. Similarly, one can show that
\begin{equation*}
\mathbb E_{\mathbb P}\left( \tilde\xi_t \right) = \sum_{k \in \mathcal{K}} p_k m_1^k \quad\mbox{and} \quad \mathbb E_{\mathbb P}\left( \tilde\xi_t^2 \right)= \sum_{k \in \mathcal{K}} p_k m_2^k.
\end{equation*}
The moment conditions in the definition of $\mathcal P$ thus reduce to
\begin{subequations}
\label{eq:mk}
\begin{align}
	\label{eq:mk-prob} &\sum_{k \in \mathcal{K}} p_k \hspace{-4.5cm}&&= 1 \\  \label{eq:mk-mean}&\sum_{k \in \mathcal{K}} p_k m_1^k \hspace{-4.5cm}&& = \mu  \\
	\label{eq:mk-variance} &\sum_{k \in \mathcal{K}} p_k m_2^k \hspace{-4.5cm}&& = \mu^2 + \sigma^2 \\ \label{eq:mk-covariance} &\sum_{k \in \mathcal{K}} \frac{p_k}{T-1} \left(T ( m_1^k)^2 - m_2^k \right) \hspace{-4.5cm}&&= \mu^2 + \rho\sigma^2.
\end{align}
\end{subequations}

In the following we will update the scenarios $\bm \xi^k$ of the distribution $\mathbb P$ iteratively in finitely many steps, always ensuring that $\mathbb P$ remains within $\mathcal P$ after each update. The terminal distribution will have the property that $\prod_{t=1}^T \xi_t^k=0$ for every $k\in\mathcal K$, which means that we will have constructed a distribution $\mathbb P\in\mathcal P$ with $\mathbb P(\prod_{t=1}^T \tilde \xi_t=0)=1$. This will establish the claim.

\paragraph{Step 1:}  Keeping the scenario probabilities as well as the scenario-wise arithmetic and quadratic means constant, we first replace each~$\bm \xi^k$ with a minimizer of the problem
\begin{equation}
	\label{opt:aux-xi0}
	\inf_{\bm\xi \geq \bm 0}~ \left\lbrace \prod_{t=1}^T \xi_t : \frac{1}{T}\sum_{t=1}^T \xi_t = m^k_1,\ \frac{1}{T}\sum_{t=1}^T \xi_t^2 = m^k_2 \right\rbrace,
\end{equation}
which depends parametrically on $m_1^k$ and $m_2^k$. By  Lemma~\ref{prop:minprod}~(i) below, problem~\eqref{opt:aux-xi0} is indeed solvable for every $k\in\mathcal K$. The new distribution with updated scenarios still belongs to $\mathcal P$ because we did not change $p_k$, $m_1^k$ and $m_2^k$, implying that the moment conditions~\eqref{eq:mk} remain valid. To gain a better understanding of the updated distribution, we define the disjoint index~sets
\begin{equation*}
\begin{aligned}
	\mathcal{K}^+ = \left\lbrace k \in \mathcal{K}: T\geq \frac{m_2^k}{( m_1^k)^2} \geq \frac{T}{T-1}  \right\rbrace \quad\text{and}\quad
	\mathcal{K}^- = \left\lbrace k \in \mathcal{K}: 1\leq \frac{m_2^k}{( m_1^k)^2}  < \frac{T}{T-1} \right\rbrace,
\end{aligned}
\end{equation*}
and note that $\mathcal K=\mathcal{K}^+ \cup \mathcal{K}^-$ by Lemma~\ref{prop:minprod}~(i) below. Lemma~\ref{prop:minprod}~(ii) further implies that 
\begin{subequations}
\begin{equation}
	\label{eq:K+}
	k\in \mathcal K^+  \iff T\geq \frac{m_2^k}{( m_1^k)^2} \geq \frac{T}{T-1}\iff \frac{1}{T}\leq \frac{m_2^k - ( m_1^k)^2 }{m_2^k} \leq \frac{T-1}{T} \iff \prod_{t=1}^T \xi_t^k=0
\end{equation}
and
\begin{equation}
	\label{eq:K-}
	k\in \mathcal K^- \iff 1\leq \frac{m_2^k}{( m_1^k)^2}  < \frac{T}{T-1}\iff 0\leq \frac{m_2^k - ( m_1^k)^2 }{m_2^k} < \frac{1}{T} \iff \prod_{t=1}^T \xi_t^k>0.
\end{equation}
\end{subequations}
We will henceforth say that $\mathcal K^+$ ($\mathcal K^-$) is the index set of the {\em absorbing} ({\em non-absorbing}) scenarios. If all scenarios are absorbing (that is, if $\mathcal K^+=\mathcal K$), then $\mathbb P(\prod_{t=1}^T \tilde \xi_t=0)=1$, and we are done.

\paragraph{Step 2:} If there exists a non-absorbing scenario $i\in \mathcal K^-$, we will alter both the scenarios and their quadratic means to make scenario $i$ absorbing, while ensuring that all scenarios $k\in\mathcal K^+$ remain absorbing. To achieve this, we consider the following family of quadratic means parameterized in $\lambda\in [0,1]$.
\begin{align} 
\label{eq:mk-lambda}
m_2^k(\lambda)=\left\{ \begin{array}{ll} 
	(1-\lambda) m_2^k+\lambda \frac{T}{T-1} (m_1^k)^2 & \text{for } k\in\mathcal K^+\\
	m_2^i + \lambda \sum_{k \in \mathcal{K}^+} \frac{p_k}{p_i} (m_2^k - \frac{T}{T-1}(m_1^k)^2) & \text{for } k=i\\
	m_2^k & \text{for } k\in\mathcal K^-\backslash\{i\}
	\end{array}\right.
\end{align}
By construction, $p_k$, $m_1^k$ and $m_2^k=m_2^k(\lambda)$ satisfy the moment conditions~\eqref{eq:mk} for every $\lambda\in [0,1]$. As in Step~1, the scenario $\bm\xi^k(\lambda)$ is then chosen to be a minimizer of problem~\eqref{opt:aux-xi0} with inputs $m_1^k$ and $m_2^k=m_2^k(\lambda)$. However, \eqref{opt:aux-xi0} could fail to be solvable for $\lambda\lesssim 1$, in which case the proposed construction would fail. Indeed, Lemma~\ref{prop:minprod}~(i) shows that \eqref{opt:aux-xi0} is only solvable when $1\leq m_2^k(\lambda)/(m_1^k)^2\leq T$. In the remainder we will demonstrate that there is $\lambda^\star\in(0,1)$ such that $\bm\xi^k(\lambda^\star)$ exists for every $k\in\mathcal K$ and such that all scenarios $k\in\mathcal K^+\cup\{i\}$ are absorbing.

Subtracting \eqref{eq:mk-covariance} from \eqref{eq:mk-variance} and dividing the difference by \eqref{eq:mk-variance} yields
\begin{equation*}
	\frac{T \sum_{k \in \mathcal{K}} p_k \left( m_2^k - ( m_1^k)^2 \right)}{(T-1)\sum_{k \in \mathcal{K}} p_k m_2^k} = 
	\frac{(1 - \rho)\sigma^2}{\mu^2 + \sigma^2} > \frac{1}{T-1},
\end{equation*}
where the inequality follows from the assumption that $T > \frac{\mu^2 + \sigma^2}{(1 - \rho)\sigma^2} + 1$. Multiplying both sides of the inequality by $\frac{T-1}{T}$ and partitioning $\mathcal{K}$ into $\mathcal{K}^+$ and $\mathcal{K}^-$ further reveals that
\begin{equation}
	\label{eq:fraction-ineq}
	\frac{\sum_{k \in \mathcal{K}^+} p_k \,m_2^k\, \frac{ m_2^k - ( m_1^k)^2}{m_2^k} + \sum_{k \in \mathcal{K}^-} p_k \,m_2^k\,\frac{m_2^k - ( m_1^k)^2}{m_2^k}}{\sum_{k \in \mathcal{K}^+} p_k m_2^k + \sum_{k \in \mathcal{K}^-} p_k m_2^k} > \frac{1}{T}.
\end{equation}
The expression on the left hand side of the above inequality represents a weighted average of the fractions $(m_2^k - ( m_1^k)^2)/m_2^k$ across all $k\in\mathcal K$. Recall from~\eqref{eq:K+} and~\eqref{eq:K-} that the fractions indexed by $k\in\mathcal K^+$ are larger or equal to $1/T$, while those indexed by $k\in\mathcal K^-$ are strictly smaller than $1/T$. The inequality~\eqref{eq:fraction-ineq} asserts that the fractions corresponding to $k\in\mathcal K^+$ dominate those corresponding to $k\in\mathcal K^-$. Thus,~\eqref{eq:fraction-ineq} remains valid if we replace $\mathcal K^-$ with $\{i\}$, that is,

\begin{equation*}
	\frac{\sum_{k \in \mathcal{K}^+} p_k\, m_2^k\, \frac{m_2^k - ( m_1^k)^2 }{m_2^k} + p_i \,m_2^i\, \frac{m_2^i - ( m_1^i)^2 }{m_2^i}}{\sum_{k \in \mathcal{K}^+} p_k m_2^k+ p_i m_2^i} > \frac{1}{T},
\end{equation*}
which is equivalent to
\begin{equation}
	\label{eq:fraction-ineq2}
	\frac{\sum_{k \in \mathcal{K}^+} p_k \frac{( m_1^k)^2}{T-1}  + p_i \left( m_2^i + \sum_{k \in \mathcal{K}^+} \frac{p_k}{p_i} \left( m_2^k - \frac{T}{T-1}(m_1^k)^2 \right)- ( m_1^i)^2  \right)}{\sum_{k \in \mathcal{K}^+} p_k \frac{T}{T-1} (m_1^k)^2 + p_i \left( m_2^i + \sum_{k \in \mathcal{K}^+} \frac{p_k}{p_i} \left( m_2^k - \frac{T}{T-1}(m_1^k)^2 \right) \right)} > \frac{1}{T}.
\end{equation}
Using the notation introduced in~\eqref{eq:mk-lambda}, the inequality~\eqref{eq:fraction-ineq2} can be reformulated as
\begin{equation*}
	\frac{\sum_{k \in \mathcal{K}^+} p_k\, m_2^k(1)\, \frac{m_2^k(1) - (m_1^k)^2 }{m_2^k(1)} + p_i \,m_2^i(1)\, \frac{m_2^i(1) - (m_1^i)^2}{m_2^i(1)}}{\sum_{k \in \mathcal{K}^+} p_k m_2^k(1)+ p_i m_2^i(1)} > \frac{1}{T},
\end{equation*}
which constitutes a weighted average of the fractions $(m_2^k(1) - (m_1^k)^2)/m_2^k(1)$ across all $k\in\mathcal K^+\cup\{i\}$. By construction, we have $(m_2^k(1) - (m_1^k)^2)/m_2^k(1)=\frac{1}{T}$ for every $k\in\mathcal K^+$, and thus the average on the left hand side of the above inequality can exceed $\frac{1}{T}$ only if 
\begin{equation*}
	\frac{m_2^i(1) - (m_1^i)^2}{m_2^i(1)}>\frac{1}{T}.
\end{equation*}
As $i\in\mathcal K^-$, the relation~\eqref{eq:K-} further implies that
\begin{equation*}
	\frac{m_2^i(0) - (m_1^i)^2}{m_2^i(0)}=\frac{m_2^i - (m_1^i)^2}{m_2^i}<\frac{1}{T}.
\end{equation*}
The intermediate value theorem then guarantees the existence of $\lambda^\star\in (0,1)$ with 
\begin{equation*}
	\frac{m_2^i(\lambda^\star) - (m_1^i)^2}{m_2^i(\lambda^\star)}=\frac{1}{T} \quad\iff\quad \frac{m_2^i(\lambda^\star)}{(m_1^i)^2}=\frac{T}{T-1}.
\end{equation*}
By construction, we thus have $1\leq m_2^k(\lambda^\star)/(m_1^k)^2\leq T$ for every $k\in\mathcal K$, which implies via Lemma~\ref{prop:minprod}~(i) that the corresponding scenarios $\bm\xi^k(\lambda^\star)$ are well-defined. Our construction also guarantees that $\frac{T}{T-1}\leq m_2^k(\lambda^\star)/(m_1^k)^2\leq T$ for every $k\in\mathcal K^+\cup\{i\}$, which implies via Lemma~\ref{prop:minprod}~(ii) that the corresponding scenarios $\bm\xi^k(\lambda^\star)$ are absorbing. Thus, by replacing $\bm\xi^k$ with $\bm\xi^k(\lambda^\star)$ in \eqref{eq:symm-dist} we obtain a new distribution $\mathbb P\in\mathcal P$ with more absorbing scenarios. As the total number of scenarios is finite, we can repeat Step~2 finitely many times to construct a distribution $\mathbb P\in\mathcal P$ that has only absorbing scenarios. Thus, the claim follows. \qed
\end{proof}

The proof of Theorem~\ref{thm:rprob_eq0} relies on the following auxiliary result.

\begin{lemma}
\label{prop:minprod}
Assume that $m_1,m_2>0$ and consider the parametric program
\begin{equation}
\label{opt:aux-xi}
\begin{aligned}
	\inf_{\bm\xi \geq \bm 0}~ \left\lbrace \prod_{t=1}^T \xi_t : \frac{1}{T}\sum_{t=1}^T \xi_t = m_1,\ \frac{1}{T}\sum_{t=1}^T \xi_t^2 = m_2 \right\rbrace.
\end{aligned}
\end{equation}
Then, the following statements hold:
\begin{enumerate}
	\item[(i)]
	Problem~\eqref{opt:aux-xi} is feasible and solvable iff $T \geq \frac{m_2}{m_1^2} \geq 1$.
	\item[(ii)]
	The optimal value of \eqref{opt:aux-xi} is zero iff $T \geq \frac{m_2}{m_1^2} \geq \frac{T}{T-1}$.
\end{enumerate}
\end{lemma}

\begin{figure}[tb]
 	\centering
	\includegraphics[width=0.37\paperwidth]{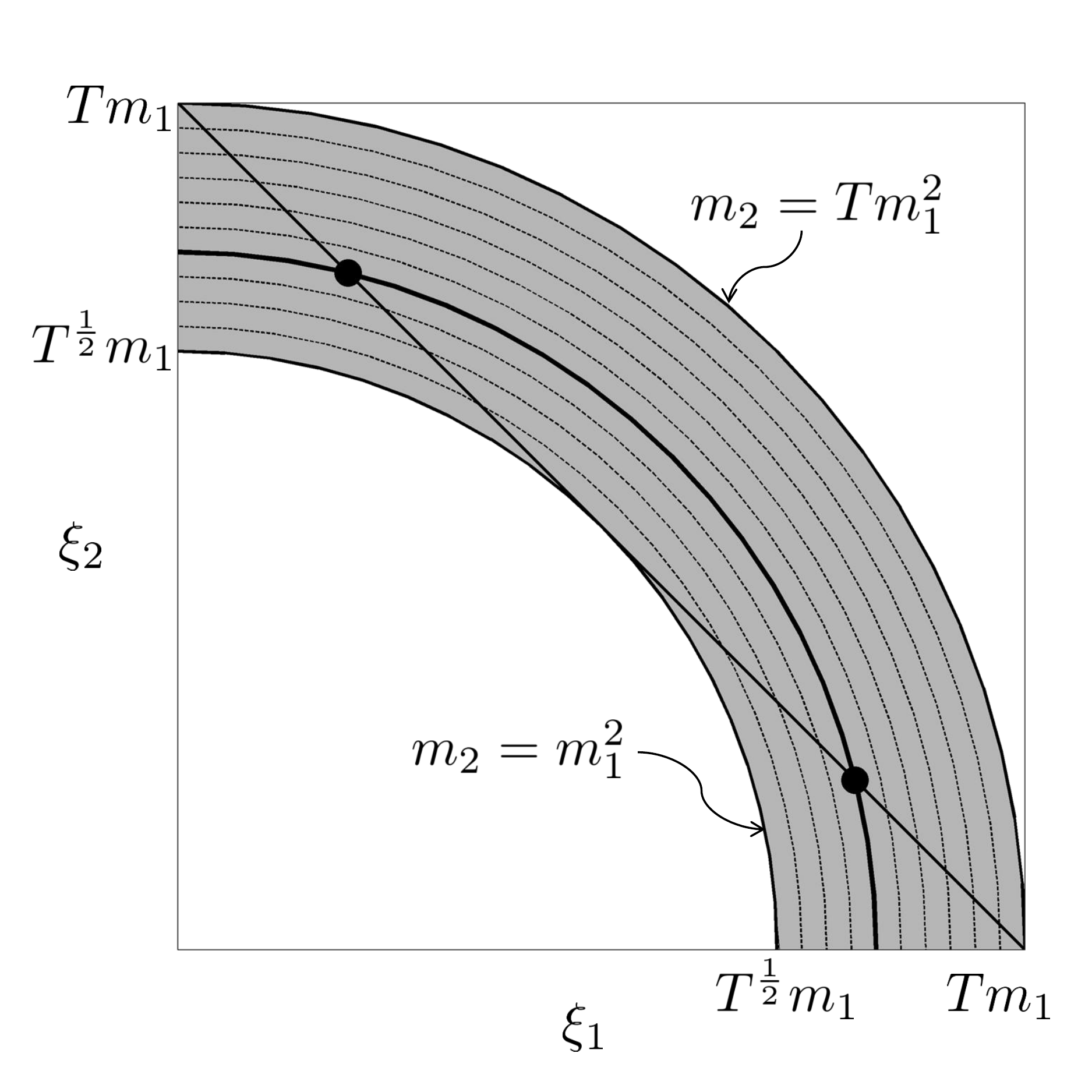}
	\caption{Feasible region of problem~\eqref{opt:aux-xi} for $T = 2$ and different values of $m_1$ and $m_2$. The diagonal line corresponds to the constraint $\frac{1}{T}\sum_{t=1}^T \xi_t = m_1$, and each dotted curve corresponds to the constraint $\frac{1}{T}\sum_{t=1}^T \xi_t^2 = m_2$ for some combination of $m_1$ and $m_2$. The innermost and the outermost curves correspond to the cases where $m_2/m_1^2 = 1$ and $m_2/m_1^2 = T$, respectively. The feasible region for the $(m_1, m_2)$-combination represented by the bold curve is given by the two dots.}
	\label{fig:minprod}
\end{figure}

\noindent Figure~\ref{fig:minprod} visualizes how the feasible set of problem~\eqref{opt:aux-xi} depends on $m_1$ and $m_2$.

\begin{proof}[Proof of Lemma~\ref{prop:minprod}:]
As for assertion~(i), assume that there is $\bm \xi$ feasible in~\eqref{opt:aux-xi}. We then have $\frac{1}{T}\sum_{t=1}^T \xi_t = m_1$, which implies that $Tm_1^2 \geq m_2 \geq m_1^2$ since $\left \Vert \bm{\xi} \right \rVert_1 \geq \left \Vert \bm{\xi} \right \rVert_2 \geq \frac{1}{\sqrt{T}} \left \Vert \bm{\xi} \right \rVert_1$. Conversely, if $T \geq \frac{m_2}{m_1^2} \geq 1$, we may define $\bm{\xi} = (z, \frac{m_1T - z}{T-1}, \hdots, \frac{m_1T - z}{T-1})$ for some $z\in [m_1, T m_1]$ to be chosen later. By construction, we have $\frac{1}{T}\sum_{t=1}^T \xi_t = m_1$ irrespective of $z$, while
	\begin{equation*}
	\begin{aligned}
		\frac{1}{T} \sum_{t=1}^T \xi_t^2 
		~=~  \frac{z^2}{T} + \frac{T-1}{T}\left( \frac{m_1T - z}{T-1}\right)^2
	\end{aligned}
	\end{equation*}
changes continuously from $m_1^2$ to $Tm_1^2$ when $z$ is swept from $m_1$ to $Tm_1$. Thus, by the intermediate value theorem, we may assume that $\frac{1}{T} \sum_{t=1}^T \xi_t^2=m_2\in [m_1^2, Tm_1^2]$ for some suitably chosen $z\in [m_1,Tm_1]$. We conclude that~\eqref{opt:aux-xi} is feasible whenever $T \geq \frac{m_2}{m_1^2} \geq 1$. In that case, however, \eqref{opt:aux-xi} is also solvable as the objective function is continuous and the feasible set is compact.

To prove assertion~(ii), we observe that the optimal value of~\eqref{opt:aux-xi} vanishes iff the problem admits a minimizer 
$\bm \xi$ with $\prod_{t=1}^T \xi_t = 0$. More precisely, by permutation symmetry, the minimum of~\eqref{opt:aux-xi} vanishes iff there exists $\bm \xi$ with $\xi_T=0$, $\frac{1}{T} \sum_{t=1}^{T-1} \xi_t = m_1$ and $\frac{1}{T} \sum_{t=1}^{T-1} \xi^2_t = m_2$. By assertion~(i), however, the last two inequalities are satisfiable iff 
\[
	T - 1 \geq \frac{m_2 \left( \frac{T}{T-1} \right)}{\left( m_1\left( \frac{T}{T-1} \right) \right)^2} \geq 1 \quad \iff\quad T \geq \frac{m_2}{m_1^2} \geq \frac{T}{T-1},
\]
and thus the claim follows.
\qed
\end{proof}

\section{Right-Sided Chebyshev Bounds}
\label{section:left-prob}

We now study \emph{right-sided Chebyshev bounds} of the form
\begin{equation*}
\text{R}(\gamma) = \sup_{\mathbb{P} \in \mathcal{P}} \mathbb{P} \left( \prod_{t=1}^T \tilde{\xi}_t \geq \gamma \right),
\end{equation*}
where the ambiguity set $\mathcal{P}$ is defined in~\eqref{eq:our_ambiguity_set}. We first present the main result of this section.

\begin{thm}[Right-Sided Chebyshev Bound]
\label{thm:lprob}
Let $\gamma > 0$. For all $T \geq 3$ the right-sided Chebyshev bound $\text{\em R} (\gamma)$ coincides with the optimal objective value of the semidefinite program
\begin{equation}
\label{eq:lprob0}
\begin{aligned}
	&\inf && \alpha + T\mu\beta + 
			 T(\mu^2 + \sigma^2)\gamma_1 +
			 T(T\mu^2 + \sigma^2 + (T-1)\rho\sigma^2)\gamma_2 \\
	&\st  && \alpha, \beta, \gamma_1, \gamma_2 \in \mathbb{R},~
	         \lambda_1,\lambda_2,\lambda_3 \geq 0,~
	         \bm p \in \mathbb{R}^{2T+1}, ~\bm P \in \mathbb{S}^{T+1}_+,~
	         \bm q \in \mathbb{R}^{2T-1}, ~\bm Q \in \mathbb{S}^{T}_+ \\
	&     && \alpha \geq 0, \quad
	         \alpha \geq 1 - \lambda_3 T \gamma^{1/T}, \quad
	         \gamma_1 + T\gamma_2 \geq 0, \quad
	         \gamma_1 + \gamma_2 \geq 0 \\
	&	  && \gamma_2 + \frac{\gamma_1}{T} + \alpha \geq \left\Vert \left( \beta - \lambda_1, \gamma_2 + \frac{\gamma_1}{T} - \alpha \right)\right\Vert_2 \\
	&     && \gamma_2 + \gamma_1 + \alpha \geq \left\Vert \left( \beta - \lambda_2, \gamma_2 + \gamma_1 - \alpha \right) \right\Vert_2 \\
	&     && \gamma_2 + \frac{\gamma_1}{T} + \lambda_3 T \gamma^{1/T} + \alpha - 1 \geq \left\Vert \left( \beta - \lambda_3, \gamma_2 + \frac{\gamma_1}{T} - \lambda_3 T \gamma^{1/T} - \alpha + 1 \right) \right\Vert_2 \\
	&     && p_0 = (T-1)\gamma_1\gamma^{\frac{2}{T-1}} + (T-1)^2\gamma_2\gamma^{\frac{2}{T-1}}, \quad  p_1+q_0 =(T-1)\beta\gamma^{\frac{1}{T-1}} \\
	&     && p_2+q_1= \alpha-1, \quad p_T+q_{T-1}=  2(T-1)\gamma_2\gamma^{\frac{1}{T-1}} ,\quad p_{T+1}+q_T= \beta\\ 
	&     && p_{2T}= \gamma_1+\gamma_2,\quad p_t + q_{t-1} = 0 \quad \forall t = 3, \hdots, T-1,T+2,\ldots , 2T-1 \\
	&     && p_t = \textstyle\sum_{i+j = t} P_{i,j} \quad \forall t = 0, \hdots, 2T, \quad
			 q_t = \textstyle\sum_{i+j = t} Q_{i,j} \quad \forall t = 0, \hdots, 2T-2,
\end{aligned}
\end{equation}
where we use the convention that the entries of $\bm p$, $\bm P$, $\bm q$ and $\bm Q$ are numbered starting from~$0$. For $T=2$, $\text{\em R} (\gamma)$ is given by a variant of \eqref{eq:lprob0} where the constraints $p_2+q_1= \alpha-1$ and $p_T+q_{T-1}=  2(T-1)\gamma_2\gamma^{\frac{1}{T-1}}$ are combined to $p_2+q_1= \alpha-1+2(T-1)\gamma_2\gamma^{\frac{1}{T-1}}$.
\end{thm}
\begin{proof}
Using similar arguments as in the proof of Theorem~\ref{thm:rprob}, one first shows that the worst-case probability problem $\sup_{\mathbb{P} \in \mathcal{P}} \mathbb{P} ( \prod_{t=1}^T \tilde{\xi}_t \geq \gamma)$ admits a strong dual which constitutes a semi-infinite optimization problem. Exploiting this problem's permutation symmetry, one can further show that its optimal value amounts to
\begin{equation}
\label{opt:lprob_dual_simpl}
\begin{array}{rl}
	\text{R}(\gamma) =\
	\inf & \alpha + T\mu\beta +	T(\mu^2 + \sigma^2)\gamma_1 + T\left[ T\mu^2 + \sigma^2 + (T-1)\rho\sigma^2\right] \gamma_2 \\
	\st  & \alpha, \beta, \gamma_1, \gamma_2 \in \mathbb{R} \\
	& \alpha + \beta \Vert \bm{\xi} \Vert_1 + \gamma_1 \Vert \bm{\xi} \Vert_2^2 + \gamma_2 \Vert \bm{\xi} \Vert_1^2 \geq 0 \quad \forall \bm{\xi} \geq \bm{0} \\
	& \alpha + \beta \Vert \bm{\xi} \Vert_1 + \gamma_1 \Vert \bm{\xi} \Vert_2^2 + \gamma_2 \Vert \bm{\xi} \Vert_1^2 \geq 1
\quad \forall \bm{\xi} \geq \bm{0}:\ \textstyle\prod_{t=1}^{T} \xi_t \geq \gamma.
	\end{array}
\end{equation}
Details are omitted for brevity of exposition.
Lemma~\ref{lem:semi_infinite_constraints} then implies that~(\ref{opt:lprob_dual_simpl}) reduces to
\begin{equation}
\label{opt:lprob_dual_simpl2}
\begin{array}{rl}
	\text{R}(\gamma) =\
	\inf  & \alpha + T\mu\beta +	T(\mu^2 + \sigma^2)\gamma_1 + T(T\mu^2 + \sigma^2 + (T-1)\rho\sigma^2)\gamma_2 \\
	\st   & \alpha, \beta, \gamma_1, \gamma_2 \in \mathbb{R} \\
	& \displaystyle \inf_{s \geq 0}\ \alpha + \beta s + \gamma_2 s^2 + \frac{\gamma_1}{T}s^2 \geq 0 \\
	& \displaystyle \inf_{s \geq 0}\ \alpha + \beta s + \gamma_2 s^2 + \gamma_1 s^2 \geq 0 \\
	& \displaystyle \inf_{s \geq T\gamma^{1/T}}\ \alpha + \beta s + \gamma_2 s^2 + \frac{\gamma_1}{T} s^2 \geq 1 \\
	& \displaystyle \inf_{s \geq T\gamma^{1/T}}\ \alpha + \beta s + \gamma_2 s^2 + \gamma_1 s^2 g_T\left( \frac{\gamma}{s^T}, \infty \right) \geq 1.
\end{array}
\end{equation}
By leveraging the $\mathcal S$-lemma and a well-known reformulation of hyperbolic constraints as second-order cone constraints, one can use similar arguments as in the proof of Theorem~\ref{thm:rprob} to show that the first three constraints in~\eqref{opt:lprob_dual_simpl2} hold iff there exist $\lambda_1, \lambda_2, \lambda_3 \geq 0$ satisfying
\begin{align*}
	&\alpha \geq 0, \quad
	         \alpha \geq 1 - \lambda_3 T \gamma^{1/T}, \quad
	         \gamma_1 + T\gamma_2 \geq 0, \quad
	         \gamma_1 + \gamma_2 \geq 0 \\
	&\gamma_2 + \frac{\gamma_1}{T} + \alpha \geq \left\Vert \left( \beta - \lambda_1, \gamma_2 + \frac{\gamma_1}{T} - \alpha \right)\right\Vert_2 \\
	&\gamma_2 + \gamma_1 + \alpha \geq \left\Vert \left( \beta - \lambda_2, \gamma_2 + \gamma_1 - \alpha \right) \right\Vert_2 \\
	&\gamma_2 + \frac{\gamma_1}{T} + \lambda_3 T \gamma^{1/T} + \alpha - 1 \geq \left\Vert \left( \beta - \lambda_3, \gamma_2 + \frac{\gamma_1}{T} - \lambda_3 T \gamma^{1/T} - \alpha + 1 \right) \right\Vert_2.
\end{align*}
By Lemma~\ref{prop:g_t-decomposition} below, the last semi-infinite constraint in~\eqref{opt:lprob_dual_simpl2} can be re-expressed as
\begin{equation*}
	\inf_{s \geq T\gamma^{1/T}, \,\underline\xi, \overline\xi \geq 0} \left\{ \alpha + \beta s + \gamma_2 s^2 + \gamma_1 \left[\underline\xi{}^2 + (T-1)\overline\xi{}^2\right] : \underline\xi + (T-1)\overline\xi = s,~ \underline\xi \,\overline\xi{}^{T-1} = \gamma\right\}\geq 1,
\end{equation*}
which is identical to~\eqref{eq:rprob_dual_3_5}. The claim then follows by replacing this constraint with its explicit semidefinite reformulation familiar from Theorem~\ref{thm:rprob}. \qed
\end{proof}

The proof of Theorem~\ref{thm:lprob} relies on 2 auxiliary lemmas, which we prove next.

\begin{lemma}
\label{prop:g_t-decomposition}
For $T \geq 2$, $\overline{\gamma} =\infty$ and $\underline{\gamma} \geq 0$, the optimal value $g_T (\underline{\gamma}, \infty)$ of~\eqref{opt:aux-g} equals
\begin{equation*}
	g_T(\underline{\gamma}, \infty) = \left\{ \begin{array}{ll} 
	\displaystyle \max_{\underline \xi \geq 0,\, \overline \xi \geq 0} \left\{  \underline \xi^2 + (T-1) \overline \xi{}^2 : \underline \xi + (T-1) \overline \xi = 1,~ \underline \xi \, \overline \xi{}^{T-1} = \underline{\gamma} \right\} & \text{if } 0 \leq \underline{\gamma} \leq \underline{\gamma} T^{-T}\\
	-\infty & \text{if } \underline{\gamma} > T^{-T}.
	\end{array} \right.
\end{equation*}
\end{lemma}
\begin{proof}
If $\underline{\gamma} > T^{-T}$, then the maximization problem~\eqref{opt:aux-g} is infeasible due to the inequality of arithmetic and geometric means, and thus we have $g_T (\underline{\gamma}, \infty) = -\infty$. For $\underline{\gamma} = T^{-T}$, the unique feasible solution of~\eqref{opt:aux-g} is $\bm{\xi} = \frac{1}{T}\bm 1$, which implies that $g_T(\underline{\gamma}, \infty) = \frac{1}{T}$. Moreover, for $\underline{\gamma} = 0$, the last constraint in~\eqref{opt:aux-g} becomes redundant. In this case $g_T (\underline{\gamma}, \infty)$ is optimized by $\bm{\xi} = \mathbf{e}_i$, and thus we find $g_T(\underline{\gamma}, \infty) =1$. Lastly, for $0 < \underline{\gamma} < T^{-T}$, the maximization problem~\eqref{opt:aux-g} is feasible, and every feasible solution has strictly positive components. In addition, the product constraint $\prod_{t=1}^T \xi_t\geq \underline{\gamma}$ is binding at optimality for otherwise convex combinations of the optimal solution $\bm{\xi}$ with $\mathbf{e}_i$, where $i \in \arg \max \{ \bm{\xi}_j \, : \, j = 1, \ldots, T \}$, would improve the objective function of~\eqref{opt:aux-g}, which is a contradiction. We thus conclude that
\begin{equation*}
g_T(\underline{\gamma}, \infty) = \left\{ \begin{array}{ll}
	1 & \text{if } \underline{\gamma} =0,\\
	\max_{\bm \xi > \bm 0}  \left\{\Vert \bm{\xi} \Vert_2^2 :  \Vert \bm{\xi} \Vert_1 = 1,~ \textstyle \prod_{t=1}^T \xi_t = \underline{\gamma} \right\} & \text{if } 0< \underline{\gamma} <T^{-T},\\
	\frac{1}{T} & \text{if } \underline{\gamma} = T^{-T}, \\
	-\infty & \text{if } \underline{\gamma} > T^{-T}.
\end{array} \right.
\end{equation*}
As in the proof of Lemma~\ref{prop:f_t-decomposition}, for $0< \underline{\gamma} <T^{-T}$ one can use the optimality conditions of~\eqref{opt:aux-g} to show that
\begin{equation}
	\label{eq:aux-g2}
	g_T(\underline{\gamma}, \infty) =\max_{k \in\{ 1, \ldots, \lfloor \frac{T}{2}\rfloor\}} g_{T,k}(\underline{\gamma}),
\end{equation} 
%
where the functions $g_{T,k}:(0, T^{-T}) \rightarrow \mathbb{R}$, $k = 1, 2, \hdots, \lfloor \frac{T}{2} \rfloor$, are defined through
\begin{equation}
\label{opt:aux-g-enum}
\begin{aligned}
	g_{T,k}(\underline{\gamma}) = \max_{\underline \xi > 0,\, \overline \xi > 0}\left\{  k \underline \xi^2 + (T-k) \overline \xi{}^2 : k \underline \xi + (T-k) \overline \xi = 1,~ \underline \xi^k \overline \xi{}^{T-k} = \underline{\gamma}\right\}.
\end{aligned}
\end{equation}
Lemma~\ref{prop:aux-g-optimal} below asserts that the maximum in~\eqref{eq:aux-g2} is attained at $k=1$. We thus obtain
\begin{equation*}
g_T(\underline{\gamma}, \infty) = \left\{ \begin{array}{ll}
	1 & \text{if } \underline{\gamma} =0,\\
	\max_{\underline \xi \geq 0,\, \overline \xi \geq 0} \left\{  \underline \xi^2 + (T-1) \overline \xi{}^2 : \underline \xi + (T-1) \overline \xi = 1,~ \underline \xi \, \overline \xi{}^{T-1} = \underline{\gamma} \right\} & \text{if } 0< \underline{\gamma} <T^{-T},\\
	\frac{1}{T} & \text{if } \underline{\gamma} = T^{-T}, \\
	-\infty & \text{if } \underline{\gamma} > T^{-T}.
\end{array} \right.
\end{equation*}
The statement of the lemma now follows since the maximization problem in the equation above evaluates to $1$ at $\underline{\gamma} = 0$ and to $1 / T$ at $\underline{\gamma} = T^{-T}$. Indeed, the maximization problem is bounded above by $\max_{\Vert \bm{\xi} \Vert_1 = 1}  \Vert \bm{\xi} \Vert_2^2$, and the optimal value $1$ of this bound is achieved by the feasible solution $(\underline{\xi}, \overline{\xi}) = (1, 0)$ of the maximization problem at $\underline{\gamma} = 0$. Likewise, $g_T (T^{-T}, \infty)$ is bounded above by $\max_{\bm\xi \geq \bm 0} \{ \Vert \bm{\xi} \Vert_2^2 \, : \, \Vert \bm\xi \Vert_1 = 1,\, \prod_t \xi_t = T^{-T} \}$, and the optimal value $1/T$ of this bound is achieved by the feasible solution $(\underline{\xi}, \overline{\xi}) = (\frac{1}{T}, \frac{1}{T})$ of the maximization problem.
\qed
\end{proof}

\begin{lemma}
\label{prop:aux-g-optimal}
For $T \geq 2$ and $0 < \underline{\gamma} < T^{-T}$, the optimal value of~\eqref{eq:aux-g2} is given by $g_{T,1}(\underline{\gamma})$.
\end{lemma}
\begin{proof}
	The proof widely parallels that of Lemma~\ref{prop:aux-f-optimal} and is therefore omitted. 
%
%
\qed
\end{proof}

We now show that in the extreme case, the weak-sense geometric random walk $\bm{\tilde \pi}=\{\tilde\pi_T\}_{T\in\mathbb N}$ defined through $\tilde \pi_T=\prod_{t=1}^T \tilde\xi_t$ weakly exceeds the deterministic growth process $\{ \mu^T \}_{T\in\mathbb N}$ with certainty for any time horizon $T$, assuming that $\rho \geq 0$. The result can be viewed as the right-sided analogue of Theorem~\ref{thm:rprob_eq0}.

\begin{prop}\label{proposition_xy}
If $\rho \geq 0$, then $\text{\em R}(\gamma) = 1$ for all $\gamma \leq \mu^T$.
\end{prop}
\begin{proof}
The objective function of problem~\eqref{opt:lprob_dual_simpl2} can be reformulated as
\begin{equation*}
	(\alpha + T\mu\beta + T\mu^2\gamma_1 + T^2\mu^2\gamma_2) + T\sigma^2(\gamma_1 + (1+(T-1)\rho)\gamma_2).
\end{equation*}
For $\gamma \leq \mu^T$, the first term equals the left hand side of the third semi-infinite constraint in~\eqref{opt:lprob_dual_simpl2} if we set $s = T \mu$, and it must therefore be greater than or equal to 1. In the second term, the factor $(\gamma_1 + (1+(T-1)\rho)\gamma_2)$ can be expressed as the linear combination $\rho \cdot (\gamma_1 + T\gamma_2) + (1 - \rho) \cdot (\gamma_1 + \gamma_2)$. For $\rho \geq 0$, this linear combination becomes a convex combination, and the claim follows since $\gamma_1 + T \gamma_2 \geq 0$ and $\gamma_1 + \gamma_2 \geq 0$ are explicit constraints in the equivalent reformulation~\eqref{eq:lprob0}.
\qed
\end{proof}

We highlight that Proposition~\ref{proposition_xy} breaks down for $\rho < 0$.

\section{Covariance Bounds}
\label{section:cov_bounds}

The ambiguity set $\mathcal{P}$ reflects the assumption that the covariance matrix $\bm{\Sigma}$ is known \emph{precisely} and that the (co-)variances of the components of $\bm{\tilde{\xi}}$ are permutation symmetric. Either assumption may prove overly restrictive in practice. In this section, we therefore assume that only an upper bound on the covariance matrix is available. More precisely, we consider the ambiguity set
\begin{equation*}
\begin{aligned}
	\mathcal{P}' = 
	\left \lbrace 
		\mathbb{P} \in \mathcal M_+(\mathbb R^T_+)~: 
		\begin{array}{l}
			\mathbb{P} \left( \tilde{\bm{\xi}} \geq \bm 0 \right) = 1 ,\; \mathbb{E}_\mathbb{P} \left( \tilde{\bm{\xi}} \right) = \bm \mu ,\;
			\mathbb{E}_\mathbb{P} \left( \tilde{\bm{\xi}}\tilde{\bm{\xi}}^\intercal \right) \preceq \bm \Sigma+\bm \mu\bm\mu^\intercal  
		\end{array}
	\right \rbrace,
\end{aligned}
\end{equation*}
where $\bm \mu$ and $\bm\Sigma$ are defined as in Section~\ref{section:introduction}. For $\gamma > 0$, we are then interested in quantifying \emph{relaxed left-sided and right-sided Chebyshev bounds} of the form
\begin{equation*}
\begin{aligned}
	\text{L}' (\gamma) = \sup_{\mathbb{P} \in \mathcal{P}'} \mathbb{P} \left( \prod_{t=1}^T \tilde{\xi}_t \leq \gamma \right)\quad \text{and}\quad
	\text{R}'(\gamma) = \sup_{\mathbb{P} \in \mathcal{P}'} \mathbb{P} \left( \prod_{t=1}^T \tilde{\xi}_t \geq \gamma \right).
\end{aligned}
\end{equation*}

In the following, we analyze each of these relaxed bounds in turn.

\begin{thm}[Relaxed Left-Sided Chebyshev Bound]
\label{thm:rprob_bnd}
The relaxed left-sided Chebyshev bound satisfies $\text{\em L}'(\gamma) = \text{\em L}(\gamma)$ for all $\gamma > 0$.
\end{thm}
\begin{proof}

By repeating the first few steps of the proof of Theorem~\ref{thm:rprob}, one can show that $\text{L}'(\gamma)$ coincides with the optimal value of~\eqref{opt:rprob_dual} with the extra constraint $\bm\Gamma\succeq \bm 0$. In this case Lemma~\ref{prop:rprob_symmetry} remains valid and implies that we can restrict attention to permutation-symmetric solutions of the form $\bm\Gamma=\gamma_1 \mathbb I + \gamma_2 \bm 1 \bm 1^\intercal$ for some $\gamma_1, \gamma_2 \in \mathbb R$. As $\bm\Gamma=\gamma_1 \mathbb I + \gamma_2 \bm 1 \bm 1^\intercal \succeq \bm 0$ iff $\gamma_1 + T\gamma_2 \geq 0$ and $\gamma_1 \geq 0$ by virtue of~\cite[Proposition~4]{Rujeerapaiboon16}, we may then conclude that $\text{L}'(\gamma)$ coincides with the optimal value of~\eqref{opt:rprob_dual_simpl} with the extra constraints $\gamma_1 + T\gamma_2 \geq 0$ and $\gamma_1 \geq 0$. Note that~\eqref{opt:rprob_dual_simpl} is equivalent to~\eqref{eq:rprob0} and~\eqref{opt:rprob_dual_simpl2}. As $\gamma_1 + T\gamma_2 \geq 0$ is an explicit constraint of problem~\eqref{eq:rprob0}, it is necessarily an implicit constraint of the problems \eqref{opt:rprob_dual_simpl} and~\eqref{opt:rprob_dual_simpl2}. Thus, $\text{L}'(\gamma)$ coincides with the optimal value of~\eqref{opt:rprob_dual_simpl2} with the extra constraint $\gamma_1 \geq 0$. To prove the identity $\text{L}(\gamma) = \text{L}'(\gamma)$, it is therefore sufficient to show that appending the extra constraint $\gamma_1 \geq 0$ has no impact on the optimal value of~\eqref{opt:rprob_dual_simpl2}. 

To this end, fix any feasible solution of problem~\eqref{opt:rprob_dual_simpl2} with $\gamma_1 < 0$. As this solution must satisfy the constraint $\alpha + s\beta + s^2 \gamma_2 + s^2 \gamma_1 \geq 1$ for every $s\geq 0$ and as $s=T\mu > 0$, we have\begin{equation}
	\label{eq:Napat-strikes-back}
	\alpha + T \mu\beta + T^2\mu^2\gamma_2 + T^2\mu^2\gamma_1 \geq 1.
\end{equation}
Moreover, the objective function of \eqref{opt:rprob_dual_simpl2} can be reformulated as
\begin{equation*}
\begin{aligned}
	\alpha + &T\mu\beta + T(\mu^2 + \sigma^2)\gamma_1 + T\left[T\mu^2 + \sigma^2 + (T-1)\rho\sigma^2\right]\gamma_2 \\
	&= 
	\left( \alpha + T \mu\beta + T^2\mu^2\gamma_2 + T^2\mu^2\gamma_1 \right) + T(1-T) \left( \mu^2 + \rho\sigma^2 \right)\gamma_1 + T \sigma^2 (1 + (T-1)\rho)(\gamma_1 + \gamma_2),
\end{aligned}
\end{equation*}
which constitutes a sum of three terms. The first term in the sum is greater than or equal to 1 because of~\eqref{eq:Napat-strikes-back}, and the second term is strictly positive because $T \geq 2$, $\gamma_1<0$ and $\mu^2 + \rho\sigma^2 > 0$. The third term is non-negative because $\rho > -1 / (T - 1)$ and $\gamma_1+\gamma_2\geq 0$ is an explicit constraint of~\eqref{eq:rprob0} and thus an implicit constraint of~\eqref{opt:rprob_dual_simpl2}. In summary, we have shown that the objective value of any feasible solution of~\eqref{opt:rprob_dual_simpl2} with $\gamma_1<0$ is strictly greater than 1. As the optimal value $\text{L}(\gamma)$ of~\eqref{opt:rprob_dual_simpl2} represents a probability, however, we conclude that no feasible solution with $\gamma_1 < 0$ can optimize~\eqref{opt:rprob_dual_simpl2}. Thus, the extra constraint $\gamma_1\geq 0$ does not change the optimal value of~\eqref{opt:rprob_dual_simpl2}, and the claim follows. \qed
\end{proof}

\begin{thm}[Relaxed Right-Sided Chebyshev Bound]
\label{thm:lprob_bnd}
The relaxed right-sided Chebyshev bound admits the analytical solution
\begin{equation*}
\text{\em R}'(\gamma) = 
\begin{cases}
	1 & \text{if}\ 0<\gamma \leq \mu^T, \\
    \mu \gamma^{-1/T} & \text{if}\ \mu^T<\gamma <\left( \mu + \frac{\sigma^2\theta}{T\mu} \right)^T, \\
    \frac{\sigma^2\theta}{\sigma^2\theta + T(\mu - \gamma^{1/T})^2}  & \text{if}\ \gamma \geq \left( \mu + \frac{\sigma^2\theta}{T\mu} \right)^T,
\end{cases}
\end{equation*}
where $\theta=1 + (T-1)\rho>0$.
\end{thm}
\begin{proof}
Using similar arguments as in the proof of the previous theorem, one can show that $\text{R}'(\gamma)$ coincides with the optimal value of the following semi-infinite optimization problem:
\begin{equation}
\label{opt:lprob_dual}
\begin{array}{rl}
	\text{R}'(\gamma) =\
	\inf & \alpha + \mu \bm{1^{\intercal} \beta} + 
			 \left< (1 - \rho)\sigma^2\mathbb{I} + \left( \mu^2 + \rho\sigma^2 \right) \bm{11}^\intercal, \bm{\Gamma}\right> \\
	\st  & \alpha \in \mathbb{R},~ \bm{\beta} \in \mathbb{R}^T,~ \bm{\Gamma} \in \mathbb{S}^T_+ \\
	& \alpha + \bm{\overline \xi{}^{\intercal} \beta} + \bm{\overline \xi{}^{\intercal}\Gamma \overline \xi} \geq 0 \quad \forall \bm{\overline \xi} \geq \bm{0} \\
	& \alpha + \bm{\underline \xi^{\intercal} \beta} + \bm{\underline \xi^{\intercal}\Gamma \underline \xi} \geq 1 \quad \forall \bm{\underline \xi} \geq \bm{0}:\ \textstyle\prod_{t=1}^{T} \underline \xi_t \geq \gamma
	\end{array}
\end{equation}
Without loss of generality, we use different symbols $\bm{\overline\xi}$ and $\bm{\underline{\xi}}$ to denote the uncertain parameters in the two semi-infinite constraints, respectively.
Note that~(\ref{opt:lprob_dual}) can be viewed as the robust counterpart of an uncertain convex program with constraint-wise uncertainty sets~\cite{BenTal09}. As the left hand sides of the robust constraints are convex in the respective uncertainties, the `{\em primal worst equals dual best}' duality scheme portrayed in~\cite[Theorem~4.1]{Beck09} implies that~\eqref{opt:lprob_dual} is equivalent to
\begin{equation}
\label{opt:lprob_bidual}
\begin{array}{rl}
	\text{R}'(\gamma) =\
	\sup\	& q \\
	\st	& p,q \in \mathbb{R}_+,~ \bm{\overline \xi},\bm{\underline\xi} \in \mathbb{R}^+_T,~ \prod_{t=1}^{T}\underline \xi_t \geq \gamma \\
			& p + q = 1 \\
			& p \bm{\overline \xi} + q \bm{\underline\xi} = \mu\bm{1} \\
			& p \bm{\overline\xi\, \overline\xi}^\intercal + q \bm{\underline \xi\, \underline \xi}^\intercal \preceq (1 - \rho)\sigma^2\mathbb{I} + \left( \mu^2 + \rho\sigma^2 \right) \bm{11}^\intercal,
\end{array}
\end{equation}
where $p$ and $q$ represent dual variables assigned to the two robust constraints in~\eqref{opt:lprob_dual}. Thus, the primal uncertain convex program~\eqref{opt:lprob_dual} is solved under the worst possible realizations of $\bm{\overline \xi}$ and $\bm{\underline\xi}$, while the dual uncertain convex program~\eqref{opt:lprob_bidual} is solved under the best possible realizations, in which case $\bm{\overline \xi}$ and $\bm{\underline\xi}$ become decision variables. Problem~\eqref{opt:lprob_bidual} has intuitive appeal as it can be interpreted as a restriction of the original worst-case probability problem that minimizes over all two-point distributions in the ambiguity set $\mathcal P'$ with scenarios $\bm{\overline \xi}$ and $\bm{\underline \xi}$ and corresponding probabilities $p$ and $q$, respectively. Note that~\eqref{opt:lprob_bidual} constitutes a non-convex program because it involves multilinear terms in the decisions. Using the variable transformations $\bm u\leftarrow p\bm{\overline\xi}$ and $\bm v\leftarrow q\bm{\underline\xi}$ we can reformulate~\eqref{opt:lprob_bidual} as
\begin{equation}
\label{opt:lprob_dual_dual}
\begin{array}{rl}
	\text{R}'(\gamma) =\
	\sup\	& q \\
	\st	& p,q \in \mathbb{R}_+,~ \bm{u},\bm{v} \in \mathbb{R}_+^T\\
			& \prod_{t=1}^T v_t \geq q^T \gamma \\
			& p + q = 1 \\
			& \bm{u} + \bm{v} = \mu\bm{1} \\
			& \textstyle\frac{1}{p}\bm{uu}^\intercal + \frac{1}{q}\bm{vv}^\intercal \preceq (1 - \rho)\sigma^2\mathbb{I} + \left( \mu^2 + \rho\sigma^2 \right) \bm{11}^\intercal.
\end{array}
\end{equation}
Note that if $p=0$ ($q=0$), then $\bm u=\bm 0$ ($\bm v=\bm 0$) for otherwise the matrix inequality is not satisfiable. In~\eqref{opt:lprob_dual_dual} and below we adhere to the convention that $0/0=0$, which reflects the idea that a scenario with zero probability mass should have zero weight in the covariance matrix. 
Observe that problem~\eqref{opt:lprob_dual_dual} is a convex program. In particular, the first constraint is convex because of the concavity of geometric means, and the last constraint is convex due to a standard Schur complement argument. Exploiting the problem's permutation symmetry and convexity, one can proceed as in Lemma \ref{prop:rprob_symmetry} to show that~\eqref{opt:lprob_dual_dual} has a permutation symmetric minimizer of the form $\bm{u} = u\bm{1}$ and $\bm{v} = v\bm{1}$ for some scalar decision variables $u, v \in \mathbb{R}_+$. Restricting the search to permutation symmetric solutions, problem~\eqref{opt:lprob_dual_dual} can therefore be reformulated as
\begin{equation}
\label{opt:lprob_dual_dual_1}
\begin{array}{rl}
	\text{R}'(\gamma) =\
	\sup\ & q \\
	\st	& p, q, u, v \in \mathbb{R}_+ \\
			& v \geq q \gamma^{1/T} \\
			& p + q = 1 \\
			& u + v = \mu \\
			& (1 - \rho)\sigma^2\mathbb{I} + \left( \mu^2 + \rho\sigma^2 - \textstyle\frac{u^2}{p} - \frac{v^2}{q} \right) \bm{11}^\intercal \succeq \bm{0}.
\end{array}
\end{equation}
It can be shown that the eigenvalues of the matrix $(1 - \rho)\sigma^2\mathbb{I} + ( \mu^2 + \rho\sigma^2 - \textstyle\frac{u^2}{p} - \frac{v^2}{q}) \bm{11}^\intercal$ are given by $(1-\rho)\sigma^2$ and $(1 - \rho)\sigma^2 + T(\mu^2 + \rho\sigma^2 - \textstyle\frac{u^2}{p} - \textstyle\frac{v^2}{q})$; see {\em e.g.}~\cite[Proposition~4]{Rujeerapaiboon16}. Since $(1-\rho)\sigma^2 > 0$ by assumption, the matrix inequality in~\eqref{opt:lprob_dual_dual_1} is equivalent to the scalar constraint
\begin{equation}
\label{eq:lprob_dual_dual_3_1}
	(1 - \rho)\sigma^2 + T \left(\mu^2 + \rho\sigma^2 - \textstyle\frac{u^2}{p} - \textstyle\frac{v^2}{q} \right) \geq 0.
\end{equation}
Any feasible solution of~\eqref{opt:lprob_dual_dual_1} satisfies $q\gamma^{1/T} \leq v \leq \mu$, implying that the optimal value of~\eqref{opt:lprob_dual_dual_1} is bounded above by $\min\{1, \mu\gamma^{-1/T}\}$. For $0<\gamma^{1/T} \leq \mu$, an optimal solution of~\eqref{opt:lprob_dual_dual_1} is then given by $(p, q, u, v) = (0, 1,0, \mu)$, and the optimal value is equal to 1. For $\mu<\gamma^{1/T} < \mu + \frac{(1+(T-1)\rho)\sigma^2}{T\mu}$, on the other hand, an optimal solution is given by $(p, q, u, v) = (1 - \mu \gamma^{-1/T}, \mu \gamma^{-1/T}, 0, \mu)$ with corresponding optimal value $\mu \gamma^{-1/T}$. Indeed, any larger value of $q$ would require a larger value of $v$, which in turn would violate the non-negativity of $u$ as $u + v = \mu$. One can show that the constraint~\eqref{eq:lprob_dual_dual_3_1} is always {\em in}active at this solution. For $\gamma^{1/T} \geq \mu + \frac{(1+(T-1)\rho)\sigma^2}{T\mu}$, finally, the constraint~\eqref{eq:lprob_dual_dual_3_1} implies that $q$ must not exceed $\mu \gamma^{-1/T}$, which in turn implies that the constraint must be binding. Furthermore, $q$ has to be strictly positive for otherwise \eqref{opt:lprob_dual_dual_1} would be solved by $(p, q, u, v) = (1, 0, \mu, 0)$, which contradicts our earlier finding that the constraint~\eqref{eq:lprob_dual_dual_3_1} is binding. Substituting $p=1-q$ and $u=\mu-v$, the left hand side of~\eqref{eq:lprob_dual_dual_3_1} becomes a quadratic function of~$v$ parametric in $q$. We denote the two roots of this function by $v^+$ and $v^-$ and define $u^+=\mu-v^+$ and $u^-=\mu-v^-$. A direct calculation yields
\begin{equation*}
	u^\pm = (1-q) \mu \pm \sigma \sqrt{1 + (T-1)\rho} \sqrt{\frac{q(1-q)}{T}} \quad\text{and}\quad 
	 v^\pm = q \mu \mp \sigma \sqrt{1 + (T-1)\rho} \sqrt{\frac{q(1-q)}{T}} .
\end{equation*}
By construction, both $(u^+,v^+)$ and $(u^-,v^-)$ satisfy~\eqref{eq:lprob_dual_dual_3_1} as an equality. However, there is no $q\in(0,1]$ for which $(u^+,v^+)$ is feasible in~\eqref{opt:lprob_dual_dual_1}. Indeed, a direct calculation reveals that
the constraint $v^+ \geq q \gamma^{1/T}$ from~\eqref{opt:lprob_dual_dual_1} can hold only if
\begin{equation}
	\label{eq:infeasibility2}
	q (\mu - \gamma^{1/T})\geq \sigma \sqrt{1 + (T-1)\rho} \sqrt{\frac{q(1-q)}{T}}. 
\end{equation}
However,~\eqref{eq:infeasibility2} is not satisfiable as its left hand side is strictly negative by assumption, whereas its right hand side is non-negative. Therefore, $(u^+,v^+)$ is infeasible in~\eqref{opt:lprob_dual_dual_1}.

In contrast, the second solution $(u^-, v^-)$ is feasible in~\eqref{opt:lprob_dual_dual_1} if we select $q\in(0,1]$ with
\begin{equation*}
	u^-\geq 0 \quad \iff \quad q \leq \frac{T\mu^2}{T\mu^2 + \sigma^2(1 + (T-1)\rho)}
\end{equation*}
and
\begin{equation*}
	 v^- \geq q \gamma^{1/T}\quad \iff \quad q \leq  \frac{\sigma^2(1 + (T-1)\rho)}{\sigma^2(1 + (T-1)\rho) + T(\mu - \gamma^{1/T})^2}.
\end{equation*}
Problem~\eqref{opt:lprob_dual_dual_1} aims to maximize $q$, which is tantamount to setting
\begin{equation*}
\begin{aligned}
q &= \min \left\lbrace \frac{T \mu^2}{T\mu^2 + \sigma^2(1 + (T-1)\rho)}, \frac{\sigma^2(1 + (T-1)\rho)}{\sigma^2(1 + (T-1)\rho) + T(\mu - \gamma^{1/T})^2}\right\rbrace \\
  &= \frac{\sigma^2(1 + (T-1)\rho)}{\sigma^2(1 + (T-1)\rho) + T(\mu - \gamma^{1/T})^2},
\end{aligned}
\end{equation*}
where the second equality follows from $\gamma^{1/T} \geq \mu + \frac{(1+(T-1)\rho)\sigma^2}{T\mu}$. Thus, the claim follows. \qed
\end{proof}

In addition to admitting an analytical solution, the relaxed right-sided Chebyshev bounds also allow us to determine a distribution $\mathbb{P}^\star \in \mathcal{P}'$ that attains the probability bound.

\begin{cor}[Extremal Distribution]\label{cor:dist}
A distribution $\mathbb{P}^\star \in \mathcal{P}'$ attaining the relaxed right-sided Chebyshev bound $\text{\em R}'(\gamma)$ is given by $\mathbb P^\star=p^\star \delta_{[u^\star/p^\star] \bm 1}+q^\star \delta_{[v^\star/q^\star] \bm 1}$, where
\begin{equation*}
q^\star = 
\begin{cases}
	1 & \text{if}\ 0<\gamma \leq \mu^T, \\
    \mu \gamma^{-1/T} & \text{if}\ \mu^T <\gamma < \left(\mu + \frac{\sigma^2\theta}{T\mu}\right)^T, \\
    \frac{\sigma^2\theta}{\sigma^2\theta + T(\mu - \gamma^{1/T})^2} & \text{if}\ \gamma \geq \left( \mu + \frac{\sigma^2\theta}{T\mu} \right)^T,
\end{cases}
\end{equation*}
and $p^\star=1-q^\star$, as well as
\begin{equation*}
v^\star = 
\begin{cases}
	\mu & \text{if}\ 0<\gamma < \left( \mu + \frac{\sigma^2\theta}{T\mu} \right)^T, \\
    	q^\star \mu + \sigma \sqrt{\frac{\theta q^\star(1-q^\star)}{T}} & \text{if}\ \gamma \geq \left( \mu + \frac{\sigma^2\theta}{T\mu} \right)^T
\end{cases}
\end{equation*}
and $u^\star=\mu-v^\star$, where $\theta=1 + (T-1)\rho>0$.
\end{cor}
\begin{proof}
The proof follows directly from that of Theorem~\ref{thm:lprob_bnd} and is thus omitted. 
\qed
\end{proof}
 
The relaxed left-sided and right-sided Chebyshev bounds differ in the sense that the left-sided bound coincides with $\text{L} (\gamma)$, whereas $\text{R}' (\gamma)$ does not equal $\text{R} (\gamma)$ in general. The relaxed right-sided Chebyshev bound does coincide with $\text{R} (\gamma)$, however, when $T$ is sufficiently large.

\begin{prop}
\label{prop:r_equals_r_prime}
If $\mu > \sqrt{\frac{1-\rho}{T}}\sigma$, then $\text{\em R}'(\gamma) = \text{\em R}(\gamma)$ for all $\gamma \geq \overline\gamma$, where 
\begin{equation*}
	\overline\gamma^{1/T} = \mu + \frac{1}{2ab}\left( 1 + \sqrt{4ab\sqrt{\frac{1-\rho}{T}}\sigma + 1}\right)
\end{equation*}  
with $a = \mu - \sqrt{\frac{1-\rho}{T}}\sigma$, $b = \frac{T}{\sigma^2\theta}$ and $\theta = 1 + (T-1)\rho$.
\end{prop}
Note that $ab \rightarrow \infty$ and thus $\overline\gamma^{1/T} \rightarrow \mu$ whenever $T \rightarrow \infty$. The rate of convergence depends on $\mu$, $\sigma$ and $\rho$, and the fastest convergence is observed for large $\mu$ and small $\sigma$ and $\rho$.
\begin{proof}
We first show that $\overline\gamma^{1/T} > \mu + \frac{\sigma^2\theta}{T\mu}$ (Step~1), which allows us to invoke Theorem \ref{thm:lprob_bnd} to conclude that $\text{R}'(\gamma) = \frac{\sigma^2\theta}{\sigma^2\theta + T(\mu - \gamma^{1/T})^2}$. We then employ Corollary \ref{cor:dist} to construct a distribution $\mathbb P^\star \in \mathcal P'$ that satisfies $\mathbb P^\star \left( \prod_{t=1}^T \tilde\xi_t \geq \gamma \right) = \text{R}'(\gamma)$ (Step 2), and we show that a suitable perturbation of $\mathbb P^\star$ results in a distribution $\mathbb P \in \mathcal P$ that satisfies $\mathbb P \left( \prod_{t=1}^T \tilde\xi_t \geq \gamma \right) = \mathbb P^\star \left( \prod_{t=1}^T \tilde\xi_t \geq \gamma \right)$ (Step 3). The statement then follows from the fact that $\text{R} (\gamma)$ is bounded above by $\text{R}'(\gamma)$.

\paragraph{Step 1:}
We show that $\overline\gamma^{1/T}$ is the maximum root of the convex quadratic function
\begin{equation*}
q(x) = \sigma^2\theta \left[ a\left( 1 +b (\mu - x)^2 \right) - x \right],
\end{equation*}
where $a$ and $b$ are defined in the statement of the theorem,
and that this root satisfies $\overline\gamma^{1/T} > \mu + \frac{\sigma^2\theta}{T\mu}$. From the quadratic formula we know that the maximum root $x^\star$ of $q (x)$ satisfies
\begin{equation*}
	x^\star = \frac{2ab\mu + 1 + \sqrt{(2ab\mu + 1)^2 - 4a^2b(b\mu^2+1)}}{2ab}
	= \mu + \frac{1}{2ab} \left( 1 + \sqrt{4ab(\mu-a) + 1}\right),
\end{equation*}
and replacing $a$ and $b$ with their definitions reveals that $x^\star = \overline\gamma^{1/T}$. To show that $\overline\gamma^{1/T} > \mu + \frac{\sigma^2\theta}{T\mu}$, we observe that
\begin{equation*}
\begin{aligned}
	q\left(\mu + \frac{\sigma^2\theta}{T\mu}\right) &= \left( \mu - \sqrt{\frac{1-\rho}{T}}\sigma \right)  \left( \sigma^2\theta + \frac{\sigma^4\theta^2}{T\mu^2} \right) - \sigma^2\theta \left(\mu + \frac{\sigma^2\theta}{T\mu}\right) \\
	&= \sigma^2\theta (\sigma^2\theta + T\mu^2) \left( \frac{\mu - \sqrt{(1-\rho)/T}\sigma}{T  \mu^2} - \frac{1}{T\mu}\right) < 0,
\end{aligned}
\end{equation*}
as well as $q (x) \rightarrow \infty$ for $x \rightarrow \infty$ since $\mu > \sqrt{\frac{1-\rho}{T}}\sigma$. Since $q(x)$ is quadratic, both observations imply that the maximum root $x^\star = \overline\gamma^{1/T}$ of $q(x)$ indeed belongs to the interval $\left( \mu + \frac{\sigma^2\theta}{T\mu}, \infty \right)$.

\paragraph{Step 2:}
The distribution $\mathbb{P}^\star$ in Corollary~\ref{cor:dist} satisfies $\mathbb P^\star \left( \prod_{t=1}^T \tilde\xi_t \geq \gamma \right) = \text{R}'(\gamma)$. For later reference, we remark that $\mathbb{P}^\star = p^\star \delta_{[u^\star/p^\star] \bm 1} + q^\star \delta_{[v^\star/q^\star] \bm 1}$ satisfies the properties
\begin{equation}\label{eq:property_of_P_star}
	v^\star = q^\star\gamma^{1/T}, \quad
	u^\star + v^\star = \mu \quad \text{and} \quad
	\frac{(u^\star)^2}{p^\star} + \frac{(v^\star)^2}{q^\star} = \mu^2 + \frac{1}{T}(1+(T-1)\rho)\sigma^2 .
\end{equation}
Note that the last condition holds because \eqref{eq:lprob_dual_dual_3_1} is binding when $\gamma^{1/T} \geq \mu + \frac{\sigma^2\theta}{T\mu}$.

\paragraph{Step 3:}
Consider the distribution $\mathbb P$ defined through
\begin{equation*}
\begin{aligned}
	\mathbb{P} \left( \tilde{\bm\xi} = \left(\frac{u^\star}{p^\star} - \lambda\right)\bm{1} + T\lambda\mathbf e_i \right) = \frac{1}{T}p^\star, \quad i = 1, \hdots, T,  \quad \text{and} \quad
	\mathbb{P} \left( \tilde{\bm\xi} = \frac{v^\star}{q^\star}\bm{1} \right) = q^\star
\end{aligned}
\end{equation*}
with $\lambda = \sqrt{\frac{1-\rho}{p^\star T}}\sigma$.
If $\mathbb{P} \in \mathcal{P}$, then we find that
\begin{equation*}
	\text{R}(\gamma) \geq 
	\mathbb{P} \left( \prod_{t=1}^T \tilde\xi_t = \gamma \right) \geq 
	\mathbb{P} \left( \tilde{\bm\xi} = \frac{v^\star}{q^\star}\bm{1} \right) =
	q^\star = \text{R}'(\gamma),
\end{equation*}
which implies $\text{R}(\gamma) = \text{R}'(\gamma)$. We thus need to show that $\mathbb{P} \in \mathcal{P}$. To this end, we first observe that the first two moments of $\bm{\tilde{\xi}}$ under $\mathbb P$ satisfy
\begin{equation*}
\begin{aligned}
	\mathbb{E}_{\mathbb P} \left( \tilde{\bm\xi} \right) &= 
	\frac{p^\star}{T}\sum_{i=1}^T \left(\left(\frac{u^\star}{p^\star} - \lambda\right)\bm{1} + T\lambda\mathbf e_i \right) + v^\star \bm 1 = (u^\star + v^\star)\bm 1 = \bm\mu \\
	\mathbb{E}_{\mathbb P} \left( \tilde{\bm\xi}\tilde{\bm\xi}^\intercal \right) &= 
	\frac{p^\star}{T}\sum_{i=1}^T \left( \left( \frac{u^\star}{p^\star} - \lambda \right)\bm{1} + T\lambda\mathbf e_i \right)\left( \left( \frac{u^\star}{p^\star} - \lambda \right)\bm{1} + T\lambda\mathbf e_i \right)^\intercal + \frac{(v^\star)^2}{q^\star} \bm{11}^\intercal \\
	&= \frac{p^\star}{T} \left( \left( T \left( \frac{u^\star}{p^\star} - \lambda \right)^2 + 2\left(\frac{u^\star}{p^\star} - \lambda\right)T\lambda \right)\bm{11}^\intercal + T^2\lambda^2\mathbb{I}\right) + \frac{(v^\star)^2}{q^\star} \bm{11}^\intercal \\
	&= \left( \frac{(u^\star)^2}{p^\star} + \frac{(v^\star)^2}{q^\star} - p^\star\lambda^2 \right)\bm{11}^\intercal + p^\star T\lambda^2\mathbb{I} \\
	&= (\mu^2 + \rho\sigma^2)\bm{11}^\intercal + (1-\rho)\sigma^2 \mathbb{I} = \bm\Sigma + \bm\mu\bm\mu^\intercal,
\end{aligned}
\end{equation*}
where the last row is due to~\eqref{eq:property_of_P_star} and our definition of $\lambda$. It remains to be shown that $\bm{\tilde{\xi}}$ is non-negative $\mathbb P$-a.s. By construction of $\mathbb{P}$, this is the case iff $u^\star \geq p^\star \lambda$. We now observe that
\begin{equation*}
	u^\star = \mu - q^\star \gamma^{1/T} = \mu - \frac{\sigma^2\theta \gamma^{1/T}}{\sigma^2\theta + T(\mu - \gamma^{1/T})^2} \geq \sqrt{\frac{1-\rho}{T}}\sigma,
\end{equation*}
where the first identity follows from~\eqref{eq:property_of_P_star}, the second one is due to the definition of $q^\star$ in Corollary~\ref{cor:dist}, and the inequality holds since there is $C > 0$ such that
\begin{equation*}
q \left( \gamma^{1/T} \right) = C \left[ \mu - \frac{\sigma^2\theta \gamma^{1/T}}{\sigma^2\theta + T(\mu - \gamma^{1/T})^2} - \sqrt{\frac{1-\rho}{T}}\sigma \right],
\end{equation*}
and this expression is non-negative whenever $\gamma \geq \overline{\gamma}$. 
We thus conclude that 
\begin{equation*}
	\frac{(u^\star)^2}{p^\star} \geq (u^\star)^2 \geq \frac{(1-\rho)\sigma^2}{T},
\end{equation*}
which in turn implies that $u^\star \geq \sqrt{\frac{(1-\rho)p^\star}{T}}\sigma = p^\star\lambda$ as desired. The claim now follows.
\qed 
\end{proof}

\section{Extensions}
\label{sec:extensions}
The techniques developed in this paper can also be used to construct Chebyshev bounds for sums, minima and maxima of non-negative random variables. All these Cheybshev bounds can be reduced to computing $\sup_{\mathbb P \in \mathcal P} \mathbb{P} ( h(\tilde{\bm\xi}) \leq 0 )$ for some permutation-symmetric functional $h(\bm \xi)$. 
\begin{thm}
\label{thm:permutation-symmetric-h} For any permutation-symmetric continuous functional $h:\mathbb R_+^T\rightarrow\mathbb R$, we have
\begin{equation}
\label{eq:h}
\begin{array}{rcl}
	\displaystyle \sup_{\mathbb P \in \mathcal P} \mathbb{P} ( h(\tilde{\bm\xi}) \leq 0 ) ~= & 
	\inf & \alpha + T\mu\beta +	T(\mu^2 + \sigma^2)\gamma_1 + T\left[ T\mu^2 + \sigma^2 + (T-1)\rho\sigma^2\right] \gamma_2 \\
	&\st & \alpha, \lambda_1, \lambda_2 \in\mathbb R_+,~ \beta, \gamma_1, \gamma_2 \in \mathbb{R} \\
	&     & \gamma_1+\gamma_2\geq 0,~\gamma_2 + \gamma_1 + \alpha \geq \left\Vert \left( \beta - \lambda_1, \gamma_2 + \gamma_1 - \alpha \right) \right\Vert_2 \\
	&     & \frac{\gamma_1}{T}+\gamma_2\geq 0,~\gamma_2 + \frac{\gamma_1}{T} + \alpha \geq \left\Vert \left( \beta - \lambda_2, \gamma_2 + \frac{\gamma_1}{T} - \alpha \right) \right\Vert_2 \\
	&     &  \alpha + \beta s + \gamma_2 s^2 + \gamma_1 \underline{\phi}(s) \geq 1 \quad \forall s \in \mathcal S\\
	&     &  \alpha + \beta s + \gamma_2 s^2 + \gamma_1 \overline{\phi}(s) \geq 1\quad \forall s \in \mathcal S,
	\end{array}
\end{equation}
where the optimal value functions $\underline\phi(s)$ and $\overline\phi(s)$ are defined as
\begin{equation*}
	\underline\phi(s) = \inf_{\bm{\xi} \geq \bm{0}} \left\{ \Vert \bm{\xi} \Vert_2^2 : \Vert \bm{\xi} \Vert_1 = s, ~h(\bm\xi) \leq 0 \right\} \quad\text{and}\quad
	\overline\phi(s) = \sup_{\bm{\xi} \geq \bm{0}} \left\{ \Vert \bm{\xi} \Vert_2^2 : \Vert \bm{\xi} \Vert_1 = s, ~h(\bm\xi) \leq 0 \right\}
\end{equation*}
for all $s\geq 0$, while $\mathcal{S} = \left\{ s \in \mathbb R_+ : \underline\phi(s) < +\infty \right\}$ denotes the effective domain of $\underline\phi(s)$ and $\overline\phi(s)$.
\end{thm}
\begin{proof}
The proof is largely based on arguments familiar from Theorems~\ref{thm:rprob} and~\ref{thm:lprob}. Details are omitted for brevity of exposition. \qed  
\end{proof}

The significance of Theorem~\ref{thm:permutation-symmetric-h} is that it enables us to compute $\sup_{\mathbb P \in \mathcal P} \mathbb{P} ( h(\tilde{\bm\xi}) \leq 0 )$ by solving a semidefinite program whenever $\underline\phi(s)$ and $\overline\phi(s)$ are piecewise polynomials. In this case the last two constraints in~\eqref{eq:h} reduce to the requirement that a univariate piecewise polynomial, whose coefficients depend affinely on the decision variables, must be non-negative uniformly on~$\mathcal S$. Such conditions can systematically be reformulated as linear matrix inequalities~\cite{Nesterov00}.

\begin{table}
\begin{center}
\begin{sideways}
\begin{small}
\begin{tabular}{C{4cm} || l | l@{}l | l@{}l}
\multicolumn{1}{c}{Chebyshev bound} & \multicolumn{1}{c}{$h(\bm\xi)$} &  \multicolumn{2}{c}{$\underline\phi(s)$}   & \multicolumn{2}{c}{$\overline\phi(s)$} \\[1ex] \hline
 $\displaystyle \sup_{\mathbb P\in\mathcal P}\mathbb{P} \Big( \min_{t=1,\ldots,T} \tilde\xi_t\leq \gamma\Big)$ 	& $\displaystyle \min_{t=1,\ldots,T} \xi_t -\gamma$ 		& 
 $\left\{ \begin{array}{l}
		\frac{s^2}{T} \\
		\gamma^2 + \frac{1}{T-1}(s - \gamma)^2 \end{array} \right.$ & $ \begin{array}{l}
		\text{if } s \geq \gamma T \\
		\text{if } 0 \leq s < \gamma T
	\end{array}$
	& $s^2$ 
	\\ \hline
 $\displaystyle \sup_{\mathbb P\in\mathcal P}\mathbb{P} \Big( \min_{t=1,\ldots,T} \tilde\xi_t\geq \gamma\Big)$ 	& $\displaystyle  \gamma - \min_{t=1,\ldots,T} \xi_t$ 		& 
 $\left\{ \begin{array}{l}
		\frac{s^2}{T} \\
		+\infty 
	\end{array} \right.$ 
	&$ \begin{array}{l}
		\text{if } s \geq \gamma T \\
		\text{if } 0 \leq s < \gamma T
	\end{array}$
	 & $\left\{ \begin{array}{l}
		(s-\gamma T)^2 + 2\gamma(s - \gamma T) + T\gamma^2 \\
		-\infty 
	\end{array} \right.$ & $\begin{array}{l}
		\text{if } s \geq \gamma T \\
		\text{if } 0 \leq s < \gamma T
	\end{array}$
	 \\ \hline
 $\displaystyle \sup_{\mathbb P\in\mathcal P}\mathbb{P} \Big( \max_{t=1,\ldots,T} \tilde\xi_t\leq \gamma\Big)$ 	& $\displaystyle \max_{t=1,\ldots,T} \xi_t- \gamma$		& 
 $\left\{ \begin{array}{l} \frac{s^2}{T} \\ +\infty \end{array}\right.$ & 
 $\begin{array}{l} \text{if } 0 \leq s \leq \gamma T\\ \text{if } s > \gamma T\end{array}$
  & $\left\{ \begin{array}{l}
		s^2 \\
		\gamma^2 + (s - \gamma)^2 \\
		~ \vdots \\
		(T-1)\gamma^2 + (s - (T-1)\gamma)^2 \\
		-\infty 
	\end{array} \right.$ & 
	$\begin{array}{l}
		\text{if } 0 \leq s \leq \gamma \\
		\text{if } \gamma < s \leq 2\gamma \\  ~\vdots \\
		\text{if } (T-1)\gamma < s \leq T\gamma \\
		\text{if } s > T\gamma
	\end{array}$\\ \hline
 $\displaystyle \sup_{\mathbb P\in\mathcal P}\mathbb{P} \Big( \max_{t=1,\ldots,T} \tilde\xi_t\geq \gamma\Big)$ 	& $\displaystyle \gamma - \max_{t=1,\ldots,T} \xi_t$		&  
 $\left\{ \begin{array}{l}
		\frac{s^2}{T} \\
		\gamma^2 + \frac{1}{T-1}(s - \gamma)^2 \\
		+\infty 
	\end{array} \right.$ &  $ \begin{array}{l}
		\text{if } s \geq \gamma T\\
		\text{if } \gamma \leq s < \gamma T \\
		\text{if } 0 \leq s < \gamma
	\end{array}$
	 &  
	$\left\{ \begin{array}{l}
		s^2 \\
		-\infty 
	\end{array} \right.$ & 
	$\begin{array}{l}
		\text{if } s \geq \gamma \\
		\text{if } 0 \leq s < \gamma
	\end{array}$  \\ \hline
 $\displaystyle \sup_{\mathbb P\in\mathcal P}\mathbb{P} \Big( \sum_{t=1}^T \tilde\xi_t\leq \gamma\Big)$ 	& $\displaystyle  \sum_{t=1}^T \tilde\xi_t-\gamma$		& 
 $\left\{ \begin{array}{l}
		\frac{s^2}{T} \\
		+\infty  
	\end{array} \right.$
	& $\begin{array}{l}
		\text{if } 0 \leq s \leq \gamma \\
		\text{if } s > \gamma 
	\end{array}$
	 & $\left\{ \begin{array}{l}
		s^2  \\
		-\infty 
	\end{array} \right.$ & $\begin{array}{l}
		 \text{if } 0 \leq s \leq \gamma \\
		\text{if } s > \gamma 
	\end{array}$ \\ \hline
 $\displaystyle \sup_{\mathbb P\in\mathcal P}\mathbb{P} \Big( \sum_{t=1}^T \tilde\xi_t\geq \gamma\Big)$ 	& $\displaystyle \gamma- \sum_{t=1}^T \tilde\xi_t$		& 
 $\left\{ \begin{array}{l}
		\frac{s^2}{T} \\
		+\infty 
	\end{array} \right.$ & 
	$\begin{array}{l}
		\text{if } s \geq \gamma \\
		\text{if }  0 \leq s < \gamma
	\end{array}$
	 & $\left\{ \begin{array}{l}
		s^2 \\
		-\infty 
	\end{array} \right.$
	& 
	$\begin{array}{l}
		\text{if } s \geq \gamma \\
		\text{if }  0 \leq s < \gamma
	\end{array}$ \\ \hline
\end{tabular}
\end{small}
\end{sideways}
\end{center}
\caption{Chebyshev bounds equivalent to $\sup_{\mathbb P \in \mathcal P} \mathbb{P} ( h(\tilde{\bm\xi}) \leq 0 )$ for some permutation symmetric functional~$h(\bm \xi)$. These bounds coincide with the optimal value of~\eqref{eq:h}, instantiated with the respective piecewise polynomials~$\underline\phi(s)$ and~$\overline\phi(s)$.}
\label{tab:Napat-on-fire}
\end{table}

Table~\ref{tab:Napat-on-fire} lists examples of permutation-symmetric functionals $h(\bm\xi)$ that lead to piecewise polynomial mappings $\underline\phi(s)$ and $\overline\phi(s)$ and thus to computable Chebyshev bounds. Theorems~\ref{thm:sum1} and~\ref{thm:sum2} below present two special cases in which these bounds can be evaluated analytically.

\begin{thm}[Left-Sided Chebyshev Bound for Sums]
\label{thm:sum1}
For any $\gamma > 0$ we have 
\begin{equation*}
	\sup_{\mathbb{P} \in \mathcal{P}}\mathbb{P} \left( \sum_{t=1}^T \tilde{\xi}_t \geq \gamma \right) = 
	\left\{\begin{array}{ll}
		\frac{T\sigma^2\theta}{T\sigma^2\theta+(\gamma-T\mu)^2} & \text{if } \gamma\geq  T\mu+\sigma^2\theta/\mu,\\
		\frac{T\mu}{\gamma} & \text{if }T\mu\leq \gamma<T\mu+\sigma^2\theta/\mu,\\
		1 & \text{if } \gamma<T\mu,
	\end{array}\right.
\end{equation*}
where $\theta = 1 + (T-1)\rho > 0$.
\end{thm}
\begin{proof}
By Theorem~\ref{thm:permutation-symmetric-h} the Chebyshev bound $\sup_{\mathbb P \in \mathcal P}\mathbb{P}( \sum_{t=1}^T \tilde{\xi}_t \geq \gamma )$ can be reformulated as the semi-infinite program~\eqref{eq:h}, where the functions $\underline\phi(s)$ and $\overline\phi(s)$ are specified in Table~\ref{tab:Napat-on-fire}. Distinguishing the cases $\gamma_1\geq 0$ and $\gamma_1< 0$, this semi-infinite program can be reduced to a robust optimization problem with a scalar uncertain parameter by using the `{\em primal worst equals dual best}' duality scheme from robust optimization~\cite{Beck09}. One can further show that the optimal value of this problem coincides with the univariate Chebyshev bound $\sup_{\mathbb P_1 \in \mathcal P}\mathbb{P}( \tilde{\xi}\geq \gamma)$, where $\mathcal P_1$ contains all distributions of $\tilde\xi$ supported on $\mathbb R_+$ with mean $T\mu$ and variance $\sigma^2T(1+(T-1)\rho)$. The latter Chebyshev bound has an analytical formula, which can be obtained from~\eqref{eq:chebyshev-nonnegative}. \qed
\end{proof}

\begin{thm}[Right-Sided Chebyshev Bound for Sums]
\label{thm:sum2}
For any $\gamma > 0$ we have
\begin{equation*}
	\sup_{\mathbb P \in \mathcal P}\mathbb{P} \left( \sum_{t=1}^T \tilde{\xi}_t \leq \gamma \right) = 
	\left\{\begin{array}{ll}
		1 & \text{if } \gamma \geq T\mu, \\
		\frac{T\sigma^2\theta}{T\sigma^2\theta+(\gamma-T\mu)^2} & \text{otherwise,}
	\end{array}\right.
\end{equation*}
where $\theta = 1 + (T-1)\rho > 0$.
\end{thm}

\begin{proof}
The proof is widely parallel to that of Theorem \ref{thm:sum1} and is thus omitted for brevity.
\end{proof}

\section{Numerical Experiments}
\label{sec:examples}

We first compare our Chebyshev bounds $\text{R} (\gamma)$ and $\text{L} (\gamma)$ with alternative bounds proposed in the literature, as well as the relaxed Chebyshev bound $\text{R}' (\gamma)$ from Section~\ref{section:cov_bounds}.
We then present a case study that employs our left-sided Chebyshev bound $\text{L} (\gamma)$ to select financial portfolios under imprecise knowledge of the asset return distributions. All optimization problems are solved with the SDPT3 optimization software using the YALMIP interface \cite{Lofberg04, SDPT3}.

\subsection{Comparison of Chebyshev Bounds}\label{sec:num_exp:confidence_intervals}

Instead of employing the bounds $\text{R} (\gamma)$ and $\text{L} (\gamma)$ from Sections~\ref{section:right-prob} and~\ref{section:left-prob}, which are exact but may result in computationally challenging optimization problems, one can employ existing results to derive approximate bounds on the tail probabilities of a product of non-negative, permutation-symmetric random variables. In the following, we compare our bounds with two such approximations based on earlier results of Marshall and Olkin \cite{Marshall60} and Vandenberghe et al.~\cite{Vandenberghe07}. Both approximations rely on the larger ambiguity set
\begin{equation*}
\begin{aligned}
	\mathcal{P}^0 = 
	\left \lbrace 
		\mathbb{P} \in \mathcal M_+(\mathbb R^T)~: 
			\mathbb{E}_\mathbb{P} \left( \tilde{\bm{\xi}} \right) = \bm \mu ,\;
			\mathbb{E}_\mathbb{P} \left( \tilde{\bm{\xi}}\tilde{\bm{\xi}}^\intercal \right) = \bm \Sigma+\bm \mu\bm\mu^\intercal  
	\right \rbrace
\end{aligned}
\end{equation*}
with support $\mathbb{R}^T$, where $\bm{\mu} \in \mathbb{R}^T$ and $\bm{\Sigma} \in \mathbb{S}^T_+$, $\bm{\Sigma} \succ \bm{0}$, need not be permutation-symmetric.

Marshall and Olkin~\cite{Marshall60} derive a convex optimization problem that provides a tight upper bound on the probability that the random vector $\bm{\tilde{\xi}}$ is contained in a closed convex set $\mathcal{C}$, assuming that $\bm{\tilde{\xi}}$ can be governed by any distribution from the ambiguity set $\mathcal{P}^0$. The choice $\mathcal{C} = \big\{ \bm{\xi} \in \mathbb{R}^T \, : \, \prod_{t=1}^T \xi_t \geq \gamma \big\}$ allows us to approximate the right-sided Chebyshev bound $\text{R} (\gamma)$. For this special case, the bound of Marshall and Olkin has the analytical solution
\begin{equation*}
	\text{R}^\text{MO} (\gamma) = \left\{ \begin{array}{ll}
		1 & \text{if } 0 < \gamma \leq \mu^T, \\
		\frac{\sigma^2(1 + (T-1)\rho)}{\sigma^2(1 + (T-1)\rho) + T(\mu - \gamma^{1/T})^2} & \text{if } \gamma > \mu^T,
	\end{array} \right.
\end{equation*}
which follows from~\cite[Theorem~6.1]{Bertsimas05}. By construction, $\text{R}^\text{MO} (\gamma) \geq \text{R} (\gamma)$ since $\mathcal{P} \subset \mathcal{P}^0$. Note that $\text{R}^\text{MO} (\gamma)$ coincides with our relaxed Chebyshev bound $\text{R}' (\gamma)$ for $\gamma \geq \big( \mu + \frac{\sigma^2 \theta}{T \mu} \big)^T$, see Theorem~\ref{thm:lprob_bnd}. Thus, $\text{R}^\text{MO} (\gamma)$ also coincides with our right-sided Chebyshev bound $\text{R} (\gamma)$ for large values of $\gamma$, see Proposition~\ref{prop:r_equals_r_prime}. Note that the bound of Marshall and Olkin cannot be used to approximate our left-sided Chebyshev bound $\text{L} (\gamma)$ since the complement of $\mathcal{C}$ fails to be convex.

Vandenberghe et al.~\cite{Vandenberghe07} derive a semidefinite program that provides a tight upper bound on the probability that $\bm{\tilde{\xi}} \in \mathcal{C}$ for a (not necessarily convex) set $\mathcal{C} = \{ \bm{\xi} \in \mathbb{R}^T \, : \, \bm{\xi}^\intercal \bm{A}_i \bm{\xi} + 2 \bm{b}_i^\intercal \bm{\xi} + c_i < 0 \;\; \forall i = 1, \ldots, m \}$, assuming that the random vector $\bm{\tilde{\xi}}$ can be governed by any distribution from the ambiguity set $\mathcal{P}^0$. Employing a second-order Taylor approximation of $\prod_{t=1}^T \xi_t$ around $\mu \bm{1}$,
\begin{equation*}
\begin{aligned}
	\prod_{t=1}^T \xi_t
	&\approx \mu^{T-2} \left( \mu^2 + \mu(\bm{\xi} - \mu\bm{1})^\intercal \bm{1} + \frac{1}{2} (\bm{\xi} - \mu\bm{1})^\intercal (\bm{11}^\intercal - \mathbb{I})(\bm{\xi} - \mu\bm{1})\right) \\
	&= \mu^{T-2} \left( (1-T)\mu^2 + \mu \bm{\xi}^\intercal\bm{1} + \frac{1}{2} \bm{\xi}^\intercal (\bm{11}^\intercal - \mathbb{I})\bm{\xi} + \frac{1}{2}\mu^2 T(T-1) - (T-1)\mu\bm{\xi}^\intercal\bm{1} \right) \\
	&= \frac{1}{2}\mu^{T-2} \left( (T-1)(T-2)\mu^2 - 2(T-2)\mu\bm{\xi}^\intercal\bm{1} + \bm{\xi}^\intercal (\bm{11}^\intercal - \mathbb{I})\bm{\xi} \right),
\end{aligned}
\end{equation*}
we can derive an approximate right-sided Chebyshev bound $\text{R}^\text{VBC} (\gamma) = \sup_{\mathbb{P} \in \mathcal{P}^0} \mathbb{P} \left( \bm{\tilde{\xi}} \in \mathcal{C} \right)$ by replacing the product $\prod_{t=1}^T \xi_t$ with its Taylor approximation in the definition of the set $\mathcal{C}$:
\begin{equation*}
\mathcal{C} = \left\{ \bm{\xi} \in \mathbb{R}^T \, : \, \frac{1}{2}\mu^{T-2} \left( (T-1)(T-2)\mu^2 - 2(T-2)\mu\bm{\xi}^\intercal\bm{1} + \bm{\xi}^\intercal (\bm{11}^\intercal - \mathbb{I})\bm{\xi} \right) > \gamma \right\}
\end{equation*}
A similar approximation $\text{L}^\text{VBC} (\gamma)$ can be derived for our left-sided Chebyshev bound $\text{L} (\gamma)$ by considering the strict complement of $\mathcal{C}$. Note that $\text{R}^\text{VBC} (\gamma)$ and $\text{L}^\text{VBC} (\gamma)$ can over- or underestimate our bounds $\text{R} (\gamma)$ and $\text{L} (\gamma)$ due to the use of the Taylor approximation.

\begin{figure}[tb]
 	\centering
	\includegraphics[width=0.37\paperwidth]{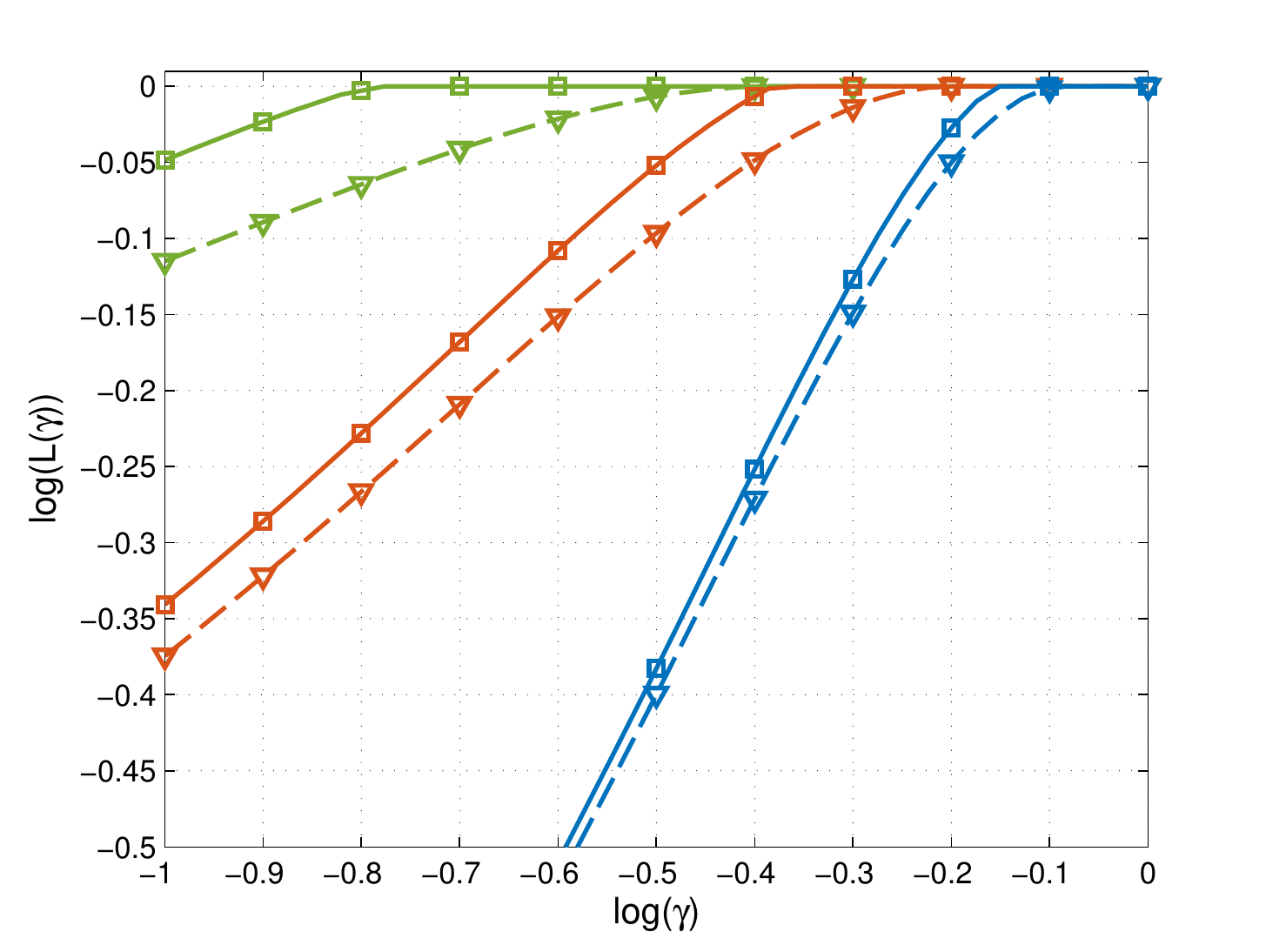}
    \includegraphics[width=0.37\paperwidth]{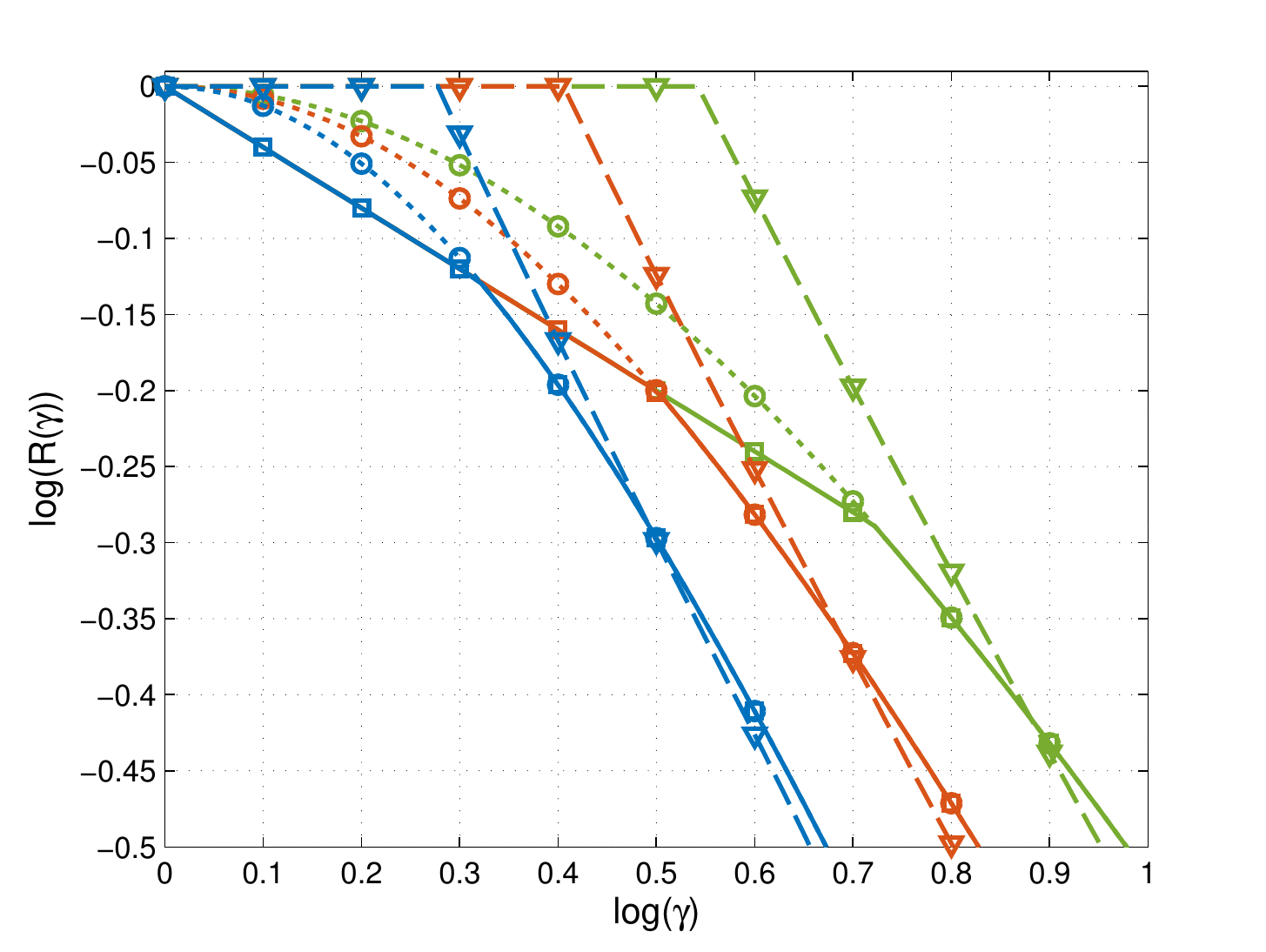} \\
    \includegraphics[width=0.37\paperwidth]{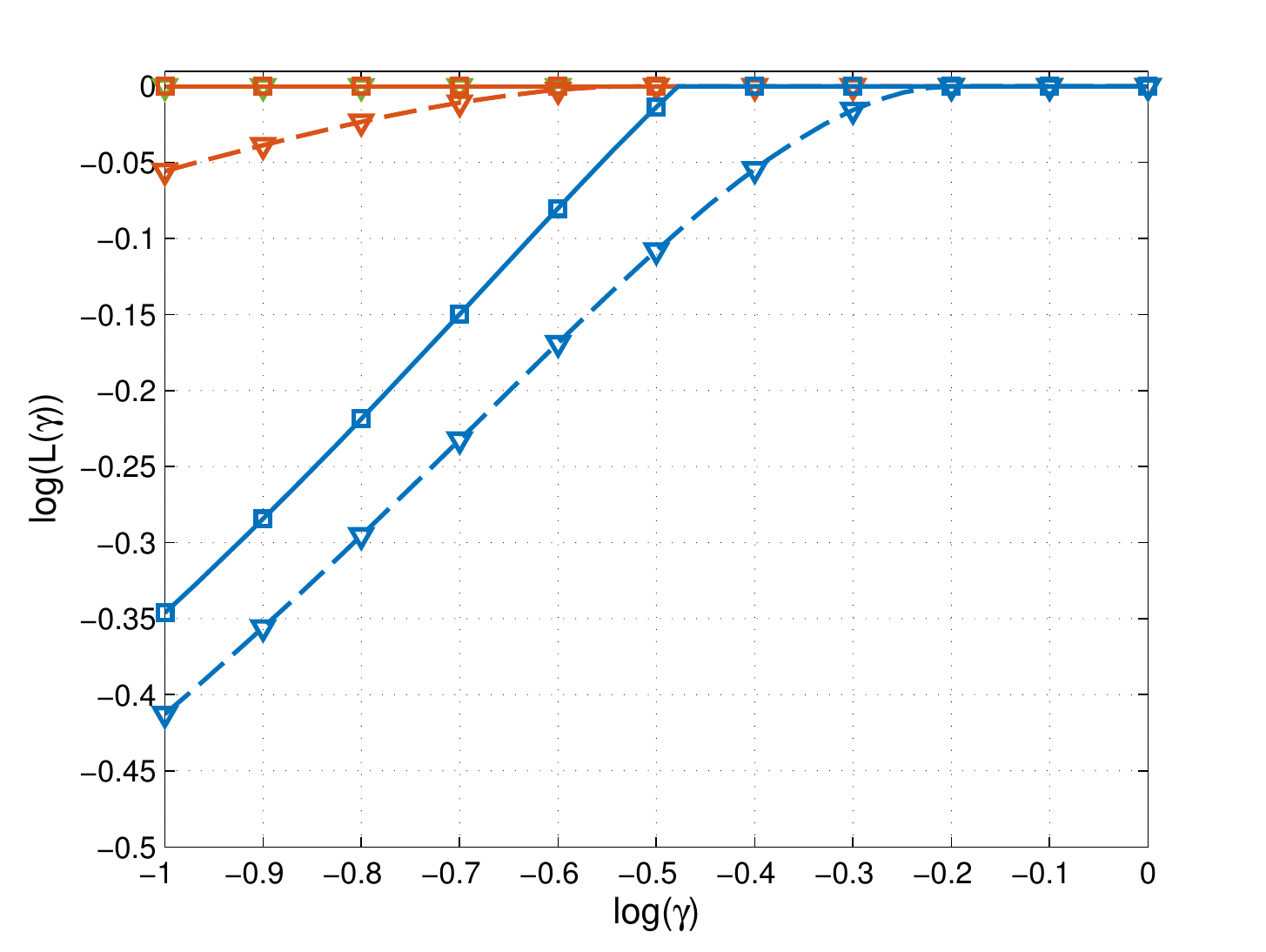}
    \includegraphics[width=0.37\paperwidth]{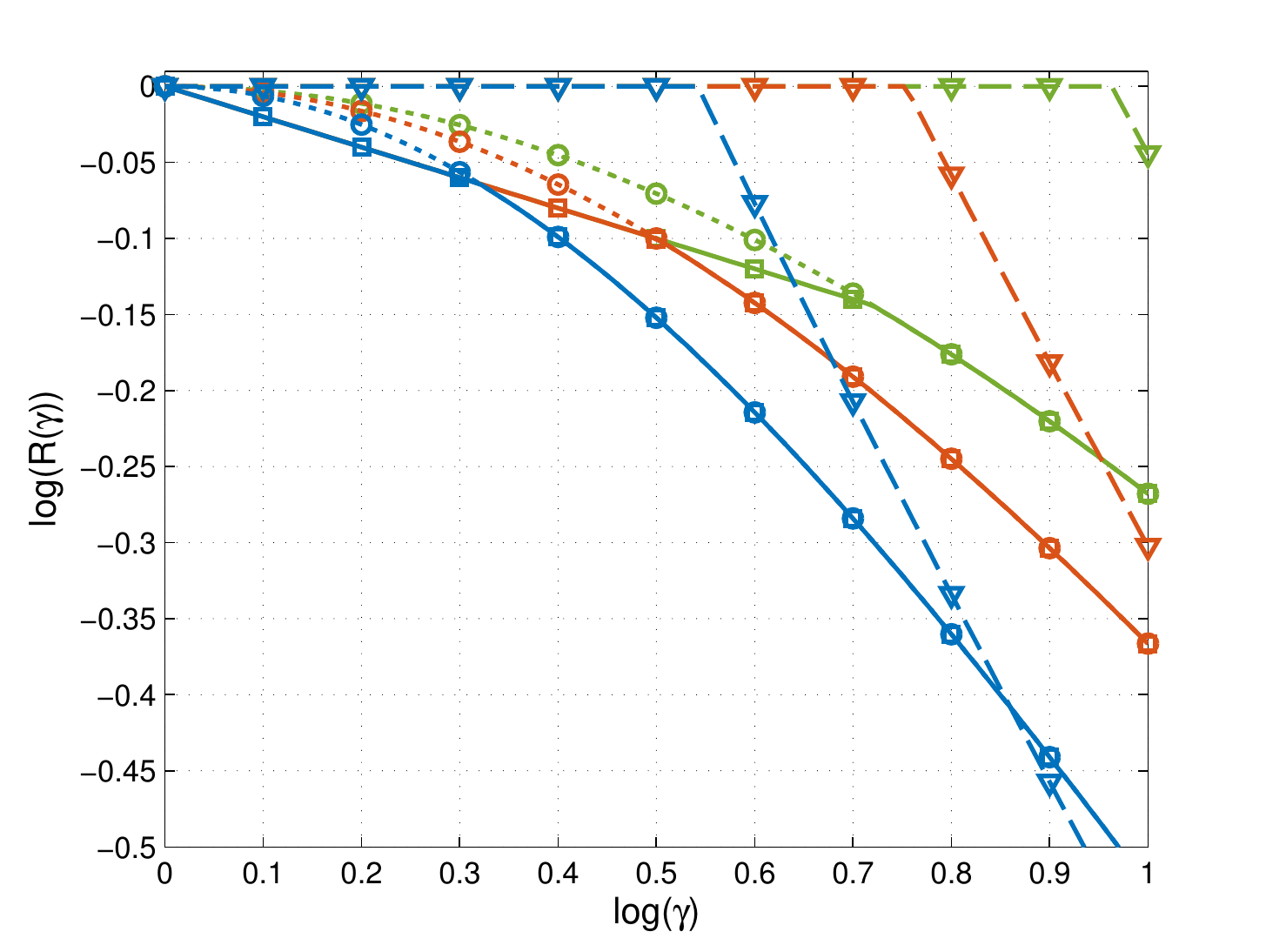}
	\caption{Comparison of the left-sided (left) and right-sided (right) Chebyshev bounds for the products of $T = 5$ (top) and $T = 10$ (bottom) random variables with $\mu = 1$ and $\rho = 0$. The solid lines with squares, the dashed lines with triangles and the dotted lines with circles represent our bounds, the VBC bounds and the MO bounds, respectively. From bottom to top, the blue, red and green lines correspond to $\sigma = 0.2$, $0.3$ and $0.4$ (left) and $\sigma = 0.4$, $0.5$ and $0.6$ (right), respectively.}
	\label{fig:compare_boyd_marshall}
\end{figure}

Figure~\ref{fig:compare_boyd_marshall} compares our Chebyshev bounds $\text{L} (\gamma)$ and $\text{R} (\gamma)$ with the approximate bounds $\text{L}^\text{VBC} (\gamma)$ and $\text{R}^\text{VBC} (\gamma)$ (`VBC bounds') as well as $\text{R}^\text{MO} (\gamma)$ (`MO bound'). As expected, the VBC bounds can over- and underestimate our bounds $\text{L} (\gamma)$ and $\text{R} (\gamma)$, whereas the MO bound consistently overestimates $\text{R} (\gamma)$. Moreover, the MO bound coincides with our right-sided Chebyshev bound for large values of $\gamma$. The quality of both approximations deteriorates with increasing $\sigma$ and decreasing $\gamma$. Interestingly, the VBC bound deterioates with increasing numbers of random variables, whereas the MO bound improves with increasing $T$. The figure shows that both approximate bounds can misestimate the bounds $\text{L} (\gamma)$ and $\text{R} (\gamma)$ substantially.

\begin{table}[tb]
\centering
\begin{tabular}{r|rrrrrrrrrr}
     \multicolumn{1}{c}{} & \multicolumn{10}{c}{Number of random variables $T$} \\
     \multicolumn{1}{c}{} & \multicolumn{1}{c}{4}    & \multicolumn{1}{c}{8}    & \multicolumn{1}{c}{12}   & \multicolumn{1}{c}{16}   & \multicolumn{1}{c}{20}   & \multicolumn{1}{c}{24}  & \multicolumn{1}{c}{28}   & \multicolumn{1}{c}{32}   & \multicolumn{1}{c}{36}    & \multicolumn{1}{c}{40}    \\ \hline
VBC bounds  & 1.02 & 1.01 & 1.06 & 1.07 & 1.11 & 1.23 & 1.29 & 1.48 & 1.72  & 2.02  \\
Our bounds  & 1.63 & 1.81 & 2.19 & 2.64 & 3.43 & 4.71 & 6.38 & 9.34 & 13.37 & 18.35 \\ \hline \hline
\end{tabular}
\caption{Runtimes (secs) required to calculate the Chebyshev bounds. Each runtime is averaged over 10 instances with randomly selected $\mu$, $\sigma$ and $\gamma$, and it includes the calculation of both the left-sided and the right-sided bounds. \label{tab:runtimes}}
\end{table}

The MO bound has an analytical solution and can therefore be computed in negligible time. In contrast, the VBC bounds and our bounds require the solution of semidefinite programs with two LMIs of size $\mathcal{O} (T^2)$. Table~\ref{tab:runtimes} compares the computation times of both bounds for products of different size $T$ on a computer with a 3.40GHz i7 CPU and 16GB RAM. While both bounds can be computed within seconds, the VBC bounds require significantly less runtime than our bounds. We attribute this to the LMI reformulations of the polynomial constraints in Theorems~\ref{thm:rprob} and~\ref{thm:lprob}, which seem to lack structure that can be exploited by SDPT3.

\begin{figure}[tb]
 	\centering
	\includegraphics[width=0.37\paperwidth]{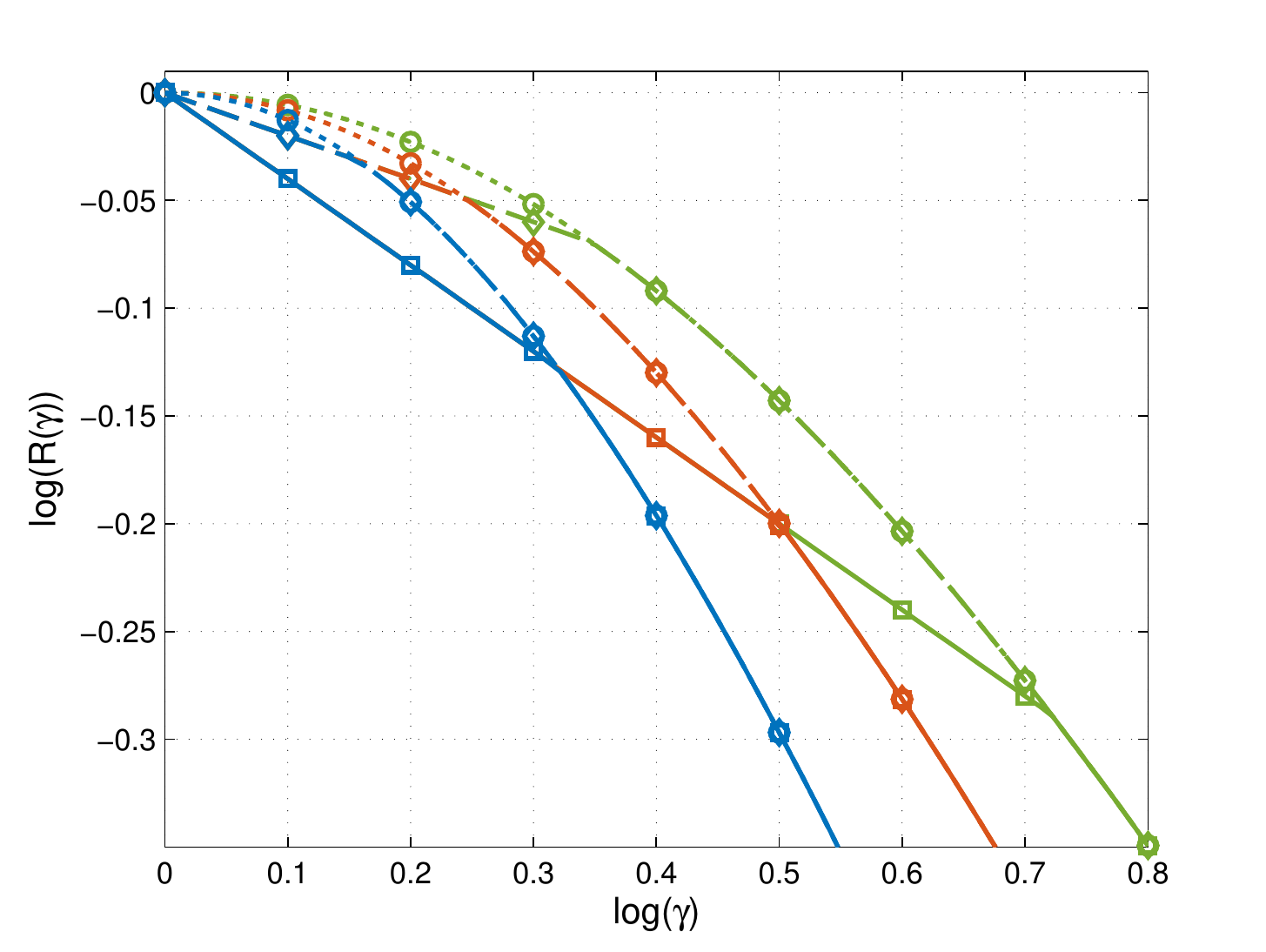}
    \includegraphics[width=0.37\paperwidth]{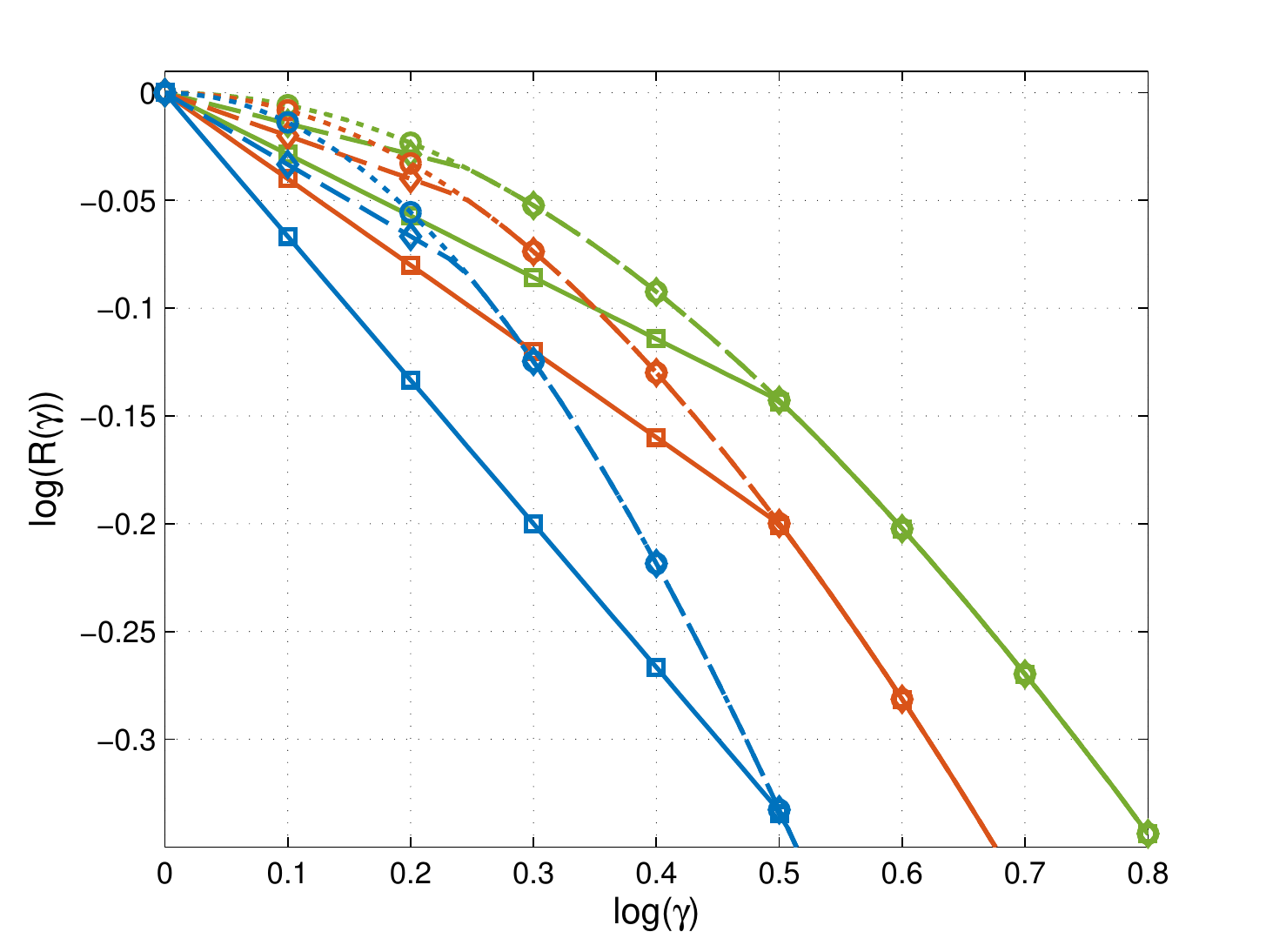} 
	\caption{Comparison of the right-sided Chebyshev bounds $\text{R} (\gamma)$ (solid lines with squares), $\text{R}' (\gamma)$ (dashed lines with diamonds) and $\text{R}^\text{MO} (\gamma)$ (dotted lines with circles) with $\mu = 1$ and $\rho = 0$. From bottom to top, the blue, red and green lines correspond to $\sigma = 0.4$, $0.5$ and $0.6$ in the left graph (with $T = 5$ fixed) and to $T = 3$, $5$ and $7$ in the right graph (with $\sigma = 0.5$ fixed), respectively. \label{fig:rprob}}
\end{figure}

Figure~\ref{fig:rprob} compares the right-sided Chebyshev bound $\text{R}(\gamma)$ with the relaxed right-sided bound $\text{R}'(\gamma)$ and the MO bound $\text{R}^\text{MO} (\gamma)$. The figure illustrates that $\text{R}^\text{MO} (\gamma)$ coincides with $\text{R}' (\gamma)$ for $\gamma \geq \big( \mu + \frac{\sigma^2 \theta}{T \mu} \big)^T$, and subsequently both bounds coincide with $\text{R} (\gamma)$ for large values of $\gamma$. The gaps between the bounds increase with larger variances $\sigma^2$, and they decrease with larger numbers of random variables $T$.

\subsection{Case Study: Financial Risk Management}

Consider an investor who allocates a limited budget to a fixed pool of $n$ assets over a time horizon of $T$ periods. We denote by $\tilde{r}_{t,i} \geq -1$, $t = 1, \ldots, T$ and $i = 1, \ldots, n$, the relative price change of asset $i$ between periods $t$ and $t + 1$. We assume that the investor pursues a fixed-mix (or constant proportions) strategy which rebalances the portfolio composition to a pre-selected set of weights $\bm{w} \in \mathcal{W} = \{ \bm{z} \in \mathbb{R}^n_+ \, : \, \mathbf{e}^\intercal \bm{z} = 1 \}$ at the beginning of each period. Note that despite being memoryless, fixed-mix strategies are dynamic since they recapitalize those assets whose returns were below average (`buy low') and divest assets whose returns were above average (`sell high'). Fixed-mix strategies generalize the well-known $1/N$-portfolio~\cite{DeMiguel09}, and they have received significant attention among both academics and practitioners.

We assume that the investor assesses the fixed-mix strategy $\bm{w}$ in view of the value-at-risk of the portfolio's terminal wealth, which is defined as
\begin{equation*}
\begin{aligned}
	\text{VaR}_{\epsilon}(\bm w)
	=~&\displaystyle \sup_{\gamma \in \mathbb{R}} \left\{ \gamma :
	 \mathbb{P} 
		\left( \prod_{t=1}^{T} (1 + \bm{w^{\intercal}}\bm{\tilde{r}}_t) > \gamma \right)
		\geq 1 - \epsilon \right\}.
\end{aligned}
\end{equation*}
Here, the asset returns $\bm{\tilde{r}}_t = (\tilde{r}_{t,i})_{i=1}^n$ are governed by the probability distribution $\mathbb{P}$, and $\epsilon$ is a pre-specified parameter that reflects the investor's risk tolerance.

Calculating the value-at-risk of a portfolio's terminal wealth requires perfect knowledge of the joint asset return distribution $\mathbb{P}$, which is unavailable in practice. Following~\cite{Rujeerapaiboon16}, we will assume that it is only known that the asset returns $\left( \tilde{\bm{r}}_t \right)_{t=1}^T$ follow a weak-sense white noise process with mean $\bm{\mu}$ and variance $\bm{\Sigma}$, that is, the asset returns are serially uncorrelated and have period-wise identical first and second-order moments. In that case, the wealth evolution $(\tilde{\xi}_t)_{t=1}^T = \left( 1 + \bm{w}^\intercal\tilde{\bm{r}}_t \right)_{t=1}^T$ also follows a weak-sense stochastic process governed by a distribution $\mathbb{P}_{\bm w}$ supported on $\mathbb R_+^T$, under which the $\tilde \xi_t$ have mean $\bm{w}^\intercal \bm{\mu}$ and variance $\bm{w}^\intercal\bm{\Sigma w}$ and are serially uncorrelated. We denote the set of all these distributions by $\mathcal{P}_{\bm w}$. In this setting, an ambiguity-averse investor may assess the fixed-mix strategy $\bm{w}$ in view of the \emph{worst-case} value-at-risk of the portfolio's terminal wealth over all distributions $\mathbb{P}_{\bm w}\in \mathcal{P}_{\bm w}$:
\begin{equation*}
\begin{aligned}
	\text{WVaR}_{\epsilon}(\bm w)
	=~&\displaystyle \sup_{\gamma \in \mathbb{R}} \left\{ \gamma :
	 \inf_{\mathbb{P}_{\bm w} \in \mathcal{P}_{\bm w}} \mathbb{P}_{\bm w} 
		\left( \prod_{t=1}^{T} \tilde \xi_t > \gamma \right)
		\geq 1 - \epsilon \right\}.
\end{aligned}
\end{equation*}
In \cite{Rujeerapaiboon16}, the worst-case value-at-risk of the portfolio's terminal wealth is replaced with a quadratic approximation. The Chebyshev bounds proposed in this paper allow us to calculate the worst-case value-at-risk exactly without resorting to any approximation. Indeed, one verifies that
\begin{equation*}
\text{WVaR}_{\epsilon}(\bm w)
		\;\; = \;\;
\sup_{\gamma \in \mathbb{R}} \left\{ \gamma :
		\sup_{\mathbb{P}_{\bm w} \in \mathcal{P}_{\bm w}} \mathbb{P}_{\bm w} \left( \prod_{t=1}^{T} \tilde{\xi}_t \leq \gamma \right)
		\leq \epsilon \right\}
		\;\; = \;\; \sup_{\gamma \in \mathbb{R}} \left\{ \gamma :
		L (\gamma; \bm{w}^\intercal\bm{\mu}, \bm{w}^\intercal\bm{\Sigma w}) \leq \epsilon \right\},
\end{equation*}
where we have made explicit the dependence of the left-sided Chebyshev bound $L$ on the mean $\bm{w}^\intercal\bm{\mu}$ and the variance $\bm{w}^\intercal\bm{\Sigma w}$ of the wealth evolution $(\tilde{\xi}_t)_{t=1}^T$. Since $L$ is monotonically non-decreasing in $\gamma$, the last expression can be evaluated efficiently through bisection on $\gamma$.

\begin{figure}[tb]
 	\centering
	\includegraphics[width=0.37\paperwidth]{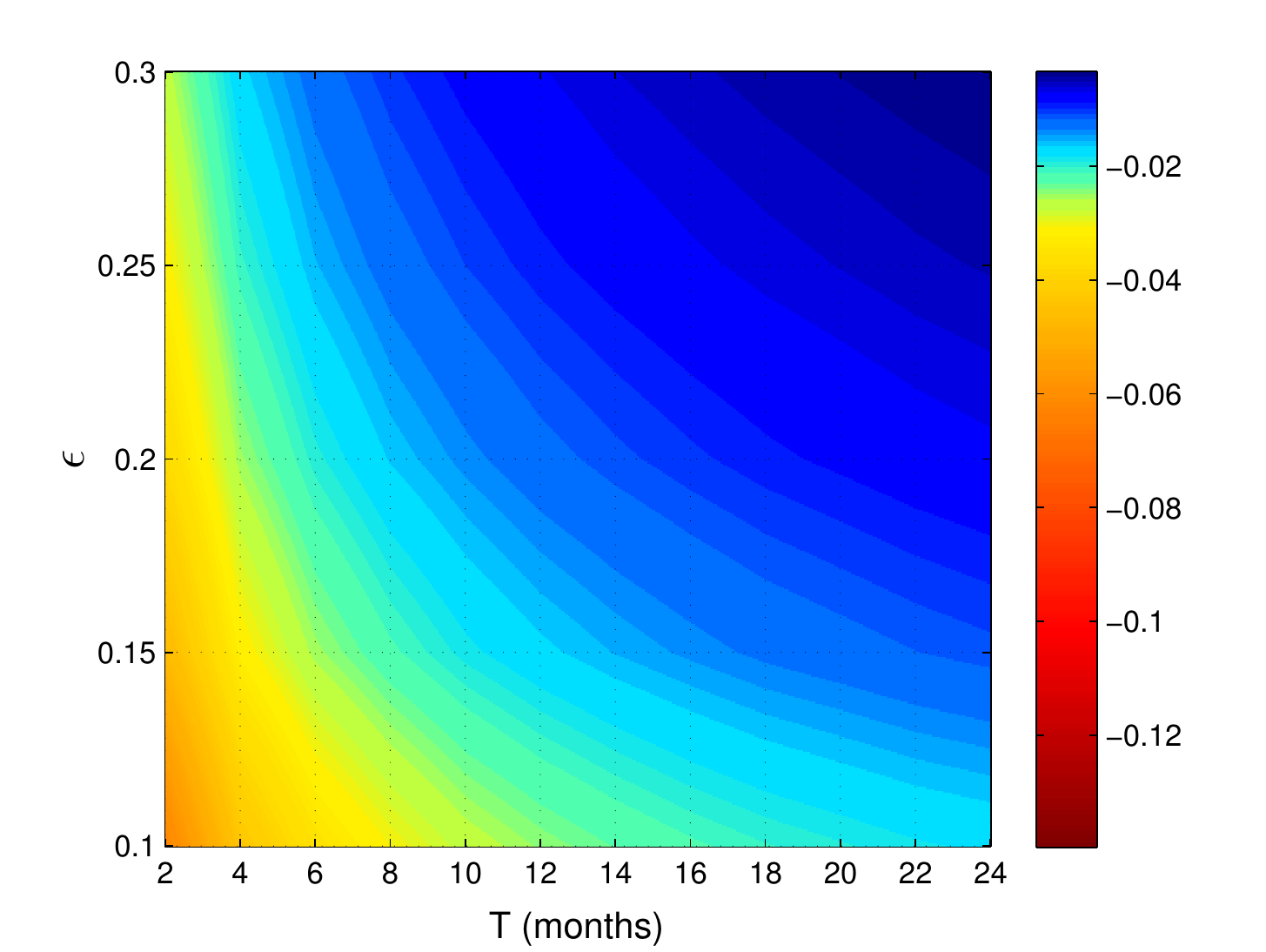}
    \includegraphics[width=0.37\paperwidth]{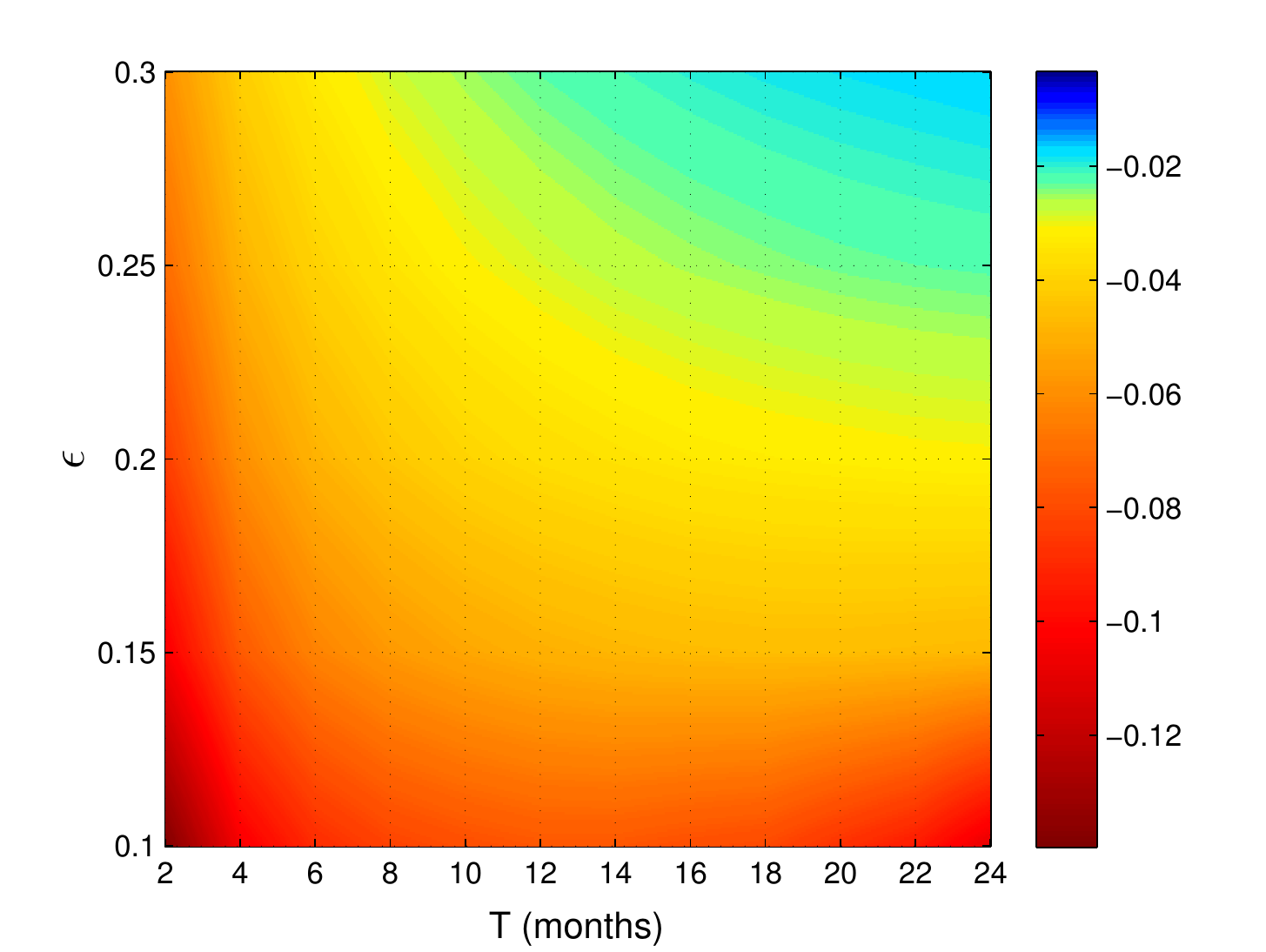}
	\caption{Wort-case value-at-risk of the growth rates of the minimum-variance (left) and maximum-expectation (right) portfolios for different investment horizons $T$ and risk tolerances $\epsilon$.} 
	\label{fig:portfolio}
\end{figure}

Figure~\ref{fig:portfolio} reports the worst-case value-at-risk of two portfolios over different time horizons $T$, where $\bm{\mu}$ and $\bm{\Sigma}$ are calibrated to the 2003--2012 period of Fama and French's 10 Industry Portfolios data set.\footnote{See \texttt{http://mba.tuck.dartmouth.edu/pages/faculty/ken.french/data\underline{~}library.html}.} The minimum-variance portfolio (left graph) corresponds to the weight vector $\bm{w} \in \mathcal{W}$ that minimizes $\bm{w}^\intercal \bm{\Sigma} \bm{w}$, whereas the maximum-expectation portfolio (right graph) invests all wealth into the asset $i$ with the highest expected return $\mu_i$. To facilitate a fair comparison among different time horizons, the graphs report the growth rates of the portfolios, that is, the logarithms of the terminal wealth, divided by the number of investment periods $T$. As expected, the minimum-variance portfolio is less risky than the maximum-expectation portfolio, and the risk of both portfolios tends to decrease when the investment horizon $T$ grows. Interestingly, however, the risk of the maximum-expectation portfolio \emph{increases} with large $T$ for low risk tolerances $\epsilon \lesssim 0.15$. This seemingly counter-intuitive effect is explained by Theorem~\ref{thm:rprob_eq0}, which states that the wealth evolution $\prod_{t=1}^T \tilde{\xi}_t$ is absorbed at $0$ for large investment horizons $T$.

In addition to \emph{evaluating} the worst-case value-at-risk of a pre-selected portfolio $\bm{w}$, an investor often seeks to determine a portfolio $\bm{w}^\star$ that \emph{optimizes} the worst-case value-at-risk. The search for optimal portfolios is greatly simplified by the observation that there is always a portfolio $\bm{w}^\star$ on the mean-variance efficient frontier that maximizes $\text{WVaR}_\epsilon (\bm{w})$ over (subsets of) $\mathcal{W}$. Indeed, Theorem~\ref{thm:rprob_bnd} implies that $\text{L}(\gamma; \bm{w}^\intercal\bm{\mu}, \bm{w}^\intercal\bm{\Sigma w}) = \text{L}'(\gamma; \bm{w}^\intercal\bm{\mu}, \bm{w}^\intercal\bm{\Sigma w})$, and one readily verifies that $\text{L}'(\gamma; \bm{w}^\intercal\bm{\mu}, \bm{w}^\intercal\bm{\Sigma w})$ is non-decreasing in both $\gamma$ and $\bm{w}^\intercal\bm{\Sigma w}$. This implies that
\begin{equation*}
\sup_{\gamma \in \mathbb{R}} \left\{ \gamma : \text{L}'(\gamma; \bm{w}^\intercal\bm{\mu}, \bm{w}^\intercal\bm{\Sigma w}) \leq \epsilon \right\}
\; \leq \;
\sup_{\gamma \in \mathbb{R}} \left\{ \gamma : \text{L}'(\gamma; \bm{w}'{}^\intercal\bm{\mu}, \bm{w}'{}^\intercal\bm{\Sigma w}') \leq \epsilon \right\}
\end{equation*}
for two portfolios $\bm{w}$ and $\bm{w}'$ that satisfy $\bm{w}^\intercal\bm{\mu} = \bm{w}'{}^\intercal\bm{\mu}$ and $\bm{w}^\intercal\bm{\Sigma w} \geq \bm{w}'{}^\intercal\bm{\Sigma w}'$. We thus conclude that among all portfolios $\bm{w} \in \mathcal{W}$ that achieve the same mean return $\bm{w}^\intercal \bm{\mu}$, the portfolio with smallest variance $\bm{w}^\intercal\bm{\Sigma w}$ provides the best worst-case value-at-risk. We can therefore identify an optimal portfolio through a one-dimensional line search over the mean-variance efficient frontier.

\paragraph{Acknowledgements} This research was supported by the Swiss National Science Foundation grant BSCGI0\_157733 and the EPSRC grants EP/M028240/1 and EP/M027856/1.

\bibliography{bibliography}
\bibliographystyle{acm}

\end{document}